\documentclass[a4paper,11pt,reqno]{amsart} 
\usepackage{amsmath}
\usepackage{amsfonts}
\usepackage[T1]{fontenc}   
\usepackage[utf8]{inputenc}
\usepackage[mathscr]{euscript}
\usepackage{verbatim}
\usepackage{amssymb,latexsym}
\usepackage{amsthm}
\usepackage{color}
\usepackage{graphicx}
\usepackage{microtype}
\usepackage[font=scriptsize]{caption}
\usepackage{bbm}
\usepackage{cite}
\usepackage{amsmath,accents}

\newcommand{\Ab}{\mathbf{A}}
\usepackage[lmargin=3.5 cm,rmargin=3.5 cm,tmargin=3.5cm,bmargin=2.5cm,paper=a4paper]{geometry}
\DeclareMathOperator{\curl}{curl}\DeclareMathOperator{\Div}{div}
 \DeclareMathOperator{\dist}{dist} 
\DeclareMathOperator{\supp}{supp} \DeclareMathOperator{\dom}{\mathbf {Dom}}
\newtheorem{thm}{Theorem}[section]

\newtheorem{lem}[thm]{Lemma}

\newtheorem{theorem}[thm]{Theorem}
\newtheorem{assumption}[thm]{Assumption}

\newtheorem{lemma}[thm]{Lemma}
\newtheorem{proposition}[thm]{Proposition}
\newtheorem{corollary}[thm]{Corollary}

\newtheorem{conj}[thm]{Conjecture}
\theoremstyle{remark}
\newtheorem{rem}[thm]{Remark}

\newcommand{\nb}{\nabla}
\newcommand{\R}{\mathbb{R}}
\newcommand{\Fb}{\mathbf{F}}

\newcommand{\N}{\mathbb{N}}
\newcommand{\C}{\mathbb{C}}

\newcommand{\Hd}{H_{{\rm div}}^1(\Omega)}
\newcommand{\Om}{\Omega}
\newcommand{\kp}{\kappa}
\newcommand{\kn}{\nabla-i\kappa H{\bf A}}
\newcommand{\Es}{{\rm E}_{\rm g.st}(\kappa,  H)}
\newcommand{\GL}{\mathcal E_{\kappa,H}}

\def\sig#1{\vbox{\hsize=5.5cm
		\kern2cm\hrule\kern1ex
		\hbox to \hsize{\strut\hfil #1 \hfil}}}
\newcommand\signatures[4]{%
	\vspace{3cm}
	\hbox to \hsize{\hfil #1, \today\hfil}
	\vspace{3cm}
	\hbox to \hsize{\quad#2\hfil\hfil #3\quad}
	\vspace{3cm}
	\hbox to \hsize{\hfil#4\hfil}}
\numberwithin{equation}{section}

\title[Distribution of superconductivity]{The distribution of  superconductivity\\ near a magnetic barrier}
\author[W. Assaad]{Wafaa Assaad}
\address{Lund University, Department of Mathematics, Lund, Sweden}
\email{wafaa.assaad@math.lth.se}
\author[A. Kachmar]{Ayman Kachmar}
\address{Lebanese University, Department of Mathematics, Hadat, Lebanon}
\email{ayman.kashmar@gmail.com}
\author[M.P. Sundqvist]{Mikael Persson-Sundqvist}
\address{Lund University, Department of Mathematics, Lund, Sweden}
\email{mikael.persson\_sundqvist@math.lth.se}
\date{\today}
\begin{document}
\maketitle
\begin{abstract}
We consider the Ginzburg--Landau functional, defined on a two-dimensional simply connected domain with smooth boundary, in the situation when the applied magnetic field is piecewise constant with a jump discontinuity along a smooth curve. In the regime of large Ginzburg--Landau parameter and strong magnetic field, we study the concentration of minimizing configurations along this discontinuity.
\end{abstract}

\section{Introduction}\label{sec:int}

\subsection{Motivation}
The Ginzburg--Landau theory, introduced in~\cite{landau1950theory}, is a  phenomenological macroscopic model describing  the response of a superconducting sample to an external magnetic field.   
The phenomenological quantities associated with a superconductor are the order parameter $\psi$ and the magnetic  potential $\Ab$, where $|\psi|^2$ measures the density of the superconducting Cooper pairs and $\curl \Ab$ represents the induced magnetic field in the sample.

In this paper, the superconducting sample is an infinite cylindrical domain subjected to a magnetic field with direction parallel to the axis of the cylinder. For this specific geometry, it is enough to consider the horizontal cross section of the sample,  $\Omega\subset\R^2$. The phenomenological configuration $(\psi,\Ab)$ is then defined on the domain $\Omega$.

The study of the Ginzburg--Landau model  in the case of a \emph{uniform} or a \emph{smooth} non-uniform applied magnetic field has been the focus of much attention in literature. We refer to the two monographs~\cite{fournais2010spectral,sandier2007vortices} for the  uniform magnetic field case. Smooth magnetic fields are the subject of the papers~\cite{attar2015energy, attar2015ground,  Helffer, lu1999estimates, pan2002schrodinger}. In this paper, we focus on  the case where the applied magnetic field is a step function, which is not covered in the aforementioned  papers. Such magnetic fields are interesting because they give rise to edge currents on the interface separating the distinct values of the magnetic field---\emph{the magnetic barrier} (see~\cite{dombrowski2014edge, hislop2016band, reijniers2000snake}). Our configuration is illustrated in   Figure~\ref{fig2}.

In an earlier contribution~\cite{Assaad}, we explored the influence of a step magnetic field on the distribution of bulk superconductivity, which highlighted the regime where an edge current might occur near the magnetic barrier. In this contribution, we will demonstrate the existence of such a current by providing examples where the superconductivity concentrates at the interface separating the distinct values of the magnetic field.

\subsection{The functional and the mathematical set-up}
We assume that the domain  $\Om$ is open in $\R^2$, bounded and simply connected. The Ginzburg--Landau (GL) free energy is given by the functional
\begin{equation}\label{eq:GL}
\GL (\psi,\Ab)= \int_\Om \Big( \big|(\nb-i\kp H {\bf
A})\psi\big|^2-\kp^2|\psi|^2+\frac{\kappa^2}{2}|\psi|^4 \Big)\,dx
 +\kp^2H^2\int_\Om\big|\curl{\bf A}-B_0\big|^2\,dx\,,
\end{equation}
with $\psi \in H^1(\Om;\C)$ and  ${\bf A}=(A_1,A_2) \in H^1(\Om;\R^2)$. Here, $\kp>0$ is a large GL parameter, the function $B_0:\Om \rightarrow [-1,1]$ is the profile of the applied magnetic field and  $H > 0$ is  the intensity of this applied magnetic field.

The parameter $\kappa$ depends on the temperature and the type of the material. It is a characteristic scale of the sample that measures the size of vortex cores (which is proportional to $\kappa^{-1}$). Vortex cores are narrow regions in the sample, which corresponds to $\kappa$ being a large parameter. That is the reason behind our analysis of the asymptotic regime $\kappa\to+\infty$, following  many early papers addressing this asymptotic regime (see e.g.~\cite{sandier2007vortices}).  
We work under the following assumptions on the domain $\Om$ and the magnetic field $B_0$~(illustrated in Figure~\ref{fig1}):
\begin{assumption}\label{assump}~
\begin{enumerate}
\item $\Omega_1\subset\Omega$ and $\Omega_2\subset\Omega$ are two disjoint open sets\,.
\item $\Omega_1$ and $\Omega_2$ have a finite number of connected components\,.
\item $\partial\Omega_1$ and $\partial\Omega_2$ are piecewise smooth with (possibly) a finite number of corners\,.
\item $\Gamma=\partial\Omega_1\cap\partial\Omega_2$ is the union of a finite number of smooth curves\,. we will refer to  $\Gamma$ as the {\bf magnetic barrier}\,.
\item $\Omega=\Omega_1\cup\Omega_2\cup\Gamma^\circ$ and  $\partial\Omega$  is smooth\,. 
\item $\Gamma\cap\partial\Omega$ is either empty or finite\,.
\item If $\Gamma\cap\partial\Omega\neq \emptyset$, then $\Gamma$ intersects $\partial\Omega$ transversely, i.e. $\nu_{\partial \Omega} \times \nu_\Gamma \neq 0$, where $\nu_{\partial \Omega}$ and  $\nu_\Gamma$ are respectively the unit normal vectors of $\partial \Omega$ and $\Gamma$\,.
\item $a \in [-1,1)\setminus\{0\}$ is a given constant\,.
\item $B_0={\mathbbm 1}_{\Omega_1}+a{\mathbbm 1}_{\Omega_2}$\,.
\end{enumerate} 
\end{assumption}
\begin{figure}
	\centering
	\includegraphics[scale=0.4]
	{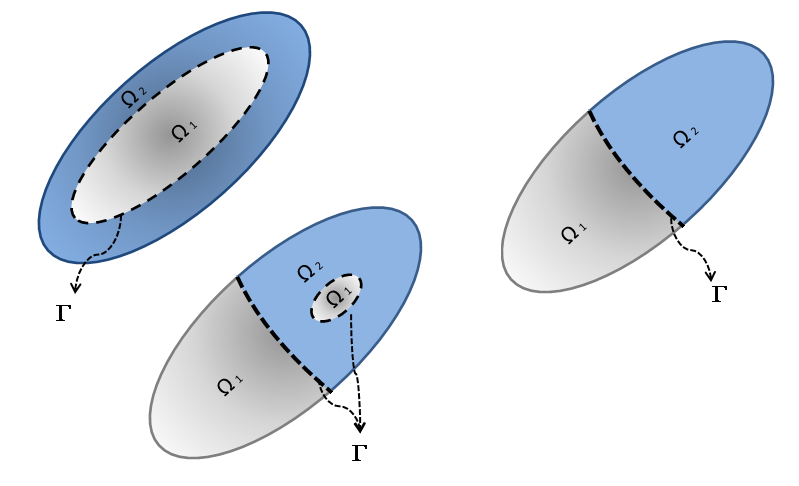} 
	\caption{Schematic representations of the set $\Om$.} 
	\label{fig1}
\end{figure}
\begin{figure}
	\centering
	\includegraphics[scale=0.4]{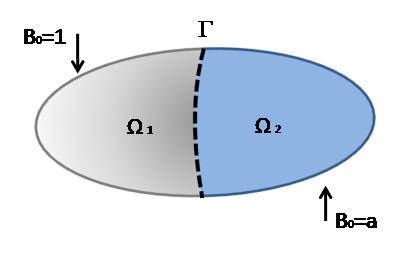} 
	\caption{Schematic representation of the set $\Om$ subjected to a step magnetic field $B_0$, with the magnetic barrier $\Gamma$. }
	\label{fig2}
\end{figure}
 The ground state of the superconductor  describes its behaviour at equilibrium. It is obtained by minimizing the GL functional in~\eqref{eq:GL} with respect to $(\psi,\Ab)$. The corresponding energy is called the ground state energy, denoted by $\Es$, where
\begin{equation*} 
\Es=\inf\{\GL(\psi,\Ab)~:~(\psi,\Ab) \in H^1(\Om;\C)\times H^1(\Om;\R^2)\}\,.
\end{equation*}
The functional in~\eqref{eq:GL} enjoys the property of  \emph{gauge invariance}. It does  not change under the  transformation $(\psi,\Ab)\mapsto (e^{i\varphi\kappa H}\psi,\Ab+\nb \varphi)$, for any (say smooth) function $\varphi:\R^2 \rightarrow \R$. It  follows that the only physically meaningful quantities are the gauge invariant ones, such as $|\psi|$ and $\curl \Ab$. The gauge invariance  permits us to restrict the GL function to the space 
$H^1(\Om;\C)\times\Hd$ where
\begin{equation}\label{eq:Hd}
\Hd=
\left\{
 \Ab \in H^1(\Om;\R^2)~:~ \Div\Ab=0 \ \mathrm{in}\ \Om,\ \Ab\cdot\nu_{\partial\Omega}=0\ \mathrm{on}\ \partial \Om
\right\}\,.
\end{equation}
More precisely, the ground state energy can be written as follows (see~\cite[Appendix~D]{fournais2010spectral})
\begin{equation} \label{eq:gr_st}
\Es=\inf\{\GL(\psi,\Ab)~:~(\psi,\Ab) \in H^1(\Om;\C)\times\Hd\}\,.
\end{equation}
This restriction allows us to make profit from some well-known  regularity properties of vector fields in $\Hd$ (see~\cite[Appendix B]{Assaad}).

Critical points $(\psi, \Ab) \in H^1(\Om;\C)\times\Hd $ of $\GL$ are weak solutions of the following GL equations:
\begin{equation}\label{eq:Euler}
\begin{cases}
\big(\kn\big)^2\psi=\kp^2(|\psi|^2-1)\psi &\mathrm{in}\ \Om\,,\\
-\nb^{\perp}  \big(\curl\Ab-B_0\big)= \frac{1}{\kp H}\mathrm{Im}\big(\overline{\psi}(\nb-i\kp H \Ab)\psi\big) & \mathrm{in}\ \Om\,,\\
\nu\cdot(\kn)\psi=0 & \mathrm{on}\ \partial \Om \,,\\
\curl\Ab=B_0 & \mathrm{on}~ \partial \Om \,.
\end{cases}
\end{equation}
Here,
\[(\nb-i\kp H\Ab)^2\psi= \Delta \psi -i\kp H(\Div\Ab)\psi -2i\kp H\Ab\cdot\nb \psi-\kp^2 H^2 |\Ab|^2\psi\]
and $\nb^\perp=(\partial_{x_2},-\partial_{x_1})\ \text{is the Hodge gradient}\,.$

\subsection{Some earlier results for uniform magnetic fields} The value of the ground state energy $\Es$ depends on $\kappa$ and $H$ in a non-trivial fashion. The physical explanation of this is that a superconductor undergoes phase transitions as the intensity of the applied magnetic field varies.

To illustrate the dependence on the intensity of the applied magnetic field, we assume that  $H=b\kappa$, for some fixed parameter $b>0$. Such magnetic field strengths are considered in many papers,~\cite{almog2007distribution, lu1999estimates, pan2002surface, sandier2003decrease}.

Assuming that the applied magnetic field is  uniform, which corresponds to taking $B_0\equiv1$ in~\eqref{eq:GL},  the following scenario takes place. If $b>\Theta_0^{-1}$, where $\Theta_0\approx 0.59$ is a universal constant  defined in~\eqref{eq:theta1} below, the only minimizer of the GL functional (up to change of gauge) is the trivial state $(0,\mathbf{\hat{F}})$ where $\curl\mathbf{\hat{F}}=1$ (see~\cite{giorgi2002breakdown}). This corresponds in Physics to the destruction of superconductivity when the sample is submitted to a large external magnetic field, and occurs when the intensity $H$ crosses a specific threshold value, the  so-called  \emph{third critical field},  denoted by $H_{C_3}$.

Another well-known critical field to be considered is the \emph{second critical field} $H_{C_2}$, which is much harder to define. When $H<H_{C_2}$, then the superconductivity is uniformly distributed in the interior of the  sample (see~\cite{sandier2003decrease}). This is the bulk superconductivity regime.   When $H_{C_2}<H<H_{C_3}$, then the surface superconductivity regime occurs: the superconductivity disappears from the interior and is localised in a thin layer near the boundary of the sample (see~\cite{almog2007distribution, helffer2011superconductivity, pan2002surface, Correggi }). The transition from surface to bulk superconductivity takes place when $H$ varies around  the critical value $\kappa$, which we informally take as the definition of $H_{C_2}$ (see~\cite{fournais2011nucleation}). 
One more critical field left is $H_{C_1}$. It marks the transition from the pure superconducting phase to the phase with vortices. We refer to~\cite{sandier2007vortices} for its definition.

\subsection{Expected behaviour of magnetic steps}

Let us return back to the  case  where the magnetic field is a step function as in Assumption~\ref{assump}. At some stage, the expected behaviour of the superconductor in question deviates from the one submitted to a uniform magnetic field.  Recently, this case was considered in~\cite{Assaad} and the following was obtained. Suppose that $H=b\kappa$ and $\kappa$ is large. If $b<1/|a|$ then the bulk superconductivity persists, if $b>1/|a|$ then superconductivity decays exponentially in the bulk of $\Om_1$ and $\Om_2$, and \emph{may} nucleate in  thin layers near $\Gamma \cup \partial \Om $ (see Assumption~\ref{assump} and  Figure~\ref{fig2}).
The present contribution  affirms the presence of superconductivity in the vicinity of $\Gamma$  when $b$ is greater than, but close to the value $1/|a|$,  for some negative values of $a$. The precise statements are given in  Theorems~\ref{thm:Eg} and~\ref{thm:E_loc} below.

The aforementioned  behaviour of the superconductor in presence of magnetic steps  is consistent with the existing literature (for instance see~\cite{dombrowski2014edge,hislop2016band,hislop2015edge,iwatsuka1985examples,reijniers2000snake}). Particularly, the case where $a \in [-1,0)$ is called the \emph{trapping magnetic step} (see~\cite{hislop2016band}), where the discontinuous magnetic field may create supercurrents (\emph{snake orbits}) flowing along the \emph{magnetic barrier} ($\Gamma$ in our context). On the other hand, no such snake orbits are formed in the case where $a \in (0,1)$, which is called the \emph{non-trapping magnetic step}. However, the approach was generally spectral where some  properties of relevant linear models were analysed ~\cite{hislop2016band,hislop2015edge,iwatsuka1985examples,reijniers2000snake}, and no estimates for the non-linear  GL energy in~\eqref{eq:GL} were established.

This contribution  together with~\cite{Assaad} provide such estimates. Particularly in the case where $a\in[-1,0)$ and $b>1/|a|$, Theorems~\ref{thm:Eg} and~\ref{thm:E_loc} below establish  global and local asymptotic estimates for the ground state energy $\Es$ and the $L^4$-norm of the minimizing  order parameter $\psi$. These theorems assert the nucleation of superconductivity near the magnetic barrier $\Gamma$ (and the surface $\partial \Om$) when $b$ exceeds  the threshold value $1/|a|$. 

\subsection{Main results}

Our results are valid under the following additional assumption.
\begin{assumption}\label{A_2}
The parameter $H$ depends on $\kappa$ in the following manner
\begin{equation}\label{eq:A_2}
	H=b\kappa\,,
	\end{equation}
where $b$ is a \textbf{fixed} parameter satisfying 
	\[b >\frac 1{|a|},\quad a \in [-1,1)\setminus\{0\}\,.\]
\end{assumption}

\begin{rem}\label{rem:a_pos}
	Even though the case $a\in(0,1)$ is included in Assumption~\ref{A_2}, it will not be central  in our study (the reader may notice this in the majority of our theorems statements). The reason is that, our main concern is to analyse the interesting phenomenon happening when the bulk superconductivity is only restricted to a narrow neighbourhood of the magnetic barrier $\Gamma$, and  this only occurs when the values of the two  magnetic fields \emph{interacting} near $\Gamma$ are of \emph{opposite signs}, that is when $a \in [-1,0)$, (see Figure~\ref{fig2}). This can be seen through the trivial cases in Section~3.2, and is consistent with the aforementioned literature findings (non-trapping magnetic steps).
	Moreover, the case $b\leq 1/|a|$ is treated previously in~\cite{Assaad} and corresponds to the bulk regime.
\end{rem}
The statements of the main theorems involve two non-decreasing continuous functions~:
\[\mathfrak e_a:\big[|a|^{-1},+\infty) \rightarrow (-\infty,0] \quad\mathrm{and}\quad  E_\mathrm{surf}: [1,+\infty) \rightarrow (-\infty,0]\,,\]
respectively defined in~\eqref{eq:eba1} and~\eqref{eq:Esurf}. The energy  $E_\mathrm{surf}$ has been studied in  many papers~\cite{almog2007distribution, Correggi ,  pan2002surface, helffer2011superconductivity, fournais2011nucleation}. We will refer to $E_\mathrm{surf}$ as the \emph{surface energy}. The function $\mathfrak e_a$ is constructed in this paper, and we will refer to it as the \emph{barrier energy}. 

\begin{rem}\label{rem:a}
It is worthy to mention that $\mathfrak e_a(b)$ vanishes if and only if  
\begin{itemize}
	\item $a\in(0,1)$\,; or
	\item $a\in[-1,0)$ and $b\geq 1/\beta_a$, where $\beta_a$ is defined in~\eqref{eq:lamda}  below and satisfies $\beta_a\in(0,|a|\,)$ (see Theorem~\ref{thm:up_eigen}).
\end{itemize}
The surface energy $E_\mathrm{surf}(b)$ vanishes if and only if $b\geq \Theta_0^{-1}$, where $\Theta_0$ is the constant in~\eqref{eq:theta1}.
\end{rem}
\begin{theorem}[Global asymptotics]\label{thm:Eg}
For all  $a \in[-1,1)\setminus\{0\}$ and $b> 1/|a|$,  the ground state energy $\Es$ in~\eqref{eq:gr_st} satisfies, when $H=b\kappa$,
		\begin{equation}\label{eq:Eg}
		\Es =E^\mathrm{L}_a(b)\kappa +o(\kappa) \qquad(\kappa\to+\infty)
		\end{equation}
	where
	\[E^\mathrm{L}_a(b)=b^{-\frac 12}\Big(|\Gamma|  \mathfrak e_a(b)+|\partial \Omega_1 \cap \partial \Omega|  E_\mathrm{surf}(b)+|\partial \Omega_2 \cap \partial \Omega|\, |a|^{-\frac 12} E_\mathrm{surf}\big(b|a|\big)\Big)\,.\]
		Furthermore,  every minimizer $(\psi, \Ab)_{\kappa,H} \in H^1(\Om;\C)\times\Hd $ of the functional  in~\eqref{eq:GL} satisfies
		\begin{equation}\label{eq:psi4}
		\int_\Omega|\psi|^4\,dx=-2E^\mathrm{L}_a(b)\kappa^{-1}+o(\kappa^{-1})\qquad(\kappa\to+\infty)\,.
			\end{equation}
\end{theorem}	
\begin{rem}\label{rem:on-MT}
In the asymptotics displayed in Theorem~\ref{thm:Eg}, the term $|\Gamma|b^{-\frac 12}\mathfrak e_a(b)$ corresponds to the energy contribution of the magnetic barrier. The rest of the terms indicate the energy contributions of the surface of the sample.  
\end{rem}
\emph{Discussion of Theorem~\ref{thm:Eg}.}

We will discuss the result in Theorem~\ref{thm:Eg} in the interesting case where the magnetic barrier $\Gamma$ intersects the boundary of $\Omega$. Hence we will assume that   $\partial \Omega_j \cap \partial \Omega\not=\emptyset$ for $j\in\{1,2\}$. When this condition is violated,  the discussion  below can be adjusted easily.

Theorem~\ref{thm:Eg} leads to the following observation that mainly relies on Remark~\ref{rem:a} and the order of the values $|a|\Theta_0$, $\Theta_0$, $\beta_a$ and $|a|$. 
\begin{itemize}
	\begin{figure}[h] 
		\centering
		\includegraphics[scale=0.3]{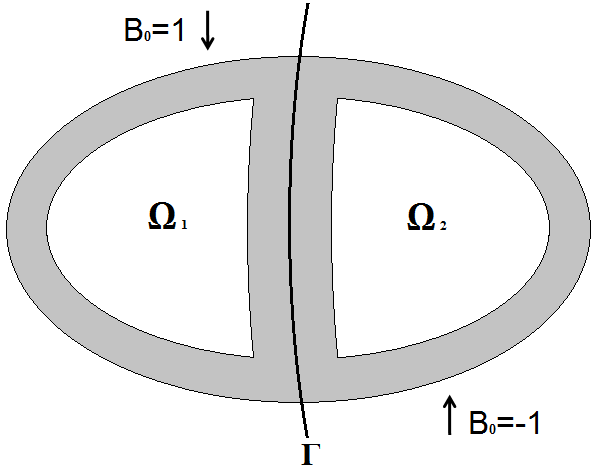} 
		\caption{Superconductivity distribution in the set $\Omega$ subjected to a magnetic field $B_0$, in the regime where $a=-1$, $H=b\kappa$ and $|a|^{-1}<b<\Theta_0^{-1}$. The white regions are in a normal state, while the dark regions carry superconductivity.}
		\label{dis_1}
	\end{figure}
	\item For $a=-1$, we have $\beta_a=\Theta_0<|a|$ (see~\eqref{eq:theta4*}). Consequently,  in light of Remark~\ref{rem:a}:
	   \begin{itemize}
	   	\item If $1<b<\Theta_0^{-1}$, then  the surface of the sample carries superconductivity and the entire bulk is in a normal state  \emph{except} for the region near the magnetic barrier (see Figure~\ref{dis_1}).  Moreover, the energy contributions of the magnetic barrier and the surface of the sample are of the same order and described by the surface energy, since in this case $\mathfrak e_a(b)=E_\mathrm{surf}(b)$,  see~\eqref{eq:conj*}. This behaviour is remarkably  distinct from the  case of a uniform applied magnetic field. 
	   	\item If $b \geq \Theta_0^{-1}$, then all the aforementioned energy contributions vanish, $E_a^\mathrm{L}(b)=0$.
	   \end{itemize}
    \item For $a \in (-1,0)$,  comparing the values $\beta_a$, $\Theta_0$ and $|a|$  is more subtle. In~\eqref{eq:beta},~\eqref{eq:teta_beta1} and Theorem~\ref{thm:up_eigen} below, we proved  that
    \begin{equation}\label{eq:comp}
    \forall~ a\in (-\Theta_0,0), \qquad |a|\Theta_0<\beta_a<|a|<\Theta_0\,.
    \end{equation} 
However, the ordering of $\beta_a$ and $\Theta_0$  is not known yet  for $a\in(-1,-\Theta_0)$.  The inequality $\beta_a<|a|$ is new and  is a slight improvement of the estimates in~\cite{hislop2016band}. With~\eqref{eq:comp} in hand, Theorem~\ref{thm:Eg} and Remark~\ref{rem:a} indicate the following behaviour  for $a \in (-\Theta_0,0)$ and  $b>|a|^{-1}$:
     \begin{figure}[h]
     	\centering
     	\includegraphics[scale=0.3]{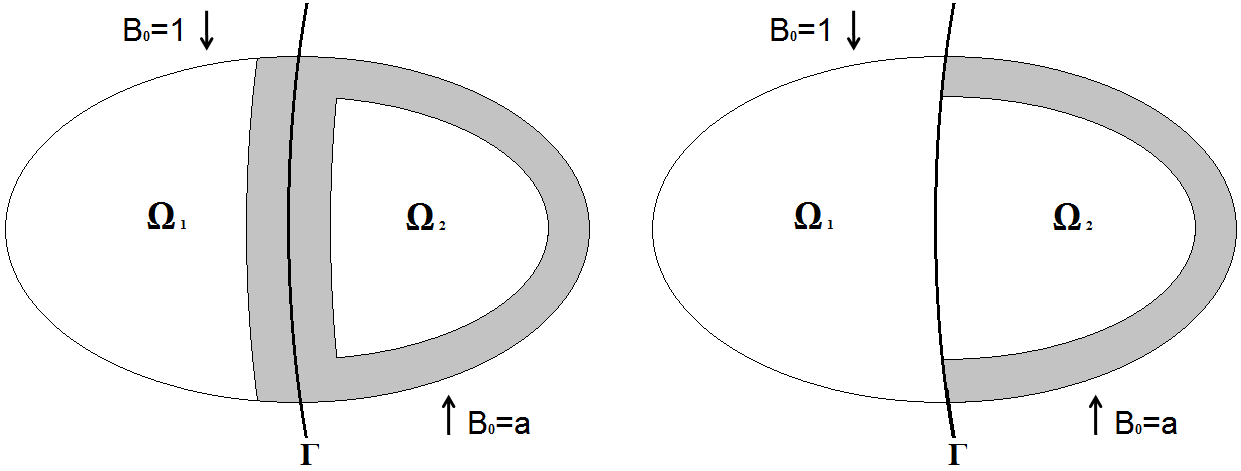} 
     	\caption{Superconductivity distribution in the set $\Omega$ subjected to a magnetic field $B_0$, in the regime where $a\in (-\Theta_0,0)$, $H=b\kappa$ and respectively $|a|^{-1} < b <\beta_a^{-1}$ and $\beta_a^{-1}\leq b<|a|^{-1}\Theta_0^{-1}$. The white regions are in a normal state, while the dark regions carry superconductivity. }
     	\label{dis_2}
     \end{figure}
    \begin{itemize}
    \item  The part of  the sample's surface near $\partial \Omega_1 \cap \partial \Omega$ does not carry superconductivity.
       \item If $|a|^{-1} < b <\beta_a^{-1}$, then \emph{surface superconductivity} is confined to the  \emph{part} of the surface near $\partial\Omega_2 \cap \partial \Omega$. At the same time,  superconductivity is observed  along the magnetic barrier $\Gamma$ (see Figure~\ref{dis_2}),  its strength is described by the function $\mathfrak e_a(b)$. This behaviour is interesting for two reasons. Firstly, it demonstrates the existence of the \emph{edge current} along the magnetic barrier, which is consistent with physics (see~\cite{hislop2016band}). Secondly,  it marks a distinct behaviour from the one known for uniform applied magnetic fields, in which case  the whole surface carries superconductivity evenly (see for instance~\cite{helffer2017decay,fournais2013ground,pan2002surface}).
    \item If $\beta_a^{-1}\leq b<|a|^{-1}\Theta_0^{-1}$, then superconductivity only survives  along  $\partial \Omega_2 \cap \partial \Omega$ (see Figure~\ref{dis_2}).  Our results then display the strength of the applied magnetic field responsible  for the breakdown of the edge current along the barrier.
    \item If $b \geq |a|^{-1}\Theta_0^{-1}$, then all energy contributions in Theorem~\ref{thm:Eg} disappear.
    \end{itemize}
\item For $a \in (0,1)$,  $\beta_a=a$ (see~\eqref{eq:lambda_a}). When $b>a^{-1}$, Theorem~\ref{thm:Eg} reveals the absence of superconductivity along the magnetic barrier. As for the distribution of superconductivity along the surface of the sample,  we distinguish between  two regimes:
	\paragraph{\itshape Regime~1, $a \in (0,\Theta_0]$} The part of the boundary,  $\partial \Omega_1 \cap \partial \Omega$, does not carry superconductivity. It remains to inspect the energy contribution of $\partial\Omega_2\cap\partial\Omega$. In that respect:
	\begin{itemize}
		\item If $a^{-1}<b<a^{-1}\Theta_0^{-1}$, then  superconductivity exists along $\partial \Omega_2 \cap \partial \Omega$.
		\item If $b \geq a^{-1}\Theta_0^{-1}$, then superconductivity disappears along $\partial \Omega_2 \cap \partial \Omega$.
	\end{itemize}
	\paragraph{\itshape Regime~2, $a \in (\Theta_0,1)$} We observe the following~:
	\begin{itemize}
		\item If $a^{-1}<b<\Theta_0^{-1}$, then the entire surface of the sample, $\partial\Omega$, is in a superconducting state, though its distribution is not uniform.
		\item If $\Theta_0^{-1}\leq b<a^{-1}\Theta_0^{-1}$, then only  $\partial \Omega_2 \cap \partial \Omega$ carries superconductivity.
		\item If $b \geq a^{-1}\Theta_0^{-1}$, then all the energy contributions in Theorem~\ref{thm:Eg} vanish.
	\end{itemize}
\end{itemize}

Our next theorem describes the local behaviour of the minimizing order parameter $\psi$. To that end, we define the following distribution in $\R^2$,
\[
C_c^\infty(\R^2) \ni \varphi \mapsto \mathcal T^b(\varphi)\,,\]
where
\[
\mathcal T^b(\varphi)=
-2b^{-\frac 12}
\left(\mathfrak e_a(b)\int_{\Gamma}\varphi \,ds_{\Gamma}
+ E_\mathrm{surf}(b)\int_{\partial \Omega_1 \cap \partial \Omega}\varphi \,ds +|a|^{-\frac 12}E_\mathrm{surf}\big(b|a|\big)\int_{\partial \Omega_2 \cap \partial \Omega}\varphi \,ds \right)\,.
\]
Here  $ds_\Gamma$ and $ds$ denote the arc-length measures  on $\Gamma$ and  $\partial \Omega$ respectively.

\begin{theorem}[Local asymptotics]\label{thm:E_loc}
	For all $a \in[-1,1)\setminus\{0\}$ and $b> 1/|a|$, if $(\psi, \Ab)_{\kappa,H} \in H^1(\Om;\C)\times\Hd $ is a minimizer of the functional  in~\eqref{eq:GL} for $H=b\kappa$, then, as $\kappa\to+\infty$,
	\[\kappa\mathcal T_\kappa^b  \rightharpoonup \mathcal T\ \mathrm{in}\ \mathcal D'(\R^2)\,,\]
where 	 $\mathcal T_\kappa$ is  the distribution in $\R^2$ defined as follows
\[C_c^\infty(\R^2)\ni \varphi\mapsto\mathcal T_\kappa^b(\varphi)=\int_\Omega |\psi|^4 \varphi\,dx\,,\]
and the convergence of $\mathcal T_\kappa^b$ to $\mathcal T$ is understood in the following sense:
\[\forall~\varphi\in C_c^\infty(\R^2)\,,\quad \lim_{\kappa\to+\infty}\kappa\mathcal T_{\kappa}^b(\varphi)=\mathcal T^b(\varphi)\,.\]
\end{theorem}

\subsection{Notation}
\begin{itemize}
\item[]
\item The letter $C$ denotes a positive constant whose value may change from one formula to another. Unless otherwise stated, the constant $C$ depends on the value of $a$ and the domain $\Omega$, and is independent of $\kappa$ and $H$ .
\item Let $a(\kp)$ and $b(\kp)$ be two positive functions, we write :
\begin{itemize}
\item $a(\kp) \ll b(\kp)$, if $a(\kp)/b(\kp) \rightarrow 0$ as $\kp \rightarrow +\infty$.
\item $a(\kp) \approx b(\kp)$, if there exist constants $\kappa_0$, $C_1$ and $C_2$ such that for all $\kp\geq\kp_0$,\break $C_1 a(\kp)\leq b(\kp) \leq C_2 a(\kp).$
\end{itemize}
\item The quantity $o(1)$ indicates a function of $\kappa$, defined by universal quantities, the domain $\Omega$, given functions, etc and  such that $|o(1)|\ll 1$. Any expression $o(1)$ is independent of the minimizer $(\psi,\Ab)$ of~\eqref{eq:GL}. Similarly, $\mathcal O(1)$ indicates a  function of $\kappa$, bounded by a constant independent of the minimizers of~\eqref{eq:GL}. 
\item  Let $n \in \N$, $p \in \N$, $N \in \N$, $\alpha \in (0,1)$,  $K \subset \R^N$ be an open set.   We use the following
H\"{o}lder space:  
\[C^{n,\alpha}({\overline K})=\left\{f \in C^n({\overline K})\ | \sup_{x\neq y\in K}\frac {|D^nf(x)-D^nf(y)|}{|x-y|^\alpha}<+\infty\right\}\,.\]
\item Let $n \in \N$, $I \subset \R$ be an open interval. We introduce the space: 
\begin{equation} \label{eq:B_n}
B^n(I)=\{u \in L^2(I):\,x^i
	D^ju \in L^2(I),\,\forall i, j \in \N\ \mathrm{s.t.}\ i+j\leq n \}\,.
\end{equation}
\end{itemize}
\subsection{Organization of the paper} The rest of the paper is divided into seven sections and one appendix. Section~\ref{sec:prel} presents some preliminaries, particularly, some a priori estimates, exponential decay results, and a linear 2D operator with a step magnetic field. Theorem~\ref{thm:up_eigen} is an improvement of a result in~\cite{hislop2016band}.

Section~\ref{sec:lim-en} introduces a reduced GL energy crucial in the study of superconductivity near the magnetic barrier $\Gamma$ and we introduce the barrier energy $\mathfrak e_a(\cdot)$.

In Section~\ref{sec:bc}, we present the Frenet coordinates defined in a tubular neighbourhood of the curve $\Gamma$.  These coordinates are frequently used  in the study of surface superconductivity (see~\cite[Appendix~F]{fournais2010spectral}).

 Sections~\ref{sec:local_en} and~\ref{sec:local_est} are devoted for the analysis of the local behaviour of the minimizing order parameter near the magnetic barrier  $\Gamma$, while Section~\ref{sec:surf} recalls well-known results about the local behaviour the order parameter near the surface  $\partial \Om$.  
 
Collecting the local estimates established in Sections~\ref{sec:local_est} and~\ref{sec:surf}, we prove in Section~\ref{sec:main} our main theorems (Theorems~\ref{thm:Eg} and \ref{thm:E_loc} above). 

Finally, in the appendix, we collect some common spectral results used throughout the paper.

One remarkable aspect of our proofs is that  we have not used the \emph{a priori} elliptic $L^\infty$-estimate $\|(\nabla-i\kappa H\Ab)\psi\|_\infty\leq C\kappa$. Such estimate is not known to hold in our case of  discontinuous magnetic field $B_0$.  Instead, we used the easy energy estimate  $\|(\nabla-i\kappa H\Ab)\psi\|_2\leq C\kappa$ and the regularity of the $\curl$-$\Div$ system (cf.~Theorem~\ref{thm:priori}).  This also spares us the complex derivation of the $L^\infty$-estimate (see~\cite[Chapter 11]{fournais2010spectral}).  We have made an effort to keep the proofs reasonably self-contained.
\section{Preliminaries}\label{sec:prel}

\subsection{A Priori Estimates}\label{sec:a-p-est}
We present some celebrated estimates needed in the sequel to control the various errors arising while estimating the energy in~\eqref{eq:GL}.

We begin by the following well-known estimate of the
order parameter:
\begin{proposition}\label{prop:psi}
If $(\psi,\Ab) \in H^1(\Om;\C)\times H^1(\Om;\R^2)$ is a weak solution to~\eqref{eq:Euler}, then
\[\|\psi\|_{L^\infty(\Omega)}\leq 1.\]
\end{proposition}
A detailed proof of Proposition~\ref{prop:psi} can be found
in~\cite[Proposition~10.3.1]{fournais2010spectral}. 
The $L^\infty$- bound in Proposition~\ref{prop:psi} is crucial in deriving some \emph{a
priori} estimates on the solutions of the Ginzburg--Landau equations~\eqref{eq:Euler} listed in Theorem~\ref{thm:priori} below. 

Recall the magnetic field $B_0$ introduced in Assumption~\ref{assump}.  In the next lemma, we will fix the gauge for the  magnetic potential  generating  $B_0$ (see~\cite[Lemma A.1]{Assaad}):
\begin{lem}\label{A_1}
	Suppose that  the conditions in Assumption~\ref{assump} hold.  There exists a unique vector field $\Fb
		\in \Hd$ such that
	\[\curl\Fb=B_0\,.\]
	Furthermore, $\Fb$ is in $C^\infty(\Omega_i)$ and in $H^2(\Omega_i)$, $i=1,2$.
\end{lem}

We collect below some useful estimates whose proofs are given in~\cite[Theorem~4.2]{Assaad}.

\begin{thm}\label{thm:priori}
Let $\alpha\in(0,1)$ be a constant. Suppose that the
conditions in Assumption~\ref{assump} hold.
There exist two constants $\kappa_0>0$ and $C>0$ such that, if~\eqref{eq:A_2} is satisfied
 and $(\psi,\Ab)\in H^1(\Om;\C)\times \Hd$ is a solution of~\eqref{eq:Euler}, then 
\begin{enumerate}
\item $\|(\nb-i \kp H\Ab)\psi\|_{L^2(\Om)}\leq C\kp$\,.
\item $\displaystyle\|{\curl(\Ab-\Fb)}\|_{L^2(\Om)}\leq  C/
\kp$\,.
\item $\Ab-\Fb\in H^2(\Omega)$ and $\displaystyle\|\Ab-\Fb\|_{H^2(\Om)}\leq  C/\kp$\,.
\item $\Ab-\Fb\in C^{0,\alpha}(\overline{\Om})$ and $\displaystyle\|\Ab-\Fb\|_{C^{0,\alpha}(\overline \Om)}\leq C/\kp$\,.
\end{enumerate}
\end{thm}

\subsection{Exponential decay of the order parameter}
The following theorem displays a regime for the intensity of the applied magnetic field where the order parameter and the GL energy are exponentially small in the bulk of the domains $\Omega_1$ and $\Omega_2$. 
\begin{theorem}\label{thm:decay}
Given $a \in[-1,1)\setminus\{0\}$ and $b >1/|a|$, there exist constants $\kappa_0>0$, $C>0$, and $\alpha_0>0$  such that, if
\[\kappa \geq \kappa_0,\quad  \kappa_0 \kappa^{-1}\leq \ell<1 \quad \mathrm{and}\ (\psi,\Ab)~\mathrm{is~a~solution~of~}~\eqref{eq:Euler}\ \mathrm{for}\ H=b\kappa\,,\] 
then
\begin{equation*}
\int_{\Omega_j\cap\{\mathrm{dist}(x,\partial\Omega_j)\geq \ell\}}
\Big(|\psi|^2+|(\nabla-i\kappa H\Ab)\psi|^2\Big)\,dx \leq C e^{-\alpha_0\kappa\ell}\,,
\end{equation*}
for $j \in \{1,2\}$.
\end{theorem}
The proof of Theorem~\ref{thm:decay} follows from stronger Agmon-type estimates established in~\cite[Theorems~1.5~\&~7.3]{Assaad}.

\subsection{Families of Sturm--Liouville operators on $L^2(\R_+)$}\label{sec:sturm1}
In this section, we will briefly present some spectral properties of the self-adjoint realization on $L^2(\R_+)$ of the Sturm--Liouville operator:
\begin{equation}\label{eq:H_gamma}
H[\gamma,\xi]=-\frac{d^2}{dt^2}+(t-\xi)^2\,,
\end{equation}
defined over the domain:
\[\dom\big(H[\gamma,\xi] \big)=\{u \in B^2(\R_+)~:~ u'(0)=\gamma u(0)\}\,,\]
where $\xi$ and $\gamma$ are two real parameters, and the space $B^n(\R_+)$ was introduced in~\eqref{eq:B_n}.
The quadratic form associated to $H[\gamma,\xi]$ is
\[B^1(\R_+)\ni u\longmapsto q[\gamma,\xi](u)=\int_0^{+\infty} \Big(|u'(t)|^2+|(t-\xi)u(t)|^2 \Big)\,dt +\gamma|u(0)|^2\,.\]
Since the embedding of the form domain $B^1(\R_+)$ in $L^2(\R_+)$ is compact, the spectrum of $H[\gamma,\xi]$ is an increasing sequence of eigenvalues tending to $+\infty$ .
Denote by $\mu(\gamma,\xi)$ the first eigenvalue of the operator $H[\gamma,\xi]$:
\begin{equation}\label{eq:mu1}
\mu(\gamma,\xi)=\inf \mathrm{sp}\big(H[\gamma,\xi]\big)=\inf_{u\in B^1(\R_+)}\frac{q[\gamma,\xi](u)}{\|u\|^2_{L^2(\R)}}
\end{equation}
and let 
\begin{equation}\label{eq:theta_gamma}
\Theta(\gamma)=\inf_{\xi \in \R} \mu(\gamma,\xi)\,.
\end{equation}

\emph{The Neumann realization.} The particular case where $\gamma=0$ corresponds to the \emph{Neumann} realization, denoted by $H^N[\xi]$, with the associated quadratic form $q^N[\xi]=q[0,\xi]$. The first eigenvalue of $H^N[\xi]$ is denoted by
\begin{equation}\label{eq:mu_n}
\mu^N(\xi)=\inf \mathrm{sp}\big(H^N[\xi]\big)=\mu(0,\xi)\,.
\end{equation}

\emph{The Dirichlet realization.} Besides the Robin and Neumann realizations, we introduce the Dirichlet realization 
\[H^D[\xi]=-\frac{d^2}{dt^2}+(t-\xi)^2\] 
with domain
\[\dom\big(H^D[\xi] \big)=\{u \in B^2(\R_+)~:~ u(0)=0\}\,.\]
The associated quadratic form is defined by
\[B^1(\R_+)\cap H_0^1(\R_+)\ni u\longmapsto q^D[\xi](u)=\int_0^{+\infty} \Big(|u'(t)|^2+|(t-\xi)u(t)|^2 \Big)\,dt\,.\]
We introduce the first eigenvalue of $H^D[\xi]$ as follows
\begin{equation}\label{eq:mu_D}
\mu^D(\xi)=\inf \mathrm{sp}\big(H^D[\xi]\big)=\inf_{u\in B^1(\R_+)\cap H^1_0(\R_+)}\frac{q^D[\xi](u)}{\|u\|^2_{L^2(\R)}}\,.
\end{equation}
The perturbation theory~\cite{Kato} ensures that the functions
\[\xi\mapsto \mu^D(\xi),\quad \xi\mapsto \mu^N(\xi)\quad,\ \mathrm{and} \quad \xi\mapsto \mu(\gamma,\xi)\]
are analytic.

In addition, recall the following well-known Sturm--Liouville theorems (For instance, see~\cite{dauge1993eigenvalues, simon1972methods}):
\begin{theorem}\label{thm:dirichlet}
 The function $\xi\mapsto \mu^D(\xi)$ introduced in~\eqref{eq:mu_D} is decreasing,
\[\displaystyle\lim_{\xi \rightarrow -\infty}\mu^D(\xi)=+\infty\quad\mathrm{and}\quad \displaystyle\lim_{\xi \rightarrow +\infty}\mu^D(\xi)=1\,.\] 
\end{theorem}
Theorems~\ref{thm:sturm} and~\ref{thm:teta_gamma} below are proved in~\cite[Section~2]{Kachmar}.
\begin{theorem}\label{thm:sturm} The following statements hold
\begin{enumerate}
\item For all $(\gamma,\xi)\in\R^2$, the first eigenvalue $\mu(\gamma,\xi)$ of $H[\gamma,\xi]$ defined in~\eqref{eq:mu1} is simple, and there exists a unique eigenfunction $\varphi_{\gamma,\xi}$ satisfying
\[\begin{cases}
-\varphi''_{\gamma,\xi}+(t-\xi)^2\varphi_{\gamma,\xi}=\mu(\gamma,\xi)\varphi_{\gamma,\xi},~\varphi_{\gamma,\xi}>0\ \mathrm{in~}\R_+\,,\\
\varphi'_{\gamma,\xi}(0)=\gamma\varphi_{\gamma,\xi}(0)\,,\\
\int_{\R_+}|\varphi_{\gamma,\xi}(t)|^2\,dt=1\,.\end{cases}
\]
\item For all $\gamma\in\R$,  $\displaystyle\lim_{\xi \rightarrow -\infty}\mu(\gamma,\xi)=+\infty$, and $\displaystyle\lim_{\xi \rightarrow +\infty}\mu(\gamma,\xi)=1$ .
\end{enumerate}
\end{theorem}
\begin{theorem} \label{thm:teta_gamma}
Let  $\Theta(\cdot)$ be the function defined in~\eqref{eq:theta_gamma}. It holds the following:
\begin{enumerate}
\item The function $\Theta(\cdot)$ is continuous and increasing\,.
\item For all $\gamma\in\R$, $-\gamma^2\leq\Theta(\gamma)<1$\,.
\item For all $\gamma<0$, 
$-\gamma^2\leq\Theta(\gamma)\leq-\gamma^2+\frac{1}{2\gamma^2}$\,.
\item For all $\gamma \geq 0$, 
$\Theta(\gamma)>0\,.$
\item There exist  $C_0,\gamma_0>0$  such that, $\forall\, \gamma \in [\gamma_0,+\infty)$,
$1-C_0\gamma \exp{(-\gamma^2)}\leq \Theta(\gamma)\,.$
\item For all $\gamma\in\R$, the function $\xi \mapsto \mu(\gamma,\xi)$ admits a unique minimum attained at 
\begin{equation}\label{eq:psi-gamma}
\xi(\gamma)=\sqrt{\Theta(\gamma)+\gamma^2}\,.
\end{equation}
Furthermore, this minimum is non-degenerate, $\partial^2_\xi\mu\big(\gamma,\xi(\gamma)\big)>0$.
\end{enumerate}
\end{theorem}

\subsection{An operator with a step magnetic field}\label{sec:sturm2}

Let  $a \in [-1,1)\setminus\{0\}$. We consider the magnetic potential $\Ab_0$ defined by
\begin{equation}\label{canon}
\Ab_0(x) = (-x_2,0) \qquad \Big(x = (x_1,x_2) \in \R ^2\Big)
\end{equation}
which satisfies $\curl\Ab_0=1$. We define the step function $\sigma$ as follows. For  $x=(x_1,x_2) \in \R^2$,
\begin{equation}\label{eq:sigma1}
\sigma(x)=\mathbbm{1}_{\R_+}(x_2)+a\mathbbm{1}_{\R_-}(x_2)\,.
\end{equation}
We introduce the self-adjoint \emph{magnetic Hamiltonian} in $L^2(\R^2)$
\begin{equation}\label{eq:ham_operator}
\mathcal L_a=-(\nabla-i\sigma\Ab_0)^2\,,
\end{equation}
where 
\[(\nabla-i\sigma\Ab_0)^2=\left(\Delta -2i\sigma\Ab_0\cdot\nb-\sigma^2 |\Ab_0|^2\right)=\partial^2_{x_2}+\big(\partial_{x_1}+i\sigma x_2\big)^2\,.\]
We denote the ground state energy of the operator $\mathcal L_a$ by
\begin{equation}\label{eq:lamda}
\beta_a= \inf \mathrm{sp}\big(\mathcal L_a \big)\,.
\end{equation}
Since the Hamiltonian defined in~\eqref{eq:ham_operator} is invariant with respect to translations in the $x_1$-direction, we can reduce it to a family of  Shr\"odinger operators on $L^2(\R)$, $\mathfrak h_a[\xi]$, parametrized by $\xi\in\R$ and called \emph{fiber operators} 
(see~\cite{hislop2016band,hislop2015edge}). The operator $\mathfrak h_a[\xi]$ is defined by
\[\mathfrak h_a[\xi]=-\frac{d^2}{dt^2}+V_a(\xi,t)\,,\]
with
\begin{equation}\label{eq:potential}
V_a(\xi,t)=
\begin{cases}
          (\xi+at)^2,& t<0\,,\\
					(\xi+t)^2,& t>0\,.
\end{cases}
\end{equation}
The domain of $\mathfrak h_a[\xi]$ is given by:
\[\dom\big(\mathfrak h_a[\xi]\big)=\left\{u\in B^1(\R)~:~\Big(-\frac{d^2}{dt^2}+V_a(\xi,t)\Big)u \in L^2(\R)\,,~u'(0_+)=u'(0_-)\right\}\]
where  $B^1(\R)$ is defined in~\eqref{eq:B_n}.
The quadratic form associated to $\mathfrak h_a[\xi]$ is
\begin{equation} \label{eq:quad}
q_a[\xi](u)=\int_\R \big(|u'(t)|^2+V_a(\xi,t)|u(t)|^2 \big)\,dt
\end{equation}
defined on the form domain
\begin{equation}\label{eq:dom_q}
\dom\big(q_a[\xi]\big)=B^1(\R)\,.
\end{equation}
The spectra of the operators $\mathcal L_a$ and $\mathfrak h_a[\xi]$ are linked together as follows (see~\cite[Sec.~4.3]{fournais2010spectral})
\begin{equation}\label{eq:L_h}
\mathrm{sp}\big(\mathcal L_a\big)=\overline{\bigcup_{\xi \in \R} \mathrm{sp}\big(\mathfrak h_a[\xi] \big)}\,.
\end{equation}
We introduce the first eigenvalue of the fiber operator $\mathfrak h_a[\xi]$,
\begin{equation}\label{mu_a_1}
\mu_a(\xi)=\inf_{u\in B^1(\R),u\neq0} \frac{q_a[\xi](u)}{\|u\|^2_{L^2(\R)}}\,.
\end{equation}
Consequently, for all $a\in[-1,1)\setminus\{0\}$,  we may express the ground state energy in~\eqref{eq:lamda} by
\begin{equation}\label{eq:beta}
\beta_a=\inf_{\xi \in \R} \mu_a(\xi)\,.
\end{equation}
Below, we collect  some properties of the eigenvalue $\mu_a(\xi)$.
\paragraph{\itshape The case $0<a<1$} This case is studied in~\cite{hislop2015edge,iwatsuka1985examples}. The  eigenvalue $\mu_a(\xi)$ is simple and is a decreasing  function of $\xi$.  The monotonicity of $\mu_a(\cdot)$ and its asymptotics  in Proposition~\ref{prop:mu_a_lim} imply that
\begin{equation*}
a<\mu_a(\xi) < 1\qquad(\xi \in \R)\,,
\end{equation*}
and that $\beta_a$ introduced in~\eqref{eq:lamda} satisfies
\begin{equation}\label{eq:lambda_a}
\beta_a=a\,.
\end{equation}
\paragraph{\itshape The case $a=-1$} This case is studied in~\cite{hislop2016band}. Using symmetry arguments, the operator $\mathfrak h_a[\xi]$  can be linked to the operator $H^N[\xi]$, the Neumann realization on $\R_+$ of $-d^2/dt^2+(t-\xi)^2$ introduced in  Section~\ref{sec:sturm1}. The first eigenvalue $\mu_a(\xi)$ of $\mathfrak h_a[\xi]$ is then simple and satisfies
\begin{equation}\label{eq:ma_mN}
\mu_a(\xi)=\mu^N(-\xi)\,,
\end{equation}
where $\mu^N(\cdot)$ is introduced in~\eqref{eq:mu_n}.
By Theorem~\ref{thm:teta_gamma}, used for $\gamma=0$, we get that
\begin{equation}\label{eq:theta2} 
0<\min_{\xi \in \R} \mu_a(\xi)=\mu_a(\zeta_0)=\Theta_0<1\,,
\end{equation}
where
\begin{equation}\label{eq:theta1}
\zeta_0=-\sqrt{\Theta_0}\quad\mathrm{and}\quad\Theta_0=\Theta(0)~\mathrm{introduced~ in~}~\eqref{eq:theta_gamma} \,.
\end{equation}
Furthermore, the minimum at $\zeta_0$ is  non-degenerate.

\paragraph{\itshape The case $-1<a<0$} See also~\cite{hislop2016band} for the study of this case. The eigenvalue $\mu_a(\xi)$  is simple, and there exists $\zeta_a<0$ satisfying 
\begin{equation}\label{eq:theta3}
|a|\geq \mu_a(\zeta_a)=\min_{\xi \in \R} \mu_a(\xi)\,.
\end{equation}
 Moreover, using the min-max principle, one can easily prove that
\begin{equation}\label{eq:teta_beta1}
|a|\Theta_0<\min_{\xi \in \R} \mu_a(\xi)\,.
\end{equation}

Combining the foregoing discussion in the case $a\in[-1,0)$, we get $\beta_a$ introduced in~\eqref{eq:lamda} satisfies
\begin{equation}\label{eq:theta4}
|a|\Theta_0\leq\beta_a \leq |a|\,, 
\end{equation}
and
\begin{equation}\label{eq:theta4*}
\beta_{-1}=\Theta_0\,.\
\end{equation}
By defining $\zeta_a=-\sqrt{\Theta_0}$ for $a=-1$, we get furthermore that, for all $a\in[-1,0)$,
\begin{equation}\label{eq:theta4**}
\beta_{a}=\mu_a(\zeta_a)\mathrm{~with~}\zeta_a<0\,.
\end{equation}

In the next theorem, we will use a direct approach, different from the one in~\cite{hislop2016band}, to establish the existence of a global minimum $\zeta_a$ in the case where $a \in (-1,0)$ and to prove that  $\beta_a<|a|$. This slightly improves the estimates in~\cite{hislop2016band} (see Remark~\ref{rem:HPRS} below). Theorem~\ref{thm:up_eigen} is necessary to validate  Assumption~\eqref{eq:A2}, under which we work in Section~\ref{sec:lim-en}.
\begin{thm} \label{thm:up_eigen} 

For all $a \in (-1,0)$, there exists $\xi <0$ such that $\mu_a(\xi)$, the first eigenvalue of the operator $\mathfrak h_a[\xi]$, satisfies
\[\mu_a(\xi) < |a|\,.\]
Consequently, the function $\xi\mapsto \mu_a(\xi)$ admits a global minimum satisfying
\[\min_{\xi \in \R} \mu_a(\xi) < |a|\,.\]
\end{thm}
\begin{proof}
The proof here is inspired by~\cite{kachmar2007perfect}.  Define the function
\begin{equation}\label{eq:u}
u(t)=
\begin{cases}
          \varphi_\gamma(0)\exp(-mt),& t\geq 0\,,\\
					\varphi_\gamma(-\sqrt{|a|}t),& t<0\,.
\end{cases}
\end{equation}
where $\gamma$ and $m$ are two positive constants to be fixed later, and $\varphi_\gamma=\varphi_{\gamma,\xi(\gamma)}$ is the normalized eigenfunction defined in Theorem~\ref{thm:teta_gamma} and  associated to the eigenvalue $\Theta(\gamma)=\mu\big(\gamma,\xi(\gamma)\big)$ introduced in~\eqref{eq:theta_gamma}. 
One can check that $u \in \dom\big(q_a[\xi]\big)$, hence by the min-max principle, for all $\xi\in\R$,
\begin{equation} \label{eq:mu_quad}
\mu_a(\xi)\leq\frac{q_a[\xi](u)}{\|u\|^2_{L^2(\R)}}\,.
\end{equation}
Pick $\xi \in \R$. We will choose $\xi$ precisely later. The quadratic form $q_a[\xi](u)$ defined in~\eqref{eq:quad} can be decomposed as follows:
\[q_a[\xi](u)=q_a^{(1)}[\xi](u)+q_a^{(2)}[\xi](u)\]
where
\[q_a^{(1)}[\xi](u) =\int_0^{+\infty} \big(|u'(t)|^2+|(t+\xi)u(t)|^2\big)\,dt\,,\]
and
\[q_a^{(2)}[\xi](u)=\int_{-\infty}^0 \big(|u'(t)|^2+|(at+\xi)u(t)|^2\big)\,dt\,.\] 
A simple computation gives
\begin{equation} \label{eq:quad1}
q_a^{(1)}[\xi](u)=\Big(\frac m2 +\frac {\xi^2}{2m} +\frac {\xi}{2m^2} +\frac 1{4m^3}\Big) |\varphi_\gamma(0)|^2\,.
\end{equation}
On the other hand, 
for $t<0$, $u(t)=\varphi_\gamma(-\sqrt{|a|}t)$, so we do the change of variable $y=-\sqrt{|a|}t$ which in turn yields
\begin{equation*} 
q_a^{(2)}[\xi](u)=\sqrt{|a|} \int_0^{+\infty} \left(\big|\varphi'_\gamma(y)\big|^2+\left|\left(y+\frac \xi {\sqrt{|a|}}\right)\varphi_\gamma(y)\right|^2\right)\,dy\,.
\end{equation*}
Now we select $\xi=-\sqrt{|a|}\xi(\gamma)$, where $\xi(\gamma)$ is the value defined in Theorem~\ref{thm:teta_gamma}. That way we get
\begin{equation} \label{eq:quad2}
q_a^{(2)}[\xi](u)=\sqrt{|a|}\, \big(\Theta(\gamma)-\gamma|\varphi_\gamma(0)|^2\big)\,.
\end{equation}
The definition of the function $u$ in~\eqref{eq:u} yields
\begin{equation} \label{eq:norm_u}
\int_{-\infty}^{+\infty} |u(t)|^2\,dt= \frac {|\varphi_\gamma(0)|^2}{2m} +\frac 1{\sqrt{|a|}}\,.
\end{equation}
Combining the results in~\eqref{eq:quad1}--\eqref{eq:norm_u} and using Theorem~\ref{thm:teta_gamma}, we rewrite~\eqref{eq:mu_quad} as follows
\begin{align*}
\mu_a(\xi)&\leq \frac {\sqrt{|a|}\Theta(\gamma)+\Big(\frac m2 -\sqrt{|a|}\gamma+\frac {\xi^2}{2m} +\frac {\xi}{2m^2} +\frac 1{4m^3}\Big)|\varphi_\gamma(0)|^2}{\frac 1 {\sqrt{|a|}}+ \frac {|\varphi_\gamma(0)|^2}{2m}}\\
                &= \frac {\sqrt{|a|}\Theta(\gamma)+\Big(\frac m2 -\sqrt{|a|}\gamma+\frac {|a|\Theta(\gamma)}{2m} +\frac {|a|\gamma^2}{2m}-\frac {\sqrt{|a|\big(\Theta(\gamma)+\gamma^2\big)}}{2m^2} +\frac 1{4m^3}\Big)|\varphi_\gamma(0)|^2}{\frac 1 {\sqrt{|a|}}+ \frac {|\varphi_\gamma(0)|^2}{2m}}\,.
\end{align*}
By Theorem~\ref{thm:teta_gamma},  for all $\gamma>0$ we have $0<\Theta(\gamma)<1$. Thus
\[\mu_a(\xi)\leq\frac {\sqrt{|a|}\Theta(\gamma)+\Big(\frac m2 -\sqrt{|a|}\gamma+\frac {|a|}{2m} +\frac {|a|\gamma^2}{2m}-\frac {\sqrt{|a|}\gamma}{2m^2} +\frac 1{4m^3}\Big)|\varphi_\gamma(0)|^2}{\frac 1 {\sqrt{|a|}}+ \frac {|\varphi_\gamma(0)|^2}{2m}}\,.\]
To finish the proof of the lemma, we choose $\gamma= \sqrt{ 1/(2|a|(1-|a|))}$ and  $m=\sqrt{|a|}\gamma$ so that the quantity
$ m/2 -\sqrt{|a|}\gamma+|a|/2m +|a|\gamma^2/2m-\sqrt{|a|}\gamma/2m^2 + 1/4m^3$ vanishes. Using the fact that $\Theta(\gamma)<1$, we obtain
\begin{equation}\label{eq:eigen_up}
\mu_a(\xi)\leq\frac {\sqrt {|a|} \Theta(\gamma)}{\frac 1{\sqrt {|a|}}+ \frac {|\varphi_\gamma(0)|^2}{2m}}<|a|\Theta(\gamma)<|a|\,.
\end{equation}
Now, the existence of the global minimum becomes a consequence of  Theorem~\ref{thm:regularity} and Proposition~\ref{prop:mu_a_lim} in Appendix~A.
\end{proof}
\begin{rem}\label{rem:HPRS}
Note that our proof of Theorem~\ref{thm:up_eigen} yields the  upper bound
\[\forall~a\in(-1,0)\,,\quad \beta_a < |a|\Theta\left(\,\sqrt{\frac1{2|a|(1-|a|)}}\,\right)\,,\]
which is stronger than  $\displaystyle\beta_a< |a|$.
\end{rem}

Collecting~\eqref{eq:lambda_a}--\eqref{eq:teta_beta1} and the result in Theorem~\ref{thm:up_eigen}, we deduce the following facts regarding the bottom of the spectrum of the   operator $\mathcal L_a$ introduced in~\eqref{eq:ham_operator}.

\begin{proposition}\label{prop:op-La}
For all  $a\in[-1,1)\setminus\{0\}$, let $\mathcal L_a$ and $\beta_a$ be as in~\eqref{eq:ham_operator} and~\eqref{eq:lamda} respectively, and $\Theta_0$ be as in~\eqref{eq:theta1}. The following statements hold
\begin{enumerate}
\item For all $a\in(0,1)$, $\beta_a=a$ .
\item For all $a\in[-1,0)$, $|a|\Theta_0 \leq\beta_a<|a|$ , there exist $\zeta_a<0$ and a function $\phi_a\in L^2(\R)$ such that
\begin{equation}\label{eq:phi}
\int_{\R} | \phi_a|^2\,dt=1, \qquad -\phi''_a(t)+V_a(\zeta_a,t)\phi_a(t)=\beta_a\phi_a(t)\ \mathrm{in}\ \R\,,
\end{equation}
where $V(\cdot,\cdot)$ is introduced in~\eqref{eq:potential}.
\item For all $a\in[-1,0)$, the function 
\begin{equation}\label{eq:psi_0} 
\psi_a(x_1,x_2)=e^{i\zeta_ax_1}\phi_a(x_2)
\end{equation}
is a bounded eigenfunction of the operator $\mathcal L_a$ and satisfies
\[\mathcal L_a\psi_a=\beta_a\psi_a\,.\]
\end{enumerate}
\end{proposition}
\section{Reduced Ginzburg--Landau Energy}\label{sec:lim-en}

\subsection{The functional and the main result}
Assume that $a\in[-1,1)\setminus\{0\}$ is fixed,  $\sigma$ is the step function defined in~\eqref{eq:sigma1} and $\Ab_0$ is the magnetic potential  defined in~\eqref{canon}. For every $R>0$, consider the strip
\begin{equation}\label{eq:SR}
S_R=(-R/2,R/2)\times(-\infty,+\infty)\,.
\end{equation}
We introduce the space
\begin{equation}\label{eq:DR}
\mathcal D_R=\left\{u \in L^2(S_R)~:~\big(\nabla-i\sigma \mathbf A_0 \big)u \in L^2(S_R),~u\Big(x_1=\pm \frac R2,x_2\Big)=0 \right\}\,.\end{equation}
For $b> 0$, we define the following Ginzburg--Landau energy  on $\mathcal D_R$ by
\begin{equation}\label{eq:Gb}
\mathcal G_{a,b,R}(u)=\int_{S_R} \left(b\big|(\nb-i\sigma \Ab_0)u\big|^2-|u|^2+\frac 12 |u|^4\right)\,dx\,,
\end{equation}
along with the ground state energy
\begin{equation}\label{eq:m0-mN}
\begin{aligned}
\mathfrak g_a (b,R)=\displaystyle\inf_{u \in \mathcal D_R}
\mathcal G_{a,b,R}(u)\,. 
\end{aligned}
\end{equation}
Our objective is to prove 

\begin{theorem}\label{thm:eba}
Assume that $a \in [-1,1)\setminus\{0\}$, $b \geq 1/|a|$, $R>0$, $\mathfrak g_a (b,R)$ is the ground state energy in~\eqref{eq:m0-mN}, and $\beta_a$ is  defined in~\eqref{eq:lamda}.

The following holds:
\begin{enumerate}
\item $\mathfrak g_a (b,R)\leq 0$.
\item If $a \in (0,1)$,  then $\mathfrak g_a (b,R)=0$.
\item If $a \in [-1,0)$,  then there exists a constant $\mathfrak e_a(b)\leq0$ such that 
\begin{equation}\label{eq:eba1}
\lim_{R\rightarrow +\infty}\frac {\mathfrak g_a (b,R)}{R}=\mathfrak e_a(b)\,.
\end{equation}
Furthermore, $\mathfrak e_a(b)=0$ if and only if $b\geq 1/\beta_a$.
\item For all $a \in [-1,0)$, the function $\left[1/|a|,+\infty\right)\ni b \mapsto \mathfrak e_a(b)$ is  monotone non-decreasing and continuous.
\item For all $a \in [-1,0)$, there exists a universal positive constant $C$  such that 
\begin{equation}\label{eq:eba2}
\forall\,R\geq 4\,,\quad \mathfrak e_a(b) \leq \frac {\mathfrak g_a (b,R)}R \leq \mathfrak e_a(b)+C\frac {b^2}{R^{\frac13}}\,.
\end{equation}
\end{enumerate}
\end{theorem}
The proof of Theorem~\ref{thm:eba} will occupy the rest of this section through a sequence of lemmas.

\subsection{The trivial case}

 We start by handling the trivial situation where the ground state energy vanishes:

\begin{lemma}\label{lem:trivial}
If $a \in [-1,1)\setminus\{0\}$ and $b\geq 1/\beta_a$, then for all $R>0$, 
	  $\mathfrak g_a (b,R)=0$.
\end{lemma}
\begin{rem}\label{rem:trivial}~
\begin{enumerate}
\item 
Under the assumptions in Lemma~\ref{lem:trivial}, the function $u=0\in\mathcal D_R$ is a minimizer  of the functional in~\eqref{eq:Gb}. 
\item When $a \in (0,1)$, $\beta_a=a$ by Proposition~\ref{prop:op-La}, hence Lemma~\ref{lem:trivial} yields that  $\mathfrak g_a (b,R)= 0$ for all $b\geq 1/a$ and $R>0$.
\end{enumerate}
\end{rem}
\begin{proof}[Proof of Lemma~\ref{lem:trivial}] 
We have the obvious upper bound	$\mathfrak g_a (b,R)\leq \mathcal G_{a,b,R}(0)=0$. Next we prove the lower bound $\mathfrak g_a (b,R)\geq 0$. Pick an arbitrary function $u \in \mathcal D_R$ and extend it  by zero on $\R^2$. Using the min-max principle, we get 
	\[\mathcal G_{a,b,R}(u)\geq b\beta_a\int_{S_R} |u|^2\,dx+\int_{S_R}\Big(-|u|^2+\frac 12 |u|^4\Big)\,dx \geq 0~\mathrm{since~}b\geq \frac1{\beta_a}\,.\]
Minimizing over $u\in\mathcal D_R$\,, we get  $\mathfrak g_a (b,R)\geq 0$.
\end{proof}

\subsection{Existence of minimizers}

Now we handle the following case (which is complementary to the one in Lemma~\ref{lem:trivial}):
\begin{equation}\label{eq:A2}
-1\leq a<0\quad \mathrm{and}\quad \frac 1{|a|} \leq b < \frac 1 {\beta_a}\,,
\end{equation} 
where $\beta_a$ is the first eigenvalue introduced in~\eqref{eq:lamda} .
Under Assumption~\eqref{eq:A2}, we will demonstrate the existence of a  minimizer of the functional in~\eqref{eq:Gb} along with an estimate of its decay at infinity. This is the content of
\begin{proposition}\label{prop:Ln}
Assume that~\eqref{eq:A2} holds. For all  $R>0$, there exists a function $\varphi_{a,b,R}\in \mathcal D_R$ such that
\begin{equation}\label{eq:phi_norm}
\mathcal G_{a,b,R}(\varphi_{a,b,R})=\mathfrak g_a (b,R)\quad\mathrm{and}\quad \|\varphi_{a,b,R}\|_{L^\infty(S_R)}\leq 1\,.
\end{equation}
Here $\mathcal G_{a,b,R}$ is the functional introduced in~\eqref{eq:Gb} and $\mathfrak g_a(b,R)$ is the ground state energy introduced in~\eqref{eq:m0-mN}.

Furthermore, there exists a universal constant $C>0$ such that, for all $R>0$, the function $\varphi_{a,b,R}$ satisfies
\begin{equation}\label{eq:Ln}
\int_{S_R\cap\{|x_2|\geq 4\}} \frac{|x_2|}{\big(\ln|x_2|\big)^2} \Big(\big|(\nabla-i\sigma\Ab_0)\varphi_{a,b,R} \big|^2+|\varphi_{a,b,R}|^2 \Big)\,dx\leq CbR\,,
\end{equation}
\begin{equation}\label{eq:Ln1}
\int_{S_R\cap\{|x_2|\geq 4\}} \frac{|x_2|^3}{\big(\ln|x_2|\big)^2}|\varphi_{a,b,R}|^4\,dx\leq Cb^2R\,,
\end{equation}
and
\begin{equation}\label{eq:phi_decay}
\int_{S_R} \Big(b\big|(\nabla-i\sigma\Ab_0)\varphi_{a,b,R}\big|^2+|\varphi_{a,b,R}|^2 \Big)\,dx\leq CbR\,.
\end{equation}
\end{proposition}
In the proof of Proposition~\ref{prop:Ln}, we will use the approach in~\cite[Theorem~3.6]{fournais2013ground} and~\cite{pan2002surface} which can be described in a heuristic manner as follows. The unboundedness of the set $S_R$ makes the existence of the minimizer $\varphi_{a,b,R}$ in~\eqref{eq:phi_norm} non-trivial. In the following, we consider a reduced Ginzburg--Landau energy $\mathcal G_{a,b,R,m}$ defined on the bounded set $S_{R,m}=(-R/2,R/2)\times(-m,m)$, and we establish some decay estimates of its minimizer~$\varphi_{a,b,R,m}$. Later, using a limiting  argument on $\mathcal G_{a,b,R,m}$ and $\varphi_{a,b,R,m}$ for large values of $m$, we obtain the existence of the minimizer $\varphi_{a,b,R}$ together with the decay properties in Proposition~\ref{prop:Ln}.

Since the proof of Proposition~\ref{prop:Ln} is lengthy, we opt to divide it into several lemmas. First, 
for every $m \in \N$, we introduce the set $S_{R,m}=(-R/2,R/2)\times(-m,m)$ and the  functional
\begin{equation}\label{eq:bound_en}
\mathcal G_{a,b,R,m}(u)=\int_{S_{R,m}} \left(b\big|(\nb-i\sigma \Ab_0)u\big|^2-|u|^2+\frac 12 |u|^4\right)\,dx 
\end{equation}
defined over the space 
\begin{multline}\label{eq:DR,m}
\mathcal D_{R,m}=\biggl\{u \in L^2(S_{R,m})~:~\big(\nabla-i\sigma\mathbf A_0 \big)u \in L^2(S_{R,m}),\\
~u\Big(x_1=\pm \frac R2,\cdot\Big)=u\Big(\cdot,x_2=\pm m \Big)=0\,\biggr\}\,.
\end{multline}
Here $\sigma$ was defined in~\eqref{eq:sigma1}. Now we define
the ground state energy
\begin{equation} \label{eq:gr_m}
\mathfrak g_a(b,R,m)=\inf_{u \in \mathcal D_{R,m}} \mathcal G_{a,b,R,m}(u)\,.
\end{equation}
\begin{lemma}\label{lem:Ln_mm}
	Assume that~\eqref{eq:A2} holds. There exists a univeral constant $C>0$, and for all $R>0,\ m\geq 1$, there exists a function $\varphi_{a,b,R,m}\in \mathcal D_{R,m}$ satisfying,
	\begin{equation}\label{eq:phi_norm_m} \|\varphi_{a,b,R,m}\|_{L^\infty(S_{R,m})}\leq 1\,,
	\end{equation}
	\begin{equation}\label{eq:Ln_m}
	\int_{S_{R,m}\cap\{|x_2|\geq 4\}} \frac{|x_2|}{\big(\ln|x_2|\big)^2} \Big(\big|(\nabla-i\sigma\Ab_0)\varphi_{a,b,R,m} \big|^2+|\varphi_{a,b,R,m}|^2 \Big)\,dx\leq CbR\,,
	\end{equation}
	\begin{equation}\label{eq:Ln_m1}
	\int_{S_{R,m}\cap\{|x_2|\geq 4\}} \frac{|x_2|^3}{\big(\ln|x_2|\big)^2}|\varphi_{a,b,R,m}|^4\,dx\leq Cb^2R\,,
	\end{equation}
	and
	\begin{equation}\label{eq:G_m}
	\mathcal G_{a,b,R,m}(\varphi_{a,b,R,m})=\mathfrak g_a(b,R,m)\,.	
	\end{equation}
	Here	  $\mathcal G_{a,b,R,m}$ is the functional introduced in~\eqref{eq:bound_en} and $\mathfrak g_a(b,R,m)$ is the ground state energy introduced in~\eqref{eq:gr_m}.
\end{lemma}
\begin{proof}
The proof is reminiscent of the one in~\cite[Theorem~4.1]{pan2002surface}.  The boundedness and the regularity of the domain $S_{R,m}$ guarantee  the existence of a minimizer $\varphi_{m}:=\varphi_{a,b,R,m}$ of $\mathcal G_{a,b,R,m}$ in $\mathcal D_{R,m}$, satisfying
\begin{equation}\label{eq:var_m}
-b(\nb-i\sigma \Ab_0)^2\varphi_m=(1-|\varphi_m|^2)\varphi_m\qquad \mathrm{in}\ S_{R,m}\,,
\end{equation}
see e.g.~\cite[Chapter~11]{fournais2010spectral}. Furthermore, Proposition 10.3.1 in~\cite{fournais2010spectral} ensures that
\begin{equation*}
\|\varphi_m\|_{L^\infty(S_{R,m})}\leq 1\,.
\end{equation*}
Next, select $\chi \in C^\infty(\R)$ such that 
$\chi(x_2)=0$ if $|x_2|\leq 1$, and $\chi(x_2)=|x_2|^{2}/\ln|x_2|$ if $|x_2|\geq 4$\,.
The function $\chi$ consequently satisfies 
\[0<|\chi'(x_2)|<\frac{3\sqrt{|x_2|}}{2\ln|x_2|} \qquad \text{for all}\ |x_2|\geq 4\,.\] 
Multiply~\eqref{eq:var_m} by $\chi^2\overline{\varphi_m}$ and integrate by parts, 
\begin{equation}\label{eq:var_chi}
\int_{S_{R,m}} \left(b\big|(\nb-i\sigma \Ab_0)\chi\varphi_m\big|^2-\chi^2|\varphi_m|^2+ \chi^2|\varphi_m|^4\right)\,dx=b \int_{S_{R,m}}\chi'^2|\varphi_m|^2\,dx\,.
\end{equation}
Since the function $x\mapsto \chi(x_2)\varphi_m(x)$ is supported in $S_{R,m}\cap\{|x_2|\geq 1\}$ where $\curl(\sigma\mathbf A_0)=\sigma$, we can apply the spectral inequality in~\cite[Lemma 1.4.1]{fournais2010spectral} to get, under the assumption $1/|a| \leq b < 1/ \beta_a$,
\begin{align}
b\int_{S_{R,m}}\big|(\nb-i\sigma \Ab_0)\chi\varphi_m\big|^2\,dx 
                                                     &\geq b \int_{S_{R,m}}|\sigma|\chi^2|\varphi_m|^2\,dx  \nonumber\\
																										 &\geq \int_{S_{R,m}}\chi^2|\varphi_m|^2\,dx\,. \label{eq:spectral}
\end{align}
It follows from~\eqref{eq:var_chi} and~\eqref{eq:spectral}
\begin{align} 
\int_{S_{R,m}}\chi^2(x_2)|\varphi_m|^4\,dx &\leq b \int_{S_{R,m}}\chi'^2(x_2)|\varphi_m|^2\,dx \nonumber\\
                                  &\leq b \int_{S_{R,m}\cap\{|x_2|\geq 4 \}}\chi'^2(x_2)|\varphi_m|^2\,dx + b\int_{S_{R,m}\cap\{|x_2|< 4 \}}\chi'^2(x_2)|\varphi_m|^2\,dx \nonumber\\
	&\leq Cb  \int_{S_{R,m}\cap\{|x_2|\geq 4 \}}\frac{|x_2|}{\big(\ln|x_2|\big)^2}|\varphi_m|^2\,dx + C bR \,.\label{eq:chi_chi'}
\end{align}
Using the H\"older inequality,
\begin{align} 
 &\int_{S_{R,m}\cap\{|x_2|\geq 4 \}}\frac{|x_2|}{\big(\ln|x_2|\big)^2}|\varphi_m|^2\,dx\\
   &\quad \leq \left (\int_{S_{R,m}\cap\{|x_2|\geq 4 \}}\frac 1 {|x_2|(\ln|x_2|)^2}\,dx \right)^\frac 12 \left(\int_{S_{R,m}\cap\{|x_2|\geq 4 \}}\frac{|x_2|^3}{(\ln|x_2|)^2}|\varphi_m|^4\,dx \right)^\frac 12  \nonumber\\
&\quad\leq C R^\frac 12 \left(\int_{S_{R,m}\cap\{|x_2|\geq 4 \}}\frac{|x_2|^3}{|(\ln|x_2|)^2}|\varphi_m|^4\,dx \right)^\frac 12\,. \label{eq:chi_1}
\end{align}
Now, using Cauchy-Schwarz inequality, the properties of $\chi$  in~\eqref{eq:chi_chi'} and~\eqref{eq:chi_1}, we obtain
\begin{equation}\label{eq:chi_1*}
\begin{aligned}
\int_{S_{R,m}\cap\{|x_2|\geq 4 \}}\frac{|x_2|^3}{(\ln|x_2|)^2}|\varphi_m|^4\,dx &\leq \int_{S_{R,m}}\chi^2(x_2)|\varphi_m|^4\,dx   \\
                                                                          &\leq C {R^\frac 12}b \left(\int_{S_{R,m}\cap\{|x_2|\geq 4 \}}\frac{|x_2|^3}{|(\ln|x_2|)^2}|\varphi_m|^4\,dx \right)^\frac 12+ C bR  \\
																																					&\leq C {b^2} R +C bR\,. 
\end{aligned}
\end{equation}
Consequently, under the assumption $1\leq 1/|a| \leq b < 1 /\beta_a$, we get~\eqref{eq:Ln_m1}.
Inserting~\eqref{eq:Ln_m1} into~\eqref{eq:chi_1}, we get
\begin{equation}\label{eq:chi_3}
\int_{S_{R,m}\cap\{|x_2|\geq 4 \}}\frac{|x_2|}{\big(\ln|x_2|\big)^2}|\varphi_m|^2\,dx  \leq CbR\,.
\end{equation}
We still need to establish
\begin{equation}\label{eq:chi_3*}
\int_{S_{R,m} \cap\{|x_2|\geq 4 \}}\frac{|x_2|}{\big(\ln|x_2|\big)^2}\big|(\nabla-i\sigma\Ab_0)\varphi_m \big|^2\,dx \leq CbR\,.\end{equation}
To that end, we select $\eta \in C^\infty(\R)$ such that $\eta(x_2)=0$ if $|x_2|\leq 1$, and $\eta(x_2)=\sqrt{|x_2|}/\ln|x_2|$ if $|x_2|\geq 4$.
Multiplying the equation in~\eqref{eq:var_m} by $\eta\overline{\varphi_m}$ and integrating over $S_{R,m}$, we get 
\begin{multline}\label{eq:ims}
b \int_{S_{R,m}\cap\{|x_2|\geq 4 \}} \big|(\nabla-i\sigma\Ab_0)\eta(x_2)\varphi_m \big|^2\,dx\\
 = \int_{S_{R,m}\cap\{|x_2|\geq 4 \}} \Big(\eta^2(x_2)|\varphi_m|^2-\eta^2(x_2)|\varphi_m|^4+b\eta'^2(x_2)|\varphi_m|^2 \Big) \,dx\,.
\end{multline}
It is easy to check by a straightforward computation and  Cauchy's inequality that
\begin{align*}
\eta^2(x_2)&\big|(\nabla-i\sigma\Ab_0)\varphi_m \big|^2\\
&\leq\big|(\nabla-i\sigma\Ab_0)\eta(x_2)\varphi_m \big|^2+2\big|Re\left\langle \varphi_m \eta'(x_2),\eta(x_2)(\nabla-i\sigma\Ab_0)\varphi_m\right\rangle  \big|-\eta'^2(x_2)|\varphi_m|^2\,,\\
       &\leq \big|(\nabla-i\sigma\Ab_0)\eta(x_2)\varphi_m \big|^2+\frac 12\eta^2(x_2)\big|(\nabla-i\sigma\Ab_0)\varphi_m \big|^2+\eta'^2(x_2)|\varphi_m|^2\,.
\end{align*}
Integrating, we get
\begin{multline}  \label{eq:eta3}
\int_{S_{R,m}\cap\{|x_2|\geq 4 \}} \eta^2(x_2)\big|(\nabla-i\sigma\Ab_0)\varphi_m \big|^2\,dx\\
\leq 2\int_{S_{R,m}\cap\{|x_2|\geq 4 \}} \Big(\big|(\nabla-i\sigma\Ab_0)\eta(x_2)\varphi_m \big|^2 + \eta'^2(x_2)|\varphi_m|^2\Big) \, dx\,.
\end{multline}
Combining~\eqref{eq:ims} and~\eqref{eq:eta3}, we get
\begin{multline} \label{eq:eta4}
b \int_{S_{R,m}\cap\{|x_2|\geq 4 \}} \eta^2(x_2)\big|(\nabla-i\sigma\Ab_0)\varphi_m \big|^2\,dx\leq 2\int_{S_{R,m}\cap\{|x_2|\geq 4 \}} \eta^2(x_2)|\varphi_m|^2\,dx\\
+4b\int_{S_{R,m}\cap\{|x_2|\geq 4 \}}\eta'^2(x_2)|\varphi_m|^2 \,dx\,. 
\end{multline}
The definition of $\eta$ yields that, in $S_{R,m}\cap \{|x_2|\geq 4\}$, $\eta^2=|x_2|/(\ln|x_2|)^2$, and $\eta'^2\leq 4\eta^2$. Hence,~\eqref{eq:chi_3} and~\eqref{eq:eta4} imply~\eqref{eq:chi_3*}.
\end{proof}
\begin{corollary}\label{cor:varphi_m2}
	There exists a universal constant $C>0$ such that, if~\eqref{eq:A2} holds, the minimizer  $\varphi_{a,b,R,m}$  in Lemma~\ref{lem:Ln_mm} satisfies, for all $R>0,\ m \in \N$,
	\begin{equation}\label{eq:phi_m_decay}
	\int_{S_{R,m}} b\Big|\big(\nabla-i\sigma\Ab_0\big)\varphi_{a,b,R,m}\Big|^2+|\varphi_{a,b,R,m}|^2\,dx \leq CbR\,.
	\end{equation} 
\end{corollary}
\begin{proof}
For the sake of brevity, we will write $\varphi_m$ for $\varphi_{a,b,R,m}$.
	Using~\eqref{eq:chi_3} and the fact that $ |x_2|/\big(\ln|x_2|\big)^2\geq 1$,  we get
	\[\int_{S_{R,m}\cap\{|x_2|\geq 4 \}}|\varphi_m|^2\,dx  \leq CbR\,.\]
	On the other hand, using $\|\varphi_m\|_\infty \leq 1$ and $b>1$ we get
	\[\int_{S_{R,m}\cap\{|x_2|< 4 \}}|\varphi_m|^2\,dx  \leq C bR.\]
	Next, since $\varphi_m$  satisfies
	\[-b(\nb-i\sigma \Ab_0)^2\varphi_m=(1-|\varphi_m|^2)\varphi_m\qquad \mathrm{ in}\ S_{R,m}\,,\]
	 a simple integration by parts over $S_{R,m}$ yields
	\begin{align*}
	\int_{S_{R,m}} b\big|(\nabla-i\sigma\Ab_0)\varphi_m \big|^2\,dx &=\int_{S_{R,m}}|\varphi_m|^2\,dx-\int_{S_{R,m}}|\varphi_m|^4\,dx\\
	&\leq \int_{S_{R,m}}|\varphi_m|^2\,dx\\
	&\leq CbR\,.
	\end{align*}
\end{proof}
Now, we will investigate the regularity  of the minimizer $\varphi_{a,b,R,m}$ in~Lemma~\ref{lem:Ln_mm}. We have to be careful  at this point since the magnetic field is a step function and therefore has singularities. As  a byproduct, we will extract a convergent subsequence of $(\varphi_{a,b,R,m})_{m\geq 1}$.

We will use the following terminology. Let $\Omega\subset\R^2$ be an open set. If $(u_m)_{m\geq 1}$ is a sequence in $H^k(\Omega)$, then by saying that $(u_m)$ is  bounded/convergent in $H^k_\mathrm{loc}(\Omega)$, we mean that it is bounded/convergent in $H^k(K)$, for every $K\subset\Omega$ open and relatively compact. A similar terminology applies for boundedness/convergence in $C_\mathrm{loc}^{k,\alpha}(\Omega)$: A sequence  $(u_m)_{m\geq 1}$ is bounded/convergent in $C_\mathrm{loc}^{k,\alpha}(\Omega)$ if it is bounded/convergent in $C^{k,\alpha}(\overline{K})$, for every $K\subset\Omega$ open and relatively compact.

\begin{lemma}\label{lem:phim_bound}
	Assume that~\eqref{eq:A2} holds. Let $R>0$ and $\alpha\in(0,1)$ be fixed. The sequence $\big(\varphi_{a,b,R,m}\big)_{m \geq 1}$ defined by Lemma~\ref{lem:Ln_mm} is bounded in $H_\mathrm{loc}^3(S_R)$ and consequently in $C_\mathrm{loc}^{1,\alpha}(S_R)$.
	\end{lemma}
\begin{proof}
	For simplicity, we will write $\varphi_m=\varphi_{a,b,R,m}$. The proof is split into three steps.
	
	\paragraph{\itshape Step~1} 
	
	We first prove the boundedness of $\big(\varphi_m\big)$ in $H^2_\mathrm{loc}(S_R)$. 
	 Using~\eqref{eq:var_m} we may write
	\begin{equation}\label{eq:delta_phi}
	\Delta \varphi_m=\frac 1b\Big(|\varphi_m|^2-1 \Big)\varphi_m+2i\sigma\Ab_0\cdot \nabla \varphi_m+|\sigma|^2|\Ab_0|^2\varphi_m\,.
	\end{equation}
Assume that $K \subset S_R$ is  open and relatively compact. Choose an open and bounded set $\widetilde K$ such that $\overline{K}\subset \widetilde K\subset S_R$. There exists $m_0 \in \N$ such that for all $m \geq m_0$, $\widetilde K \subset S_{R,m}$ and by  Cauchy's inequality,
	\[\int_{\widetilde K} |\nabla\varphi_m|^2\,dx \leq 2\int_{\widetilde K} \big|\big(\nabla-i \sigma \Ab_0 \big)\varphi_m \big|^2\,dx+2\int_{\widetilde K} |\sigma|^2|\Ab_0|^2|\varphi_m|^2\,dx\,.\]
	Using $|\varphi_m| \leq 1$, the decay estimate in~\eqref{eq:phi_m_decay} and the boundedness of $\sigma$ and $\Ab_0$ in $\widetilde K$, we get a constant $C=C(\widetilde K,R)$ such that
	\[
\int_{\widetilde K} |\nabla\varphi_m|^2\,dx \leq C\,,
	\]
	and
	\begin{equation*}
	\int_{\widetilde K} |\Delta\varphi_m|^2\,dx\leq C\,,
	\end{equation*}
	in light of~\eqref{eq:delta_phi}.  
	By the interior elliptic estimates (see for instance~\cite[Section E.4.1]{fournais2010spectral}), we get that $\varphi_m\in H^2(K)$ and
	\begin{equation}\label{eq:bd-H2-loc}
	\|\varphi_m\|_{H^2(K)}\leq C\left(\|\Delta\varphi_m\|_{L^2(\widetilde K)}+\|\varphi_m\|_{L^2(\widetilde K)}\right)\leq \widetilde C\,,
	\end{equation}
	where $\widetilde C$ is a constant independent from $m$. This proves that $(\varphi_m)_{m\geq 1}$ is bounded in $H^2_\mathrm{loc}(S_R)$.

	\paragraph{\itshape Step~2} 
	
	Here we will improve the result in Step~1  and prove that $(\varphi_m)_{m\geq 1}$ is bounded in $H^3_\mathrm{loc}(S_R)$. 
	It is enough to prove that the sequence $\big(\nabla \varphi_m\big)_{m\geq 1}$ is bounded in $H^2_\mathrm{loc}(S_R)$.
	 
	Let $\varsigma_m=\partial_{x_2} \varphi_m$. We will prove that $\big(\Delta\varsigma_m\big)_{m\geq 1}$ is bounded in  $L^2_\mathrm{loc}(S_R)$. Recall that, for all $x=(x_1,x_2) \in \R^2$,
	\[
	\Ab_0(x)=(-x_2,0)\quad\mathrm{and}\quad
	\sigma(x)=\mathbbm{1}_{\R_+}(x_2)+a\mathbbm{1}_{\R_-}(x_2)\,,
	\]
	hence,
	\begin{align}
	&\Big(\sigma\Ab_0\Big)(x)=\Big(-x_2\mathbbm{1}_{\R_+}(x_2)-ax_2\mathbbm{1}_{\R_-}(x_2),0\Big)\,,\label{eq:deriv1}\\
		&\Big(\sigma^2|\Ab_0|^2\Big)(x)=x_2^2\mathbbm{1}_{\R_+}(x_2)+a^2x_2^2\mathbbm{1}_{\R_-}(x_2)\,. \label{eq:deriv2}
	\end{align}
	Obviously, the functions in~\eqref{eq:deriv1} and~\eqref{eq:deriv2} admit respectively the following weak partial derivatives
	\begin{align}
	&\partial_{x_2}\Big(\sigma\Ab_0\Big)(x)=\Big(-\mathbbm{1}_{\R_+}(x_2)-a\mathbbm{1}_{\R_-}(x_2),0\Big)=\Big(-\sigma(x),0\Big)\, ,\label{eq:deriv3}\\
		&\partial_{x_2}\Big(\sigma^2|\Ab_0|^2\Big)(x)=2x_2\mathbbm{1}_{\R_+}(x_2)+2a^2x_2\mathbbm{1}_{\R_-}(x_2)=2x_2\sigma^2(x)\,. \label{eq:deriv4}
	\end{align}
	A straightforward computation using~\eqref{eq:delta_phi},~\eqref{eq:deriv3} and~\eqref{eq:deriv4} yields 
	\begin{multline*}
	\Delta\varsigma_m=\partial_{x_2}\Delta\varphi_m\\=\frac 1b \varphi_m^2\partial_{x_2}\overline{\varphi_m}+\frac 1b |\varphi_m|^2\partial_{x_2}\varphi_m-2i\sigma x_2\partial_{x_2}\partial_{x_1}\varphi_m-2i\sigma \partial_{x_1}\varphi_m+\sigma^2 x_2^2\partial_{x_2}\varphi_m+2\sigma^2x_2\varphi_m\,,\end{multline*}
	in the sense of weak derivatives. By Step~1, the sequence $(\varphi_m)$ is bounded in $H^2_\mathrm{loc}(S_R)$. Consequently, since $|\varphi_m|\leq 1$, it is clear that $(\Delta\varsigma_m)_{m\geq 1}$ is bounded in $L^2_\mathrm{loc}(S_R)$.   By the interior elliptic estimates, we get that $(\varsigma_m=\partial_{x_2}\varphi_m)_{m\geq 1}$ is bounded in $H^2_\mathrm{loc}(S_R)$.
	
	In a similar fashion, we prove that $(\partial_{x_1}\varphi_m)_{m\geq 1}$ is bounded in $H^2_\mathrm{loc}(S_R)$.

	\paragraph{\itshape Step~3}
	
	 Finally, for every relatively compact open set $K\subset \Omega$, the space  $H^3(K)$ is embedded  in $C^{1,\alpha}(\overline{K})$. Consequently,  $\big(\varphi_m\big)$ is bounded in $C^{1,\alpha}_\mathrm{loc}(S_R)$.
	\end{proof}
 \begin{lemma}\label{lem:phi_m_conv}
 	Assume that $R>0$ and that~\eqref{eq:A2} holds.  Let $\big(\varphi_{a,b,R,m}\big)_{m \geq1}$ be the sequence defined in Lemma~\ref{lem:Ln_mm}. There exists a function $\varphi_{a,b,R}\in H^3_\mathrm{loc}(S_R)$ and  a subsequence, denoted by  $\big(\varphi_{a,b,R,m}\big)_{m \geq1}$, such that
 	\[\varphi_{a,b,R,m}\longrightarrow \varphi_{a,b,R}\ \mathrm{in}\ H^2_\mathrm{loc}(S_R)\quad\mathrm{and}\quad
 	\varphi_{a,b,R,m}\longrightarrow \varphi_{a,b,R}\ \mathrm{in}\ C^{0,\alpha}_\mathrm{loc}(S_R)\quad\big(\alpha\in(0,1)\big)\,.\]
 	
 	Furthermore, for all $\alpha\in(0,1)$, $\varphi_{a,b,R}\in C_\mathrm{loc}^{1,\alpha}(S_R)$.
 	\end{lemma}
\begin{proof}
We continue writing $\varphi_m$ for $\varphi_{a,b,R,m}$ and $\varphi$ for $\varphi_{a,b,R}$. In the sequel, let $\alpha\in(0,1)$ be fixed.

Let $K\subset S_R$ be open and relatively compact.  By Lemma~\ref{lem:phim_bound}, $(\varphi_m)_{m\geq 1}$ is bounded in $H^3(K)$, hence it has a   weakly convergent subsequence by the Banach--Alaoglu theorem. By the compact embedding of $H^3(K)$ in $H^2(K)$, and of $H^2(K)$ in $C^{0,\alpha}(\overline{K})$, we may extract a subsequence, that we denote  by $(\varphi_m)$, such that it is strongly convergent in $H^2(K)$ and $C^{0,\alpha}(\overline{K})$. This subsequence and its limit $\varphi_K$ are independent of $\alpha$; we will prove that they   are actually independent of the relatively compact set $K$. This will be done by the standard  Cantor's diagonal process that we outline below.

For all $p\in\N$, set $K_p=(- R/2,R/2)\times(-p,p)$.  
	Let $I_0=\mathbb N$. The sequence $(\varphi_m)_{m\in I_0}$ has a subsequence  $(\varphi_m)_{m\in I_1}$ such that it is weakly convergent in $H^3(K_1)$, and strongly convergent in $H^2(K_1)$ and $C^{0,\alpha}(\overline{K_1})$. We denote the limit of this sequence by $\varphi_1$. Note that $\varphi_1\in H^3(K_1)$.	
By iteration, we obtain a collection of functions $(\varphi_p)_{p\in\N}$ and a collection of subsequences, $(\varphi_m)_{m\in I_p}$, such that
	\begin{itemize}
	\item $I_1\supset I_2\supset I_3\supset\cdots$\,.
	\item for every $p\in\mathbb N$, $(\varphi_m)_{m\in I_p}$ is a subsequence of $(\varphi_m)_{m\in I_{p-1}}$\,.
	\item for every $p\in\mathbb N$, the subsequence  $(\varphi_m)_{m\in I_p}$ converges weakly to $\varphi_p$ in $H^3( K_p)$\,.
	\item for every $p\in\mathbb N$, the subsequence  $(\varphi_m)_{m\in I_p}$ converges strongly to $\varphi_p$ in $H^2( K_p)$ and $C^{0,\alpha}(\overline{K_p})$\,.
	\end{itemize}
The Sobolev embedding  of $H^3(K_p)$ in $C^{1,\alpha}(\overline{K_p})$ yields that $\varphi_p\in C^{1,\alpha}(\overline{K_p})$. It is useful to note that 
\begin{equation}\label{eq:p_q}
\mathrm{If}\ p<q,\ \mathrm{then}\ \varphi_p=\varphi_q\ \mathrm{in}\ K_p.
\end{equation} 
Indeed, the strong convergence of  $(\varphi_m)_{m\in I_q}$ to $\varphi_q$ in $H^2( K_q)$ implies the following pointwise convergence of  $(\varphi_m)_{m\in I_q}$ in $K_q$ (along a subsequence)
\[\lim_{m \rightarrow +\infty}\varphi_m(x)=\varphi_q(x),\ \text{a.e. in}\ K_q\,.\]
But $K_p \subset K_q$, then
\begin{equation}\label{eq:K_p}
\lim_{m \rightarrow +\infty}\varphi_m(x)=\varphi_q(x),\ \text{a.e. in}\ K_p\,.
\end{equation}
Similarly,  the strong convergence of  $(\varphi_m)_{m\in I_p}$ to $\varphi_p$ in $H^2( K_p)$ implies
\[\lim_{m \rightarrow +\infty}\varphi_m(x)=\varphi_p(x),\ \text{a.e. in}\ K_p\,.\]
Since $I_q \subset I_p$, we get the following pointwise convergence of  $(\varphi_m)_{m\in I_q}$ in $K_p$
\begin{equation}\label{eq:K_q}
\lim_{m \rightarrow +\infty}\varphi_m(x)=\varphi_p(x),\ \text{a.e. in}\ K_p\,.
\end{equation}
Having in hand the continuity of $\varphi_p$ and $\varphi_q$,~\eqref{eq:p_q} follows from~\eqref{eq:K_p} and~\eqref{eq:K_q}.

		Now, we are ready to define the limit function $\varphi$ in $S_R=(-R/2, R/2)\times (-\infty,+\infty)$ as follows. Let $x\in S_R$. There exists $p\in \mathbb N$ such that $ x\in K_p$. We then define $\varphi(x)=\varphi_p(x)$. The function $\varphi$ is well defined by~\eqref{eq:p_q} and belongs to $H^3_\mathrm{loc}(S_R)$, consequently to $C_\mathrm{loc}^{1,\alpha}(S_R)$.
		Next, we will construct a subsequence $(\varphi_m)_{m \in I}$ of $(\varphi_m)_{m \in I_0}$ (with $I \subset I_0$) that converges weakly to the function $\varphi$ in\  $H^3(K_p)$, for all $p\in\N$.
		For all $p\geq 1$, the set $I_p\subset \N$ consists of a strictly  increasing sequence $\{ n_1(p),n_2(p),...\}$; let $n_p$ be the  $p^{th}$ element of $I_p$, i.e.  $n_p=n_p(p)$. By induction, we can prove that, for all $p,k\in\mathbb N$ (with $k\geq 2$), 
		$n_k(p+1)>n_{k-1}(p+1)\geq n_{k-1}(p)$. Thus, for all $p\in\mathbb N$, $n_{p+1}:=n_{p+1}(p+1)>n_{p}(p)=n_{p}$. 
		We define the index set $I=\{n_1,n_2,...\}$ and note that $(\varphi_m)_{m\in I}$ is a subsequence of $(\varphi_m)_{m\geq 1}$, because $n_1<n_2<...$. Also, it is a subsequence  of $(\varphi_m)_{m\in I_p}$, for every $p \in \N$.
		Consequently, for all $p\in\mathbb N$, the following strong convergence holds
		\begin{equation}\label{eq:cantor*}
		\varphi_m\underset{\substack{m\to+\infty\\ m\in I}}{\longrightarrow } \varphi\  \text{in}\ H^2(K_p)~\mathrm{and}~C^{0,\alpha}(\overline{K_p})\,.
		\end{equation}
Finally, if $K\subset S_R$ is an arbitrary open and relatively compact set, then there exists $p\in\N$ such that $K\subset K_p$. Consequently, we inherit from~\eqref{eq:cantor*} that $(\varphi_m)_{m\in I}$ converges to $\varphi$ in $H^2(K)$ and $C^{0,\alpha}(\overline{K})$.
\end{proof}
\begin{lemma}\label{lem:phi_prop}
		Assume that $R>0$ and~\eqref{eq:A2} holds. Let $\varphi_{a,b,R}$ be the function defined by Lemma~\ref{lem:phi_m_conv}. The following statements hold:
		 \[\varphi_{a,b,R} \in \mathcal D_R\,,\]
			 \begin{equation}\label{eq:phi_bound}
		     |\varphi_{a,b,R}| \leq 1 \qquad \mathrm{in}\ S_{R}\,,
			\end{equation}
			 \begin{equation}\label{eq:var1}
			-b(\nb-i\sigma \Ab_0)^2\varphi_{a,b,R}=(1-|\varphi_{a,b,R}|^2)\varphi_{a,b,R}\qquad \mathrm{in}\ S_{R}\,,
			\end{equation}
			 \begin{equation}\label{eq:Ln3}
			\int_{S_R\cap\{|x_2|\geq 4\}} \frac{|x_2|}{\big(\ln|x_2|\big)^2} \Big(\big|(\nabla-i\sigma\Ab_0)\varphi_{a,b,R} \big|^2+|\varphi_{a,b,R}|^2 \Big)\,dx\leq CbR\,,
			\end{equation}
			 \begin{equation}\label{eq:Ln4}
			\int_{S_R\cap\{|x_2|\geq 4\}} \frac{|x_2|^3}{\big(\ln|x_2|\big)^2}|\varphi_{a,b,R}|^4\,dx\leq Cb^2R\,,
			\end{equation}
			 \begin{equation}\label{eq:Ln5}
			\int_{S_R} \Big(b\big|(\nabla-i\sigma\Ab_0)\varphi_{a,b,R}\big|^2+|\varphi_{a,b,R}|^2 \Big)\,dx\leq CbR\,,
			\end{equation}
			where $C>0$ is a universal constant  and $\mathcal D_R$ is the space introduced in~\eqref{eq:DR}.
\end{lemma}
\begin{proof}
Let $(\varphi_{a,b,R,m})$ be the subsequence in Lemma~\ref{lem:phi_m_conv}. Again, we will use $(\varphi_m)$ and $\varphi$ for  $(\varphi_{a,b,R,m})$ and $\varphi_{a,b,R}$ respectively.

By Lemma~\ref{lem:Ln_mm}, the inequality $|\varphi_m|\leq 1$ holds for all $m$.   The inequality $|\varphi|\leq 1$ then follows from the uniform convergence of $(\varphi_m)$ stated in Lemma~\ref{lem:phi_m_conv}.	
By the convergence of  $\big(\varphi_m\big)$  in $H^2_\mathrm{loc}(S_R)$ and $C^{0,\alpha}_\mathrm{loc}(S_R)$, we get~\eqref{eq:var1}
 from
 \[-b(\nb-i\sigma \Ab_0)^2\varphi_m=(1-|\varphi_m|^2)\varphi_m\,.\]
Now we prove that $\varphi \in \mathcal D_R$. Pick an arbitrary integer $m_0 \geq 1$. For all $m \geq m_0$, $S_{R,m_0} \subset S_{R,m}$. Thus using the decay of $\varphi_m$ in~\eqref{eq:phi_m_decay} we have
\begin{align*}
\int_{S_{R,m_0}}|\varphi_m|^2\,dx &\leq \int_{S_{R,m}}|\varphi_m|^2\,dx\\ 
                                  &\leq CbR \,.
\end{align*}
	The uniform convergence of $\big(\varphi_m\big)$ to $\varphi$ gives us
	\begin{align*}
	\int_{S_{R,m_0}}|\varphi|^2\,dx &=\lim_{m \rightarrow +\infty}	\int_{S_{R,m_0}}|\varphi_m|^2\,dx\\
	&\leq CbR\,.
	\end{align*}
Taking $m_0\to+\infty$, we write by  the monotone convergence theorem, 
\[	\int_{S_R}|\varphi|^2\,dx \leq CbR\,.\]
This proves that $\varphi\in L^2(S_R)$. Next we will prove that $(\nabla-i\sigma\Ab_0)\varphi\in L^2(S_R)$. In light of the convergence of $(\varphi_m)$ in $H^1_\mathrm{loc}(S_R)$, we can refine the subsequence $(\varphi_m)$ so that
\[(\nabla-i\sigma\Ab_0)\varphi_m\to (\nabla-i\sigma\Ab_0)\varphi~\mathrm{a.e.}\] Furthermore, by Lemma~\ref{lem:phim_bound},  $\big(\varphi_m\big)$ is bounded in $C^1_\mathrm{loc}(S_R)$, hence in $C^1(S_{R,m_0})$, for all $m_0\geq 1$. Using the  dominated convergence theorem and the estimate  in~\eqref{eq:phi_m_decay}, we may write, for all $m_0\geq 1$, 
\begin{align*}
\int_{S_{R,m_0}}\big|(\nabla-i\sigma\Ab_0)\varphi\big|^2\,dx &=\lim_{m \rightarrow +\infty} \int_{S_{R,m_0}}\big|(\nabla-i\sigma\Ab_0)\varphi_m\big|^2\,dx\\
&\leq C R \,.
\end{align*}
Sending $m_0$ to $+\infty$ and using the monotone convergence theorem, we get 
\[\int_{S_R}|(\nabla-i\sigma\Ab_0)\varphi|^2\,dx \leq CR \,.\]
Thus, we have proved that $\varphi,(\nabla-i\sigma\Ab_0)\varphi\in L^2(S_R)$. It remains to prove that $\varphi$ satisfies the boundary condition
\[\varphi\left(x_1=\pm \frac R2,x_2\right)=0, \qquad \text{for all}\ x_2 \in \R\,.\]
To see this, let $x_2 \in \R$. There exists $m_0$ such that $x_2 \in (-m_0,m_0)$. By the convergence of $(\varphi_m)$ to $\varphi$ in $\ C^{0,\alpha}(\overline{S_{R,m_0}})$, we get 
	\[  \varphi\left(x_1=\pm \frac R2,x_2\right)=\lim_{m \rightarrow +\infty} \varphi_m\left(x_1=\pm \frac R2,x_2\right)=0\,.\]
Finally, we may use  similar limiting arguments to pass from the decay estimates of $\varphi_m$ in~\eqref{eq:Ln_m} and~\eqref{eq:Ln_m1}  to the decay estimates of $\varphi$ in~\eqref{eq:Ln3} and~\eqref{eq:Ln4}. 
\end{proof}

Now, we are ready to establish the existence of a minimizer of the Ginzburg--Landau energy $\mathcal G(a,b,R)$ defined in the unbounded set $S_R$.

\begin{lemma}\label{lem:min_ex}
Assume that~\eqref{eq:A2} holds. For all $R>0$, the function $\varphi_{a,b,R}\in\mathcal D_R$ defined in Lemma~\ref{lem:phi_m_conv} is a minimizer of $\mathcal G_{a,b,R}$, that is	
	\begin{equation*}
	\mathcal G_{a,b,R}(\varphi_{a,b,R})=\mathfrak g_a (b,R).
	\end{equation*}
	Here  $\mathcal G_{a,b,R}$ is the functional introduced in~\eqref{eq:Gb} and $\mathfrak g_a (b,R)$ is the ground state energy defined in~\eqref{eq:m0-mN}. 
\end{lemma}	
\begin{proof}
The proof is divided into three steps.

\paragraph{\itshape Step~1. (Convergence of the ground state energy)}

Let  $\mathfrak g_a(b,R,m)$ and $\mathfrak g_a (b,R)$ be the energies defined in~\eqref{eq:m0-mN} and~\eqref{eq:gr_m} respectively. In this step, we will prove that 
\begin{equation}\label{eq:Step1}
\lim_{m\to+\infty}\mathfrak g_a(b,R,m)=\mathfrak g_a (b,R)\,.
\end{equation}
Let $u \in \mathcal D_{R,m}$. We can extend $u$ by $0$ to a function $\tilde{u} \in \mathcal D_R$. As an immediate consequence, we get $\mathfrak g_a(b,R,m)\geq \mathfrak g_a (b,R)$, for all $m \in \N$. Thus, 
\begin{equation}\label{eq:inf*}
\liminf_{m\rightarrow +\infty} \mathfrak g_a(b,R,m)\geq \mathfrak g_a (b,R)\,.
\end{equation}
Next, we will  prove that
\begin{equation}\label{eq:sup}
\limsup_{m\rightarrow +\infty} \mathfrak g_a(b,R,m)\leq \mathfrak g_a (b,R)\,.
\end{equation}
Consider $(\varphi_n)\subset \mathcal D_R$ a minimizing sequence of $\mathcal G_{a,b,R}$, that is
\[\mathfrak g_a (b,R)=\lim_{n \rightarrow +\infty} \mathcal G_{a,b,R}(\varphi_n)\]
Let $\vartheta \in C_c^\infty(\R)$ be a cut-off function satisfying
\[0\leq\vartheta\leq 1\ \mathrm{in}\ \R,\quad \supp\vartheta\subset(-1,1),\quad \vartheta=1\ \mathrm{in}\ \left[-\frac12,\frac12\right]\,.\]
Consider the re-scaled function $\vartheta_m(x_2)=\vartheta(x_2/m)$. The function $\vartheta_m(x_2)\varphi_n(x)$ restricted to $S_{R,m}$ belongs to $\mathcal D_{R,m}$ and consequently
\begin{equation}\label{eq:rm}
\mathfrak g_a(b,R,m)\leq \mathcal G_{a,b,R}(\vartheta_m \varphi_n)\,.
\end{equation}
By Cauchy's inequality, for all $\epsilon \in (0,1)$ 
\[\big|(\nb-i\sigma \Ab_0)\vartheta_m \varphi_n\big|^2\leq (1+\epsilon)\big| \vartheta_m (\nb-i\sigma \Ab_0)\varphi_n\big|^2+2\epsilon^{-1}|\nabla \vartheta_m|^2|\varphi_n|^2\,.\]
Thus, using the definition of the ground state energy $\mathfrak g_a(b,R,m)$ and the functional $\mathcal G_{a,b,R}$ in~\eqref{eq:gr_m} and \eqref{eq:Gb} respectively,  we obtain
\begin{equation}\label{eq:rm1}
\mathfrak g_a(b,R,m)\leq (1+\epsilon)\mathcal G_{a,b,R}(\varphi_n)+\frac{2b\epsilon^{-1}}{m^2}\|\vartheta'\|^2_{L^\infty(\R)}\int_{S_R}|\varphi_n|^2\,dx +\int_{S_R} (1-\vartheta_m^2+\epsilon)|\varphi_n|^2\,dx\,.
\end{equation}
Introducing $\displaystyle\limsup_{m\rightarrow +\infty}$ on both sides of~\eqref{eq:rm1}, and using the dominated convergence theorem, we get
\[\displaystyle\limsup_{m\rightarrow +\infty}\mathfrak g_a(b,R,m)\leq (1+\epsilon)\mathcal G_{a,b,R}(\varphi_n)+\epsilon \int_{S_R}|\varphi_n|^2\,dx\]
Taking the successive limits  $\epsilon\rightarrow 0_+$ then  $n \rightarrow +\infty$, we get~\eqref{eq:sup}. Combining~\eqref{eq:inf*} and~\eqref{eq:sup}, we get~\eqref{eq:Step1}.

\paragraph{\itshape Step~2. (The $L^4$-norm of the limit function)}

Let $(\varphi_m=\varphi_{a,b,R,m})$ be the sequence in Lemma~\ref{lem:phi_m_conv} which converges to the function $\varphi=\varphi_{a,b,R}$. We would like to verify that  the limit  function $\varphi$  is  a minimizer of the functional $\mathcal G_{a,b,R}$.
To that end, we will prove first that
\begin{equation}\label{eq:phi_4}
\lim_{m\rightarrow +\infty}\int_{S_{R,m}}|\varphi_m|^4\,dx=\int_{S_R}|\varphi|^4\,dx\,.
\end{equation}
We begin by proving that
\begin{equation}\label{eq:liminf}
\liminf_{m\rightarrow +\infty}\int_{S_{R,m}}|\varphi_m|^4\,dx\geq\int_{S_R}|\varphi|^4\,dx\,.
\end{equation}
Pick a fixed integer  $m_0\geq 1$. Since $S_{R,m}\supset S_{R,m_0}$ for all $m\geq m_0$, the following  inequality holds 
\begin{equation}\label{eq:cv1}
\int_{S_{R,m}}|\varphi_m|^4\,dx\geq\int_{S_{R,m_0}}|\varphi_m|^4\,dx\,.
\end{equation}
In addition, having in hand the uniform convergence of $\varphi_m$ to $\varphi$ on the compact set $S_{R,m_0}$, we get as $m\rightarrow \infty$
\begin{equation}\label{eq:cv2}
\int_{S_{R,m_0}}|\varphi_m|^4\,dx\rightarrow\int_{S_{R,m_0}}|\varphi|^4\,dx\,. 
\end{equation}
We introduce $\liminf_{m\rightarrow +\infty}$ on both sides of~\eqref{eq:cv1}, and we use~\eqref{eq:cv2} to get
\[\liminf_{m\rightarrow +\infty}\int_{S_{R,m}}|\varphi_m|^4\,dx\geq\int_{S_{R,m_0}}|\varphi|^4\,dx\,.\]
This is true for every integer $m_0 \geq 1$. Consequently~\eqref{eq:liminf} simply follows by applying the monotone convergence theorem.

Next, we prove that
\begin{equation}\label{eq:lim_sup}
\limsup_{m\rightarrow +\infty}\int_{S_{R,m}}|\varphi_m|^4\,dx\leq\int_{S_R}|\varphi|^4\,dx\,.
\end{equation}
Let $C$ be the universal constant in~\eqref{eq:Ln_m1}, $\epsilon>0$ be fixed, and $R>0$ be arbitrary. We select an integer $m_0 \geq 1$ such that
\begin{equation}\label{eq:m_0}
\frac {Cb^2R}{m_0}<\epsilon\,.
\end{equation}
In light of~\eqref{eq:cv2}, there exists 
$m_1\geq m_0$ such that
\[\forall~m\geq m_1,\quad \left|\int_{S_{R,m_0}}|\varphi_m|^4\,dx-\int_{S_{R,m_0}}|\varphi|^4\,dx\right|\leq \epsilon\,.\]
Noticing that
\[\int_{S_{R,m_0}}|\varphi|^4\,dx\leq\int_{S_R}|\varphi|^4\,dx\,,\]
we may write, for all $m\geq m_1$
\begin{equation}\label{eq:cv3}
\int_{S_{R,m_0}}|\varphi_m|^4\,dx\leq\int_{S_R}|\varphi|^4\,dx+\epsilon\,.
\end{equation} 
On the other hand,  for $|x_2|\geq m_0\geq 1$ we have,
\[m_0\leq \frac{|x_2|^3}{\big(\ln|x_2|\big)^2}\,.\]
Thus, the estimate in~\eqref{eq:Ln_m1} yields for all $m\geq m_0$,  
\begin{equation}\label{eq:cv4}
\int_{S_{R,m}\cap \{|x_2|\geq m_0 \}}|\varphi_m|^4\,dx\leq \underset{\mathrm{by~}\eqref{eq:m_0}}{\underbrace{\frac {Cb^2R}{m_0}<\epsilon}}\,.
\end{equation}
Combining~\eqref{eq:cv3} and~\eqref{eq:cv4}, we get for all $m\geq m_1\geq m_0$
\begin{align*}
\int_{S_{R,m}}|\varphi_m|^4\,dx&=\int_{S_{R,m_0}}|\varphi_m|^4\,dx+\int_{S_{R,m}\cap \{|x_2|\geq m_0 \}}|\varphi_m|^4\,dx\,,\\
&\leq\int_{S_R}|\varphi|^4\,dx+2\epsilon\,.
\end{align*}
Taking the successive limits $m\to+\infty$ then $\epsilon\to0_+$, we get~\eqref{eq:lim_sup}.

\paragraph{\itshape Step~3. (The limit function is a minimizer)}

The convergence in~\eqref{eq:phi_4} is crucial in establishing that $\varphi$ is a minimizer of $\mathcal G_{a,b,R}$. In light of 
Eq.~\eqref{eq:var_m},  an integration by parts yields, for all $m\geq 1$,
\[\mathfrak g_a(b,R,m)=-\frac12 \int_{S_{R,m}} |\varphi_m|^4\,dx\,.\]
We take $m \rightarrow +\infty$, and we use the results in~\eqref{eq:Step1} and~\eqref{eq:phi_4}. We get
\begin{equation}\label{eq:e1}
	\mathfrak g_a (b,R)=-\frac 12 \int_{S_R} |\varphi|^4\,dx\,.
\end{equation}
By  Lemma~\ref{lem:phi_prop},  $\varphi\in\mathcal D_R$ and satisfies~\eqref{eq:var1}, so after  integrating  by parts, we get
\begin{equation}\label{eq:e2}
	\mathcal G_{a,b,R}(\varphi)=-\frac 12 \int_{S_R} |\varphi|^4\,dx\,.
\end{equation}
Comparing~\eqref{eq:e1} and~\eqref{eq:e2} yields that  $\mathcal G_{a,b,R}(\varphi) = \mathfrak g_a (b,R)$.
\end{proof}
\begin{proof}[Proof of Proposition~\ref{prop:Ln}]
	This proposition is simply a convenient collection in one place of already proved facts in Lemma~\ref{lem:phi_prop} and Lemma~\ref{lem:min_ex}.
	\end{proof}

\subsection{The limit energy}

In this section, we will prove the existence of the limit energy $\mathfrak e_a(b)$, defined as the limit of $\mathfrak g_a (b,R)/R$ as $R\rightarrow+\infty$. After that,  we will study,  when the parameter  $a$ is fixed, some properties of the function $b\mapsto \mathfrak e_a(b)$.

{\bfseries In the sequel, we assume that $a,b,R$ are constants such that $R\geq 1$ and~\eqref{eq:A2} holds.}

The next lemma displays some simple, yet very important, translation invariance property of the energy. This property is mainly needed in Theorem~\ref{thm:eba} to establish an  upper bound of the limit energy $\mathfrak e_a(b)$.
\begin{lemma}\label{lem:periodic}
	Let $n \in \N$. Consider  the ground state energy  $\mathfrak g_a (b,R)$ defined in~\eqref{eq:m0-mN}. It holds
	\[\mathfrak g_a(b,nR)\leq n\mathfrak g_a(b,R)\,.\]
\end{lemma}
\begin{proof}
	Let $S_R=\left(-R/2,R/2\right) \times \R$ be the strip defined in~\eqref{eq:SR}. Consider the strip $S_{R,\lambda}$ defined such that	
	\[S_{R,\lambda}=S_R+\lambda R,\qquad \lambda \in \R\,.\]
	Let $u \in \mathcal D_R$, where $\mathcal D_R$ is the domain defined in~\eqref{eq:DR}. We define the function $v$ in $S_{R,\lambda}$ as follows
	\begin{equation}\label{eq:trans}
	v(x_1,x_2)=u(x_1-\lambda R,x_2),\qquad (x_1,x_2)\in S_{R,\lambda}\,.
	\end{equation}
	An easy computation shows the invariance of the energy under the aforementioned translation, that is
	\begin{equation}\label{eq:grad_trans}
\int_{S_R} b\big|(\nb-i\sigma \Ab_0)u\big|^2\,dy=\int_{S_{R,\lambda}} b\big|(\nb-i\sigma \Ab_0)v\big|^2\,dx\,,
	\end{equation}
	and
	\begin{multline}\label{eq:En_trans}
\mathcal G_{a,b,R}(u)=\int_{S_R} \left(b\big|(\nb-i\sigma \Ab_0)u\big|^2-|u|^2+\frac 12 |u|^4\right)\,dy\\=\int_{S_{R,\lambda}} \left(b\big|(\nb-i\sigma \Ab_0)v\big|^2-|v|^2+\frac 12 |v|^4\right)\,dx\,.
	\end{multline}
	Now, let $n \in \N$. Noticing that 
	\[\overline{S_{nR}}=\left[-n\frac R2,n\frac R2\right] \times \R=\bigcup_{j \in  J}\overline{S_{R,j}}\,,\]
	where $J=\left\{(1-n)/2+k,\ 0\leq k\leq n-1\right\}$, we define a function $\tilde{u}$ in $S_{nR}$ as follows
	\[\tilde{u}(x_1,x_2)=u(x_1-jR,x_2),\quad \mathrm{if}\ (x_1,x_2)\in S_{R,j}\,.\]
	This definition is consistent, since  the sets $\big(S_{R,j}\big)$ are disjoint, their closures cover $S_{nR}$, and $u \in \mathcal D_R$ which yields that $\tilde u$ vanishes on the boundary of every $S_{R,j}$. Having~\eqref{eq:En_trans}, we get consequently
	\[\tilde{u} \in \mathcal D_{nR}\quad\mathrm{and}\quad\mathcal G_{a,b,nR}(\tilde{u})=n\mathcal G_{a,b,R}(u)\,,\]
	This yields that
	\[\mathfrak g_a(b,nR) \leq n \mathcal G_{a,b,R}(u)\,.\]
	We choose $u \in \mathcal D_R$ to be the minimizer $\varphi_{a,b,R}$ of $\mathcal G_{a,b,R}$, defined in Proposition~\ref{prop:Ln} and conclude that
	\[\mathfrak g_a(b,nR)\leq n\mathfrak g_a(b,R)\,.\]
\end{proof}
Our next result is concerned with the monotonicity of the function $R\mapsto\mathfrak g_a (b,R)$.
\begin{lemma}\label{lem:m_var}
The function $R\mapsto \mathfrak g_a (b,R)$ defined in~\eqref{eq:m0-mN} is monotone non-increasing. 
\end{lemma}
\begin{proof}
	This follows from the domain monotonicity. Indeed, let $r>0$ and $u \in H_0^1(S_R)$ be a minimizer of $\mathcal G_{a,b,R}$. Consider the function $\tilde{u} \in H_0^1(S_{R+r})$ defined as the extension of $u$ by zero on $S_{R+r}\setminus S_{R}$. Obviously
	\[\mathfrak g_a (b,R)=\mathcal G_{a,b,R}(u)= \mathcal G_{a,b,R+r}(\tilde{u})\geq \mathfrak g_a(b,R+r).\]
\end{proof}
The existence of the limit of $\mathfrak g_a (b,R)/R$ as $R\rightarrow+\infty$ will follow from a well known abstract result, see Lemma~\ref{lem:df} below. To apply this abstract result,  we need some bounds on the energy $\mathfrak g_a (b,R)$. These are  given in  Lemma~\ref{lem:m0_bound} below.
\begin{lemma} \label{lem:m0_bound}
Let $\mathfrak g_a (b,R)$ be the ground state energy in~\eqref{eq:m0-mN}. There exist positive constants $C_1$, $C_2$, and $C_3$ dependent solely on $a$ and $b$ such that
\begin{equation}\label{eq:m0_bound}
-C_1R\leq \frac{\mathfrak g_a (b,R)}{1-b\beta_a}\leq -C_2R+\frac{C_3}{R}\,.
\end{equation}
\end{lemma}
\begin{proof}~

{\bfseries Upper bound.}  Let $\theta \in C_c^\infty(\R)$ be a function satisfying
\[\supp \theta \subset \Big(-\frac 12, \frac 12\Big), \quad 0\leq\theta\leq1, \quad \theta=1\ \mathrm{in}\ \Big(-\frac 14, \frac14\Big)\,,\]
and let  
\[\theta_R(x)=\theta(x/R)\,.\]
We define the function
\[v(x)=\theta_R(x_1)\psi_a(x)\,,\]
where $\psi_a$ is the eigenfunction introduced in~\eqref{eq:psi_0}. Recall that $\psi_a(x_1,x_2)=e^{i\zeta_ax_1}\phi_a(x_2)$ with $\zeta_a$ and $\phi_a$ defined in~\eqref{eq:phi}, and $\phi_a$ satisfies $\|\phi_a\|_{L^2(\R)}=1$.

The function $\psi_a$ satisfies $-(\nabla-i\sigma\Ab_0)^2\psi_a=\beta_a\psi_a$ in $\R^2$.
Multiplying this equation by $\overline{v}=\theta^2_R\overline{\psi_a}$ then  integrating on the support of the function $v$, we get
\begin{align*}
\int_{S_R}|(\nabla-i\sigma\Ab_0)v|^2\,dx &= \beta_a \int_{S_R} \theta_R^2(x_1)|\psi_a(x)|^2\,dx+\int_{S_R}\theta_R'^2(x_1)  
\phi_a^2(x_2)\,dx \\
         												 &\leq \beta_a \int_{S_R} \theta_R^2(x_1)|\psi_a(x)|^2\,dx+ \frac{C}{R}\,.
\end{align*} 
Consequently, we get for all $t>0$
\begin{align*}
\mathfrak g_a (b,R) &\leq \mathcal G_{a,b,R}(tv) \\
           &\leq t^2(b\beta_a-1)\int_{S_R} \theta_R^2(x_1)\phi_a^2(x_2)\,dx + C b \frac{t^2}R +\frac{t^4}2 \int_{S_R} \theta_R^4(x_1)\phi_a^4(x_2)\,dx  \\
					& \leq t^2(b\beta_a-1)R + C b\frac{t^2}R +\frac{t^4}2 R \int_\R \phi_a^4(x_2)\,dx_2  \\
					& = t^2R \left((b\beta_a-1) + \frac{t^2}2 \int_\R \phi_a^4(x_2)\,dx_2 \right) + C b\frac{t^2}R\,.
\end{align*} 
We select $t$ such that  
\[(b\beta_a-1) + \frac{t^2}2 \int_\R \phi_a^4(x_2)\,dx_2 =\frac 12 (b\beta_a-1)<0~\mathrm{by~}\eqref{eq:A2}\,.\]
Let $\nu_a=\int_\R \phi_a^4(x_2)\,dx_2$. Thus, for $t=\sqrt{(1-b\beta_a)/\nu_a}$ we get 
\[\frac{\mathfrak g_a (b,R)}{1-b\beta_a}\leq -C_2R+\frac{C_3}{R}\,,\]
where $C_2=(1/2) t^2$ and $C_3=C  b/\nu_a$ depend only on $a$ and $b$.

{\bfseries Lower bound.} Let $\varphi_{a,b,R}$ be the minimizer in Proposition~\ref{prop:Ln}. For simplicity, we will write  $\varphi=\varphi_{a,b,R}$ . It follows from the   min-max principle that
\[\mathfrak g_a (b,R) = \mathcal G_{a,b,R}(\varphi) \geq (b\beta_a-1) \int_{S_R}|\varphi|^2\,dx\,.\]
By~\eqref{eq:A2}, $b\beta_a-1<0$. By~\eqref{eq:phi_decay}, 
$\displaystyle\int_{S_R} |\varphi|^2\,dx\leq CbR$, where $C>0$ is a universal constant. We choose $C_1=C/\beta_a$ and get the following inequality
\[\mathfrak g_a (b,R)\geq Cb(b\beta_a-1)R\geq -C_1(1-b\beta_a) R\,.\]
Obviously, $C_1$ depends solely on $a$.
\end{proof}
The next abstract lemma is a key-ingredient in the proof of Theorem~\ref{thm:eba}, and more precisely in establishing the existence of the limit energy $\mathfrak e_a(b)$ introduced in~\eqref{eq:eba1}. Variants of it were used in many papers, see~\cite{Fournais, fournais2013ground, pan2002surface, sandier2003decrease}. Here we use the version  from~\cite[Lemma~2.2]{Fournais}. 
\begin{lemma}\label{lem:df}
Let $\delta>0$. Consider a monotone non-increasing function $d:(\delta,+\infty)\rightarrow (-\infty,0]$ such that the function $f: (\delta,+\infty) \ni l \mapsto d(l)/l^2 \in \R$ is bounded.

Suppose that there exists a constant $C>0$ such that the estimate
\[f(nl)\geq f\big((1+\alpha)l\big)-C\left(\alpha+\frac1{\alpha^2l^2}\right)\]
holds true for all $\alpha \in (0,1)$, $n \in \N$, and $\ell \geq \ell_0$.
Then $f(l)$ has a limit $A$ as $l\rightarrow +\infty$.

Furthermore, for all $l \geq 2l_0$, the following estimate holds
\[f(l) \leq A+\frac{2C}{l^{\frac 23}}\,.\]
\end{lemma}
We will apply Lemma~\ref{lem:df} on the function $f:R\mapsto \mathfrak g_a (b,R)/R $ in order to define $\mathfrak e_a(b)$ as 
$\lim_{R\to+\infty}\mathfrak g_a (b,R)/R$. To that end, we establish that the above choice of $f$ fulfils the conditions in Lemma~\ref{lem:df}. This is the content of

\begin{lemma}\label{lem:limit}
There exists a universal constant $C>0$ such that, for all  $n \in \N$ and $\alpha \in (0,1)$, the ground state energy $\mathfrak g_a (b,R)$ defined in~\eqref{eq:m0-mN} satisfies
\begin{equation}\label{eq:m0n}
\frac{\mathfrak g_a(b,n^2R)}{n^2R} 
							     \geq \frac {\mathfrak g_a\big(b,(1+\alpha)^2R\big)}{(1+\alpha)^2R}-Cb^2\Big(\alpha +\frac{1}{\alpha^2 R}\Big)\,. 
\end{equation}
\end{lemma}
\begin{proof}

Let $n\geq1$ be a natural number, $\alpha \in (0,1)$ and consider the family of strips
\[S_j=\left(-n^2-1-\alpha+(2j-1)\Big(1+\frac \alpha2\Big),-n^2-1+(2j+1)\Big(1+\frac \alpha2\Big)\right)\times \R,\quad (j \in \mathbb Z)\]
Notice that the width of $S_j$ is $2(1+\alpha)$, and the width of the overlapping region between two strips, when it exists, is $\alpha$. We consider the partition of unity of $\R^2$ :
\[\sum_j |\chi_j|^2=1,\quad 0 \leq \chi_j \leq 1,\quad \sum_j|\nabla \chi_j|^2\leq \frac C{\alpha^2},\quad \supp\chi_j\subset S_j\,,\]
where $C$  is a universal constant. Define 
\[\chi_{R,j}(x)=\chi_j(2x/R)\]
$(\chi_{R,j})$ is then a new partition of unity satisfying
\begin{equation}\label{eq:chij}
\sum_j |\chi_{R,j}|^2=1,\quad 0 \leq \chi_{R,j} \leq 1,\quad \sum_j|\nabla \chi_{R,j}|^2\leq \frac C{\alpha^2 R^2},\quad \supp\chi_{R,j}\subset S_{R,j}\,,
\end{equation}
where $S_{R,j}=\{xR/2~:~x \in S_j\}$.
The family of strips $(S_{R,j})_{j \in \{1,2,...,n^2\}}$ yields a covering of $S_{n^2R}=\left(-n^2R/2,n^2 R/2\right)\times \R$ by $n^2$ strips, each of width $(1+\alpha)R$.
Let $\varphi_{a,b,n^2R} \in \mathcal D_{n^2R}$
be the minimizer in Proposition~\ref{prop:Ln}. We decompose the energy associated to $\varphi_{a,b,n^2R}$ as follows
\begin{align*}
\mathfrak g_a(b,n^2R) &=\mathcal G_{a,b,n^2R}(\varphi_{a,b,n^2R})\\
           &\geq \sum_{j=1}^{n^2}\Big(\mathcal G_{a,b,n^2R}(\chi_{R,j}\varphi_{a,b,n^2R})-b \big\| |\nabla \chi_{R,j}|\varphi_{a,b,n^2R}\big\|^2_{L^2\big(S_{n^2R}\big)}\Big) \\
					 &= \sum_{j=1}^{n^2}\mathcal G_{a,b,n^2R}(\chi_{R,j}\varphi_{a,b,n^2R})-b \int_{S_{n^2R}}\Big(\sum_{j=1}^{n^2} |\nabla \chi_{R,j}|^2\Big)|\varphi_{a,b,n^2R}|^2\,dx \\
					 &\geq \sum_{j=1}^{n^2}\mathcal G_{a,b,n^2R}(\chi_{R,j}\varphi_{a,b,n^2R})-C \frac{b^2n^2}{\alpha^2 R}\,.  
\end{align*}
The first inequality above follows from the celebrated \emph{IMS localization formula} (see~\cite[Theorem~3.2]{cycon2009schrodinger}), while the second comes from~\eqref{eq:phi_decay} and the properties of $(\chi_{R,j})$ in~\eqref{eq:chij}.
Notice that $\chi_{R,j}\varphi_{a,b,n^2R}$ is supported in an infinite strip of width $(1+\alpha)R$. By energy translation invariance along the $x_1$-direction (see~\eqref{eq:En_trans}), we have 
\[\mathcal G_{a,b,n^2R}(\chi_{R,j}\varphi_{a,b,n^2R})\geq \mathfrak g_a(b,(1+\alpha)R)\,.\]
As a consequence,
\[\mathfrak g_a(b,n^2R)\geq n^2 \mathfrak g_a(b,(1+\alpha)R)-C \frac {b^2n^2}{\alpha^2R}\,.\]
For $R \geq 1$, dividing both sides by $n^2R$ and using the monotonicity of $R \mapsto\mathfrak g_a (b,R)$, we get
\begin{align*}
\frac{\mathfrak g_a(b,n^2R)}{n^2R} &\geq \frac {\mathfrak g_a \big(b,(1+\alpha)R\big)}{R} - \frac{Cb^2}{\alpha^2 R^2}  \\
													  &\geq \frac {\mathfrak g_a \big(b,(1+\alpha)^2R\big)}{(1+\alpha)^2R}-Cb^2\Big(\alpha +\frac{1}{\alpha^2 R}\Big)\,.
\end{align*}
\end{proof}

\subsection{Proof of Theorem~\ref{thm:eba}}

Here we will verify all the statements appearing in Theorem~\ref{thm:eba}. Noticing that $\mathcal G_{a,b,R}(0)=0$, we get Item~(1). The second item is already proved in Lemma~\ref{lem:trivial}.
Defining $\mathfrak e_a(b)=0$ for $b\geq 1/\beta_a$, the items~(3) and~(5) hold trivially since $\mathfrak g_a(b,R)=0$ in this case. For these two items, we handle now the case where $1/|a| \leq b<1/\beta_a$.

In the proof of Item~(3), we define the two functions $d_{a,b}(l)=\mathfrak g_a(b,l^2)$ and $f_{a,b}(l)=d_{a,b}(l)/l^2$.
Using Lemmas~\ref{lem:m_var} and~\ref{lem:m0_bound}, we see that $d_{a,b}(\cdot)$ is non-positive, monotone non-increasing, and that $f_{a,b}(\cdot)$ is bounded.
Reformulating~\eqref{eq:m0n} by taking $R=l^2$, we get for $\ell \geq 2$
\[f_{a,b}(nl)\geq f_{a,b}\big((1+\alpha)l\big)-Cb^2\Big(\alpha+\frac1{\alpha^2l^2}\Big)\,.\]
Thus, the functions $d_{a,b}(l)$ and $f_{a,b}(l)$ satisfy the assumptions in Lemma~\ref{lem:df}. This assures the existence of a constant $\mathfrak e_a(b)\leq0$, depending on $a$ and $b$, such that
\[\lim_{l\rightarrow +\infty} f_{a,b}(l)=\mathfrak e_a(b)\,,\]
that is, with the choice $l=\sqrt{R}$
\[\lim_{R\rightarrow +\infty} \frac {\mathfrak g_a (b,R)}R=\mathfrak e_a(b)\,.\]
Moreover, Lemma~\ref{lem:m0_bound} ensures that $\mathfrak e_a(b)<0$.
The upper bound in Item~(5) of Theorem~\ref{thm:eba} follows from Lemma~\ref{lem:df}.
It remains to establish a lower bound for $\mathfrak g_a (b,R)/R$. Let $n\geq 1$ be an integer. By Lemma~\ref{lem:periodic},
\[\mathfrak g_a(b,nR)\leq n\,\mathfrak g_a (b,R)\,.\]
Dividing both sides by $nR$ and taking $n \rightarrow +\infty$ yields
\[\frac{\mathfrak g_a (b,R)}{R}\geq \mathfrak e_a(b)\,.\]
The monotonicity of the function $\mathfrak e_a(\cdot)$ is straightforward and follows from that of $\mathfrak g_a(\cdot,R)$. Let $\epsilon>0$, we have $\mathfrak g_a(b+\epsilon,R)\geq \mathfrak g_a (b,R)$. Dividing both sides of this inequality by $R$ then taking $R\rightarrow +\infty$ gives us $\mathfrak e_a(b+\epsilon)\geq \mathfrak e_a(b)$.
Our final task is to  prove the continuity of the function $\mathfrak e_a(\cdot)$.
Let $b \in \Big[1/|a|,1/\beta_a\Big)$, and $\epsilon>0$. We will prove that $\mathfrak e_a(\cdot)$ is right continuous at $b$. Since $\mathfrak e_a(\cdot)$ is monotone non-decreasing, 
$\mathfrak e_a(b+\epsilon)\geq \mathfrak e_a(b)$. Consequently, 
\[\liminf_{\epsilon\to0_+}\mathfrak e_a(b+\epsilon)\geq \mathfrak e_a(b)\,.\]
Hence, it is sufficient to prove that 
$\displaystyle\limsup_{\epsilon\rightarrow 0_+}\mathfrak e_a(b+\epsilon)\leq \mathfrak e_a(b)$. We may use the  following lower bound from~\eqref{eq:eba2},
\[\mathfrak e_a(b+\epsilon)\leq \frac{\mathfrak g_a(b+\epsilon,R)}{R}\]
which in turn yields 
\begin{equation}\label{eq:limsup}
\limsup_{\epsilon\rightarrow 0_+}\mathfrak e_a(b+\epsilon)\leq \limsup_{\epsilon\rightarrow 0_+}\frac{\mathfrak g_a(b+\epsilon,R)}{R}\,.
\end{equation}
Let $u \in \mathcal D_R$. We have $\mathfrak g_a(b+\epsilon,R)\leq \mathcal G_{a,b+\epsilon,R}(u)$, where the functional $\mathcal G_{a,\cdot,R}$ is defined in~\eqref{eq:Gb}.  We infer from~\eqref{eq:limsup} that
\begin{align*}
\limsup_{\epsilon\rightarrow 0_+}\mathfrak e_a(b+\epsilon)&\leq \frac 1R	\limsup_{\epsilon\rightarrow 0_+}\mathcal G_{a,b+\epsilon,R}(u)\\
                                                 &\leq \frac {\mathcal G_{a,b,R}(u)}{R}+\limsup_{\epsilon\rightarrow 0_+}\frac {\epsilon}{R}\int \big|(\nabla-i\sigma\Ab_0)u\big|^2\,dx\\
                                                 &=\frac {\mathcal G_{a,b,R}(u)}{R}\,.
\end{align*}
This is true for all $u \in \mathcal D_R$ and $R\geq 1$. Minimizing over $u\in\mathcal D_R$ yields, for all $R\geq 1$,
\[\limsup_{\epsilon\rightarrow 0_+}\mathfrak e_a(b+\epsilon)\leq \frac{\mathfrak g_a (b,R)}{R}\]
Taking $R\rightarrow +\infty$, we get the desired inequality.\\
Let $b \in \Big(1/|a|, 1/\beta_a\Big]$ and $\epsilon<0$.  Now we prove the left continuity at $b$. The monotonicity of $\mathfrak e_a(\cdot)$ yields that
$\displaystyle\limsup_{\epsilon\to0_-} \mathfrak e_a(b+\epsilon)\leq \mathfrak e_a(b)$. So,  
it is sufficient to prove that 
\[\liminf_{\epsilon\rightarrow 0_-}\mathfrak e_a(b+\epsilon)\geq \mathfrak e_a(b)\,.\]
Let   $\varphi_{a,b+\epsilon,R}$ be the minimizer of $\mathcal G_{a,b+\epsilon,R}$ defined in~\eqref{eq:phi_norm}. Using the upper bound in~\eqref{eq:eba2} together with~\eqref{eq:phi_decay}, we get
\begin{align*}
\mathfrak e_a(b+\epsilon)&\geq \frac {\mathfrak g_a(b+\epsilon,R)}R -C\frac{(b+\epsilon)^2} {R^{\frac13}}\\
               &\geq\frac {\mathcal G_{a,b+\epsilon,R}\big( \varphi_{a,b+\epsilon,R}\big)}R-C\frac{(b+\epsilon)^2} {R^{\frac13}}\\
							 &\geq \frac {\mathcal G_{a,b,R}\big( \varphi_{a,b+\epsilon,R}\big)}R+\frac{\epsilon}R\int \big|(\nabla-i\sigma\Ab_0)\varphi_{a,b+\epsilon,R}\big|^2\,dx-C\frac{(b+\epsilon)^2}{R^{\frac13}}\\
							 &\geq \frac {\mathfrak g_a (b,R)}R+C\epsilon-C\frac{(b+\epsilon)^2}{R^{\frac13}}.
\end{align*}
Hence,
\[\liminf_{\epsilon\rightarrow 0_-}\mathfrak e_a(b+\epsilon)\geq \frac {\mathfrak g_a (b,R)}R-C\frac{b^2}{R^{\frac13}}\]
Taking $R\rightarrow+\infty$,we get the desired inequality.

\subsection{An effective one-dimensional energy}\label{sec:1D}
Assume that $a \in [-1,1)\setminus\{0\}$ and $b>0$. For all $\xi \in \R$,  consider the functional
\begin{multline}\label{eq:E_1D}
\mathcal E_{a,b,\xi}^{1D}(f)=\int_{-\infty}^{0}\left(b|f'(t)|^2+b(at+\xi)^2|f(t)|^2-|f(t)|^2 +\frac 12|f(t)|^4 \right)\,dt\\+\int_{0}^{+\infty}\left(b|f'(t)|^2+b(t+\xi)^2|f(t)|^2-|f(t)|^2+\frac 12|f(t)|^4\right)\,dt\,.
\end{multline}
defined over the space $B^1(\R)$, and let
\begin{equation}\label{eq:1D-gse}
E_{a,b}^{1D}(\xi)=\inf_{f \in B^1(\R)} \mathcal E_{a,b,\xi}^{1D}(f)\,.\end{equation}
We would like to find a relationship between the 2D-energy in~\eqref{eq:m0-mN} and the 1D-energy in~\eqref{eq:1D-gse} for some specific value of $\xi$. The existing results on the Ginzburg--Landau functional with a uniform magnetic field suggest that  we should select $\xi$ so as to minimize the  function $\xi\mapsto E_{a,b}^{1D}(\xi)$, see~\cite{almog2007distribution, Correggi , pan2002surface}.  

In light of Remark~\ref{rem:trivial}, we will assume that $a$ and $b$ satisfy
\begin{equation}\label{eq:assump2}
a \in [-1,0) \qquad \mathrm{and} \qquad b \geq \frac 1 {|a|}\,.
\end{equation}
Under~\eqref{eq:assump2}, the numerical computations indicate that the global minimum $\beta_a$, defined in~\eqref{eq:lamda}, is attained at a non-degenerate {\bfseries unique} point, denoted  by 
$\zeta_a$ in~\eqref{eq:theta3} (see~\cite[Section~1.3]{hislop2016band}).  To our knowledge, such a uniqueness result has not been analytically proven yet. {\bfseries In the sequel, we will assume that uniqueness   of $\zeta_a$ holds.} Under this assumption,  Proposition~\ref{prop:mu_a_lim} yields that
for each fixed value of $b$ such that $1/|a|<b<1/\beta_a$, there exist two real numbers $\xi_1(a,b)$ and $\xi_2(a,b)$  satisfying
\[\xi_1(a,b)<\zeta_a<\xi_2(a,b)\,,\]
and
\[\left(\mu_a\right)^{-1}\left(\big(\beta_a,b^{-1}\big)\right)=\Big(\xi_1(a,b),\xi_2(a,b)\Big)\,.\]
With $\xi_1(a,b)$ and $\xi_2(a,b)$ in hand,  we can list some elementary properties of the functional $\mathcal E^{1D}_{a,b,\xi}$ in~\eqref{eq:E_1D}:
\begin{theorem}	\label{thm:E_1D}
	Let $a \in [-1,0)$ and $b \geq 1/|a|$.
	\begin{enumerate}
		\item The functional $\mathcal E^{1D}_{a,b,\xi}$ has a non-trivial minimizer in $B^1(\R)$ if and only if $1/|a| \leq b <1/\beta_a$. Furthermore, one can find a positive minimizer $f_\xi$, dependent on $a$  and $b$, such that any minimizer has the form $cf_\xi$ where $c \in \C$ and $|c|=1$.
		\item {\bfseries(Assuming that $\zeta_a$ is unique)} For $1/|a| < b <1/\beta_a$, there exists $\xi_0 \in \big(\xi_1(a,b),\xi_2(a,b)\big)$, dependent on $a$ and $b$, such that
		\[E_{a,b}^{1D}(\xi_0)=\inf_{\xi \in \R} E_{a,b}^{1D}(\xi)\,.\]
		\item (Feynman-Hellmann)
		\[\int_{-\infty}^{0}(at+\xi_0)|f_{\xi_0}(t)|^2\,dt+\int_{0}^{+\infty}(t+\xi_0)|f_{\xi_0}(t)|^2\,dt=0\,.\]
	\end{enumerate}
\end{theorem}
The proof of Theorem~\ref{thm:E_1D} may be derived exactly as  done in~\cite[Section 14.2]{fournais2010spectral} devoted to the analysis of the following 1D-functional 
\[\mathcal E_{b,\xi}^{1D}(f)=\int_{0}^{+\infty}\left(b|f'(t)|^2+b(t+\xi)^2|f(t)|^2-|f(t)|^2+\frac 12|f(t)|^4\right)\,dt\,,\]
defined over the space $B^1(\R_+)$. We introduce the ground state energy
\begin{equation}\label{eq:E-1D*}
E^{1D}_b(\xi)=\inf_{f\in B^1(\R)}\mathcal E_{b,\xi}^\mathrm{1D}(f)\,.
\end{equation}
The ground state energy in~\eqref{eq:E-1D*} plays a crucial role in the study of surface superconductivity under the presence of a uniform magnetic field (see~\cite{almog2007distribution,fournais2010spectral,helffer2011superconductivity, Correggi }). Let $\mathrm{E}_\mathrm{g.st}^\mathrm{unif}(\kappa,H)$ be the ground state energy of the functional in~\eqref{eq:GL} for $B_0=1$. Assuming that $H=b\kappa$ and $1<b<\Theta_0^{-1}$ ($\Theta_0\in(0,1)$ is a universal (spectral) constant defined in~\eqref{eq:theta1}), then as $\kappa\to+\infty$,
\begin{equation}\label{eq:E_1D_2}
\mathrm{E}_\mathrm{g.st}^\mathrm{unif}(\kappa,H)=|\partial \Omega|\kappa {b^{-\frac 12}} E_b^{1D}+\mathcal O(1)\,,
\end{equation}
where $E_b^{1D}=\inf _{\xi \in \R}E^{1D}_b(\xi)$. That has been conjectured by Pan~\cite{pan2002surface}, then proved by Almog-Helffer and Helffer-Fournais-Persson~\cite{almog2007distribution, helffer2011superconductivity} under a restrictive assumption on  $b$,  using a spectral approach. In the whole regime $b\in(1,\Theta_0^{-1})$, the upper bound part in~\eqref{eq:E_1D_2} easily holds (see~\cite[Section 14.4.2]{fournais2010spectral}), while the matching lower bound is more difficult to obtain and has been finally proved by Correggi-Rougerie~\cite{Correggi }. The proof of Correggi-Rougerie,  based on  the positivity of a certain \emph{cost function}, was markedly different from the  spectral approach of~\cite{almog2007distribution,helffer2011superconductivity}.

Going back to our step magnetic field problem and the one dimensional energy in~\eqref{eq:E_1D},  it is reasonable to make the following conjecture
\begin{conj}\label{conj:1D}
	Assume that $-1\leq a<0$ and $1/|a|< b<1/\beta_a$, where $\beta_a$ is defined in~\eqref{eq:lamda}. Then, the energy $\mathfrak e_a(b)$ introduced in~\eqref{eq:eba1} satisfies
 \[\mathfrak e_a(b)=E_{a,b}^\mathrm{1D}\,,\]
where
\begin{equation}\label{eq:E_1D_3}
E_{a,b}^{1D}=\inf_{\xi \in \R}E_{a,b}^{1D}(\xi)\,.
\end{equation}	
and $E_{a,b}^{1D}(\cdot)$ is defined in~\eqref{eq:1D-gse}.
	\end{conj}
By a symmetry argument,  Conjecture~\ref{conj:1D} trivially holds in the case  $a=-1$, namely
\begin{equation}\label{eq:conj*}
\mathfrak e_{-1}(b)=E_{-1,b}^\mathrm{1D}=E_b^\mathrm{1D}\,.
\end{equation} 
However, there are many points that do not allow us to prove this conjecture in the case where $a \in (-1,0)$. Besides the lack of the uniqueness of the minimum $\zeta_a$, the new potential term
\[(\mathsf s(t) t+\xi )^2\quad\mathrm{where}\quad\mathsf s(t)=\begin{cases}1&\mathrm{if}~t>0\,,\\a&\mathrm{if}~t<0 \,,\end{cases}\]
creates computational difficulties preventing the adoption of the proof in~\cite{Correggi}, (in particular, in the positivity proof of the cost function).
\section{The Frenet Coordinates}\label{sec:bc}

In this section, we assume that the set $\Gamma$ consists of a single smooth curve that may  intersect the boundary of $\Omega$ transversely in two points. In the general case, $\Gamma$ consists of a finite number of such curves. By working on each component separately, we reduce to the simple case above.

To study the  energy contribution along  $\Gamma$, we will use the  \emph{Frenet coordinates} which  are valid in a tubular neighbourhood of $\Gamma$. For more details regarding these coordinates, see e.g.~\cite[Appendix F]{fournais2010spectral}. We will list the basic properties of these coordinates here. 

For $t_0>0$, we define the open set 
\begin{equation}\label{eq:Gam0}
\Gamma(t_0)=\left\{ x \in \Omega~:~\dist(x,\Gamma)<t_0 \right\}\,.
\end{equation}
We introduce the function $t:\R^2\rightarrow\R$ as follows
\begin{equation}\label{eq:t_x}
t(x)=
\begin{cases}
\dist(x,\Gamma)&\mathrm{if~}  x \in \Omega_1\,,\\
-\dist(x,\Gamma)&\mathrm{if}~ x \in \Omega_2\,.
\end{cases}
\end{equation}
Let $\left(-|\Gamma|/2,|\Gamma|/2\right] \ni s\longmapsto M(s)\in\Gamma$ be  the arc-length parametrization of $\Gamma$ oriented counter clockwise. The vector
\begin{equation}\label{eq:T(s)}
T(s):=M'(s)
\end{equation}
is the unit tangent vector to $\Gamma$ at the point $M(s)$. Let $\nu(s)$ be the unit inward normal of $\partial \Omega_1$ at the point $M(s)$.
The  orientation of the parametrization $M$ is displayed as follows
\[\mathrm{det}\big(T(s),\nu(s) \big)=1\,.\]
The  curvature $k_r$ of $\Gamma$  is  defined by
\[T'(s)=k_r(s)\nu(s)\,.\]
 When $t_0$ is sufficiently small, the  transformation
\begin{equation}\label{Frenet}
\Phi~:~\left(-\frac{|\Gamma|}2,\frac{|\Gamma|}2\right]\times(-t_0,t_0)\, \ni (s,t)\longmapsto M(s)+t \nu(s) \in \Gamma(t_0)
\end{equation}
is  a diffeomorphism  whose  Jacobian is 
\[\mathfrak a(s,t)=\mathrm{det}(D\Phi)=1-tk_r(s)\,.\]
The inverse of $\Phi$,  $\Phi^{-1}$, defines a system of coordinates  for the tubular neighbourhood $ \Gamma(t_0)$ of $\Gamma$,
\[\Phi^{-1}(x)=\Big( s(x),t(x)\Big)\,.\]
To each function $u \in H_0^1\big(\Gamma(t_0) \big)$, we associate the function $\tilde{u} \in  H^1\Big(\Phi^{-1}\big(\Gamma(t_0) \big)\Big)$ as follows
\begin{equation}\label{eq:u-tu}
\tilde{u}(s,t)=u\big(\Phi(s,t)\big)\,.
\end{equation}
We also associate to any vector field $A=(A_1,A_2) \in H^1_\mathrm{loc}(\R^2,\R^2)$, the vector field
\[\tilde{A}=(\tilde{A_1},\tilde{A_2}) \in H^1\left(\left(-\frac{|\Gamma|}2,\frac{|\Gamma|}2\right]\times(-t_0,t_0),\R^2 \right)\]
where 
\begin{equation}\label{eq:A_tild1}
\tilde{A_1}(s,t)=\mathfrak a(s,t)A\big(\Phi(s,t) \big)\cdot T(s)\quad\mathrm{and}\quad \tilde{A_2}(s,t)=A\big(\Phi(s,t) \big)\cdot\nu(s)\,.
\end{equation}
Then we have the following change of variable formulae:
\begin{equation}\label{eq:A_tild2}
\int_{\Omega_1}\big|\big(\nabla-i A \big)u \big|^2\,dx=\int_{-\frac{|\Gamma|}2}^{\frac{|\Gamma|}2}\int_0^{t_0}\left(\mathfrak a^{-2}\big|(\partial_s-i\tilde{A_1})\tilde{u}\big|^2+\big|(\partial_t-i\tilde{A_2})\tilde{u}\big|^2 \right)\,  \mathfrak a\,ds\,dt\,,
\end{equation}
\begin{equation}\label{eq:A_tild3}
\int_{\Omega_2}\big|\big(\nabla-i A \big)u \big|^2\,dx=\int_{-\frac{|\Gamma|}2}^{\frac{|\Gamma|}2}\int_{-t_0}^0\left(\mathfrak a^{-2}\big|(\partial_s-i\tilde{A_1})\tilde{u}\big|^2 +\big|(\partial_t-i\tilde{A_2})\tilde{u}\big|^2 \right)\,  \mathfrak a\,ds\,dt\,,
\end{equation}
and
\begin{equation}\label{eq:A_tild4}
\int_{\R^2}|u(x)|^2\,dx=\int_{-\frac{|\Gamma|}2}^{\frac{|\Gamma|}2}\int_{-t_0}^{t_0} |\tilde{u}|^2\,  \mathfrak a\,ds\,dt\,.
\end{equation}
We define
\[\tilde{B}(s,t)=B\big(\Phi(s,t)\big),\qquad \text{for all}\ (s,t) \in \left(-\frac{|\Gamma|}2,\frac{|\Gamma|}2\right]\times(-t_0,t_0)\,.\]
Note that
\begin{align}
\Big(\partial_s \tilde{A_2}(s,t)-\partial_t \tilde{A_1}(s,t)\Big)\,ds\wedge dt
 &=B\big(\Phi(s,t)\big)\,dx\wedge dy \nonumber \\
 &=\big(1-tk_r(s) \big)\tilde{B}(s,t)\,ds\wedge dt\label{eq:B_tild}
\end{align}
which gives us 
\[\curl \tilde{A}=\partial_s \tilde{A_2}-\partial_t \tilde{A_1}=\big(1-tk_r(s) \big)\tilde{B}(s,t)\,.\]

In Propositions~\ref{prop:Anew1} and~\ref{prop:Anew2}, we will construct a special gauge transformation that will  allow us to express a given vector field in a canonical manner.
\begin{proposition}\label{prop:Anew1}
For any vector field $A=(A_1,A_2) \in  H^1(\Omega,\R^2)$, there exists a $H^2$-function $\omega$  such that the vector field defined by $A_{new}=A-\nabla \omega$ satisfies
\[\Big(\tilde{A}_{new}\Big)_2=0\quad \mathrm{in}\ \Gamma\left(\frac{t_0}2\right)\,,\]
where $\tilde{A}_{new}$ is the vector field associated to $A_{new}$ as in~\eqref{eq:A_tild1}.
\end{proposition}
\begin{proof}
Using~\eqref{eq:A_tild1}, we get
\begin{equation}\label{eq:A_tild5}
\Big(\tilde{A}_{new}\Big)_1=\tilde{A}_1+\partial_s\tilde{\omega},\qquad \Big(\tilde{A}_{new}\Big)_2=\tilde{A}_2+\partial_t\tilde{\omega}
\end{equation}
where $\tilde{\omega}(s,t)=\omega\big (\Phi(s,t)\big)$.
Due to the regularity of $t(x)$ in $\Gamma(t_0)$, we may define $\omega$ so that
\[\tilde{\omega}(s,t)=-\chi(t)\int_0^{t}\tilde{A}_2(s,\tilde{t})\,d{\tilde{t}}\,,\]
where $\chi$ is a cut-off function supported in $(-t_0,t_0)$, satisfying $0\leq\chi\leq1$ and $\chi=1$ in $\big[-t_0/2, t_0/2\big]$.
\end{proof}
\begin{proposition}\label{prop:Anew2}
Let $a \in [-1,1)\setminus\{0\}$ and $A=(A_1,A_2)$ be a vector field in $H^1(\Omega,\R^2)$.  
For any $x_0\in \Gamma$, there exists a neighbourhood $V_{x_0}$ of $x_0$ and a function $\omega_{x_0} \in H^2(V_{x_0})$ such that the vector field $A_{new}:=A-\nabla\omega_{x_0}$ satisfies
\[\Big(\tilde{A}_{new}\Big)_2=0\quad \mathrm{in}\ V_{x_0},\qquad \Big(\tilde{A}_{new}\Big)_1=0\quad \mathrm{on}\ \Gamma \cap V_{x_0}\,.\]

Furthermore, if $\curl A=\mathbbm{1}_{\Omega_1}+a\mathbbm{1}_{\Omega_2}$, then we have in $V_{x_0}$
\begin{equation}\label{eq:Anew3}
\Big(\tilde{A}_{new}\Big)_1= 
\begin{cases}
- \big(t-\frac {t^2}2 k_r(s)\big),&\mathrm{if}~t>0\,,\\
-a \big(t-\frac {t^2}2 k_r(s)\big),&\mathrm{if}~t<0\,.
\end{cases}
\end{equation}
\end{proposition}
\begin{proof}
After performing a translation,  we may suppose that the coordinates of $x_0$ are given by $(s=0,t=0)$. Define
\[V_{x_0}=\left\{x=\Phi(s,t) \in \Gamma\left (\frac{t_0}2\right ):-\frac\ell2<s<\frac\ell2 \right\}\,.\]
By Proposition~\ref{prop:Anew1}, after performing a gauge transformation, we may assume that $\tilde{A}_2=0$ in $\Gamma\left(t_0/2\right)$.
We define the function $\omega_{x_0}$ such that
\begin{equation}\label{eq:phi_x}
\tilde\omega_{x_0}(s,t)=\omega_{x_0}\big (\Phi(s,t)\big)=-\chi(s)\int_{-\ell}^s\tilde{A}_1(\tilde{s},0)\,d\tilde{s}\,, 
\end{equation}
where $\chi$ is a cut-off function   supported in $(-\ell,\ell)$, satisfying $0\leq\chi\leq1$ and $\chi=1$ in $\Big[-\ell/2, \ell/2\Big]$. Using~\eqref{eq:A_tild5}, a straightforward computation yields that $\Big(\tilde{A}_{new}\Big)_1 =0$ on $\Gamma \cap V_{x_0}$.
Now since $\Big(\tilde{A}_{new}\Big)_2 =0$ in $V_{x_0}$, then by~\eqref{eq:B_tild}
\begin{equation}\label{eq:Anew4}
\partial_t \big(\tilde{A}_{new}\big)_1=-\tilde{B}_{new}(s,t)\big( 1-tk_r(s)\big)=-\tilde{B}(s,t)\big( 1-tk_r(s)\big)=
\begin{cases}
-\big( 1-tk_r(s)\big), &t>0\,,\\
-a\big( 1-tk_r(s)\big), &t<0\,.
\end{cases}
\end{equation}
Integrating~\eqref{eq:Anew4} with respect to $t$ (starting from $t=0$), we get~\eqref{eq:Anew3}.
\end{proof}
\section{A Local Energy}\label{sec:local_en}

In this section, we will introduce a `local version' of the Ginzburg--Landau functional in~\eqref{eq:GL}. For this local functional, we will be able to write precise estimates of the ground state energy, which in turn will prove useful in estimating the ground state energy of the full functional in~\eqref{eq:GL}.

We start by introducing various (geometric) notations/assumptions. Select a positive number $t_0$ sufficiently small so that the Frenet coordinates of Section~\ref{sec:bc}  are valid in the tubular neighbourhood $\Gamma(t_0)$ defined in~\eqref{eq:Gam0}. Let $0<c_1<c_2$ be fixed constants and $\ell $ be a parameter that is allowed to vary in such a manner  that
\begin{equation}\label{eq:ell=kp}
c_1\kappa^{- \frac 34} < \ell < c_2\kappa^{- \frac 34}\,.
\end{equation}
We will refer to~\eqref{eq:ell=kp} by writing $\ell\approx \kappa^{-3/4}$.
We will assume that $\kappa$ is sufficiently large so that  $\ell<t_0/2$. 

Consider the set
\begin{equation}\label{eq:V_0}
\mathcal V(\ell)=\left\{ (s,t)\in \Gamma(t_0)~:~-\frac{\ell}2< s<\frac{\ell}2,\ -\ell< t< \ell \right\}\,,
\end{equation}
and the magnetic potential $\tilde{F}$ defined in $\mathcal V(\ell)$ by
\begin{equation} \label{F_1}
\tilde{F}(s,t)=\left(\tilde{F}_1(s,t),0\right)=\left(-\sigma\Big( t-\frac {t^2}2k_r(s)\Big),0\right)\,,
\end{equation}
where $\sigma=\sigma(s,t)$ is defined for $(s,t) \in \R^2$ by (see~\eqref{eq:sigma1})
\[\sigma(s,t)=\mathbbm{1}_{\R_+}(t)+a\mathbbm{1}_{\R_-}(t)\,.\]
Consider the domain $\mathcal D_\ell$\,:
\begin{equation}\label{eq:D}
\mathcal D_\ell=\left\{u \in H^1_0\big(\mathcal V(\ell)\big)\cap L^\infty\big(\mathcal V(\ell)\big)~:~\|u\|_\infty \leq 1 \right\}\,.
\end{equation}
For $u \in \mathcal D_\ell$, we define the (local) energy
\begin{equation}\label{eq:G_u}
\mathfrak G\big(u;\mathcal V(\ell)\big)=\int_{\mathcal V(\ell)}\left(\mathfrak a^{-2}\big|(\partial_s-i\kappa H\tilde{F}_1)u\big|^2 +|\partial_tu|^2-\kappa^2|u|^2+\frac{\kappa^2}2 |u|^4\right)\, \mathfrak a\,ds\,dt\,,
\end{equation}
where $\mathfrak a(s,t)=1-tk_r(s)$. Now we introduce the following ground state energy
\begin{equation}\label{eq:infG}
\mathfrak G_0=\mathfrak G_0(\kappa,H,\ell)=\inf_{u \in  \mathcal D_\ell}\mathfrak G\big(u;\mathcal V(\ell)\big)\,.
\end{equation}
Using standard variational methods, one can prove the existence of a minimizer $u_0$ of $\mathfrak G$
\begin{equation*}\label{eq:u_0}
\mathfrak G_0=\mathfrak G\big(u_0;\mathcal V(\ell)\big)\,,
\end{equation*} 
Our aim is to write matching upper and lower bounds for $\mathfrak G_0$ as $\kappa\to+\infty$ in the regime
\begin{equation}\label{eq:A4}
H=b\kappa,\quad a\in [-1,0)\quad \mathrm{and}\quad b \geq \frac1{|a|} \,.
\end{equation}

\subsection{Lower bound of $\mathfrak G_0$}
\begin{lemma}\label{lem:lower_G0}
Under Assumption~\eqref{eq:A4}, there exist two constants $\kappa_0>1$ and $C>0$ dependent only on $a$ and $b$  such that, if
$\kappa\geq\kappa_0$ and $\ell$ as in~\eqref{eq:ell=kp}, then
\begin{equation}\label{eq:loc_lower}
\mathfrak G_0\geq b^{-\frac 12}\kappa \ell \mathfrak e_a(b)-C\,,
\end{equation}
where $\mathfrak G_0$ and $\mathfrak e_a(b)$ are defined in~\eqref{eq:infG} and~\eqref{eq:eba1} respectively.
\end{lemma}
\begin{proof}
Notice that $\mathfrak a(s,t)$ is bounded in the set $\mathcal V(\ell)$ as follows
\begin{equation}\label{eq:a_bound}
1-C\ell\leq \mathfrak a(s,t)\leq 1+C\ell\,.
\end{equation}
Consequently 
\begin{equation}\label{eq:G}
\mathfrak G\big(u;\mathcal V(\ell)\big)\geq  (1-C\ell)\mathcal J(u)-C\kappa^2\ell\int_{\mathcal V(\ell)}|u|^2\,dsdt\,,
\end{equation}
where
\begin{equation}\label{eq:J_u}
\mathcal J(u)=\int_{\mathcal V(\ell)}\left(\big|(\partial_s-i\kappa H\tilde{F}_1)u\big|^2 +|\partial_tu|^2-\kappa^2|u|^2+\frac{\kappa^2}2 |u|^4\right)\, dsdt\,.
\end{equation}
We apply the Cauchy's inequality to get
\begin{align*}
\big|(\partial_s-i\kappa H \tilde{F}_1)u\big|^2&= \left|\left(\partial_s+i\sigma \kappa H\Big( t-\frac {t^2}2k_r(s)\Big)\right)u\right|^2 \\
                                      &\geq (1-\kappa^{-\frac 12})\big|(\partial_s+i\sigma\kappa Ht)u\big|^2-\kappa^{\frac 12}\sigma^2\kappa^2H^2\frac{t^4}4k_r^2(s)|u|^2\\
                                      &\geq (1-\kappa^{-\frac 12})\big|(\partial_s+i\sigma\kappa Ht)u\big|^2-C\kappa^{\frac 52}\ell^4H^2|u|^2\quad\mathrm{because ~}|t|\leq \ell\,.
\end{align*}
Inserting the previous estimate into~\eqref{eq:J_u} and using  the uniform bound $|u|\leq1$, we obtain
\begin{equation}\label{eq:J}
\mathcal J(u)\geq (1-\kappa^{-\frac 12})\mathcal T(u)-C\mathcal R(u)\,,
\end{equation}
where
\[\mathcal T(u)=\int_{\mathcal V(\ell)}\left(\big|(\partial_s+i\sigma \kappa Ht)u\big|^2 +|\partial_tu|^2-\kappa^2|u|^2+\frac{\kappa^2}2 |u|^4\right)\, dsdt\]
and
\[\mathcal R(u)=\kappa^{\frac 32}\ell^2+\kappa^{\frac 52}H^2\ell^6\,.\]
We introduce the following parameters 
\[R=\sqrt{\kappa H} \ell,\quad \gamma=\sqrt{\kappa H}s,\quad \tau=\sqrt{\kappa H} t\,,\]
and  define the re-scaled function
\begin{equation*}
\breve{u}(\gamma,\tau)=
\begin{cases}
u(s,t)&\mathrm{if~} (\gamma,\tau)\in \displaystyle\left(-\frac R2,\frac R2\right)\times(-R,R)\,,\\
0 &\mathrm{otherwise}\,.
\end{cases}
\end{equation*}
Recall the parameter $b= H/\kappa$ in~\eqref{eq:A_2}, in the new scale we may write 
\begin{align*}
\mathcal T(u)&=\int_{-\frac R2}^{\frac R2}\int_{-\infty}^{+\infty}\Big[|(\partial_\gamma+i\sigma\tau)\breve{u}|^2 +|\partial_\tau\breve{u}|^2-\frac \kappa H|\breve{u}|^2+\frac{\kappa}{2H} |\breve{u}|^4\Big]\, d\gamma d\tau\\
                     &=\frac 1b \mathcal G_{a,b,R}(\breve{u})\,,
\end{align*}
$\mathcal G_{a,b,R}$ is the functional introduced in~\eqref{eq:Gb}, and $\breve{u}\in\mathcal D_R$ the domain introduced in~\eqref{eq:DR} (since $u \in \mathcal D_\ell$).  Invoking  Theorem~\ref{thm:eba}, we conclude that
\begin{equation}\label{eq:T}
\mathcal T(u)\geq\frac 1b R\, \mathfrak e_a(b)\,.
\end{equation}
We plug the estimates~\eqref{eq:J} and~\eqref{eq:T} in~\eqref{eq:G}, then use $\mathfrak e_a(b)\leq 0$ and the assumptions on $\kappa$ and $\ell$ to finish the proof of Lemma~\ref{lem:lower_G0}.
\end{proof}

\subsection{Upper bound of $\mathfrak G_0$}
\begin{lemma}\label{lem:upper_G0}
Under Assumption~\eqref{eq:A4}, there exist two constants $\kappa_0>1$ and  $C>0$ dependent only on $a$ and $b$ such that, if
$\kappa\geq\kappa_0$ and $\ell$ as in~\eqref{eq:ell=kp}, then
\begin{equation}\label{eq:loc_upper}
\mathfrak G_0\leq b^{-\frac12}\kappa\ell \mathfrak e_a(b)+C\kappa^{\frac 16}\,,
\end{equation}
where $\mathfrak G_0$ and $\mathfrak e_a(b)$ are  defined in~\eqref{eq:infG} and~\eqref{eq:eba1} respectively.
\end{lemma}
\begin{proof}
For $R=\ell\sqrt{\kappa H}$, consider $\varphi=\varphi_{a,b,R}$ the minimizer of $\mathcal G_{a,b,R}$ defined in~\eqref{eq:phi_norm}.
We define the function $u$ in  $\mathcal D_\ell$ as follows
\begin{equation}\label{eq:u0}
u(s,t)=\chi\left(\frac t \ell\right)\varphi\left(s \sqrt{\kappa H},t \sqrt{\kappa H}\right)\,,
\end{equation}
where $\chi$ is a standard smooth cut-off function satisfying
\[0\leq \chi\leq 1\ \mathrm{in}\ \R,\quad \chi=1\ \mathrm{in} \ \left[-\frac12,\frac12\right]\quad \mathrm{and}\quad \supp \chi\subset[-1,1]\,.\]
Next, we define the following  function (with the  re-scaled variables)
\[v(\gamma,\tau)=u(s,t)\qquad\Big( (\gamma,\tau) \in \left(-\frac R2,\frac R2\right)\times(-R,R)\Big)\,,\]
with
\[\gamma=\sqrt{\kappa H}s,\quad \tau=\sqrt{\kappa H} t\,.\]
Using  the definition of $v$, the decay of $\varphi$ in~\eqref{eq:phi_decay}, and the bound of $\mathfrak a(s,t)$ in~\eqref{eq:a_bound}, we get
\begin{align}
\mathfrak G(u)&\leq (1+C\ell)\mathcal J(u)+ C\kappa^2\ell\int_{\mathcal V(\ell)}|u|^2\,dsdt\,, \nonumber\\
                                 &\leq (1+C\ell) \mathcal K(v) +C\kappa^2\ell^3\,, \label{eq:E5}
\end{align}
where  
$\mathcal J(u)$ was defined in~\eqref{eq:J_u}, 
\[\mathcal K(v)=\int_{-\frac R2}^{\frac R2}\int_{-R}^{R} \left[\left|\left(\partial_\gamma+i\sigma\left(\tau-\epsilon\frac{\tau^2}2 k_r\Big(\frac \gamma \epsilon\Big)\right)\right)v\right|^2+|\partial_\tau v|^2-\frac \kappa H |v|^2+\frac \kappa {2 H} |v|^4\right]\, d\gamma d\tau\,, \]
and $\epsilon=1/\sqrt{\kappa H}$.\\
Let $\chi_R(\tau)=\chi\big(\tau /R\big)=\chi\big(t/ \ell\big)$. We will estimate now each term of $\mathcal K(v)$ apart, using mainly the decay of the minimizer $\varphi$ in~\eqref{eq:phi_decay} and the properties of the function $\chi_R$.

We start with
\begin{align*}
\int_{-\frac R2}^{\frac R2}\int_{-R}^{R} |\partial_\tau v|^2\, d\gamma d\tau &= \int_{-\frac R2}^{\frac R2}\int_{-R}^{R} \Big|\chi_R(\tau)\partial_\tau \varphi+\varphi \partial_\tau \chi_R(\tau)\Big|^2\, d\gamma d\tau \,,\\
                                                                       &\leq (1+\kappa^{-\frac 14})\int_{-\frac R2}^{\frac R2}\int_{-R}^{R} \big|\chi_R(\tau)\partial_\tau \varphi\big|^2\, d\gamma d\tau  \\
                                                                       &\qquad +C\kappa^{\frac 14}\int_{-\frac R2}^{\frac R2}\int_{-R}^{R}\big|\varphi \partial_\tau \chi_R(\tau)\big|^2\, d\gamma d\tau\,, \\ 
																		&\leq (1+\kappa^{-\frac 14})\int_{-\frac R2}^{\frac R2}\int_{-\infty}^{+\infty} \big|\partial_\tau \varphi\big|^2\, d\gamma d\tau +C\kappa^{-\frac 34}\ell^{-1}\,.  
\end{align*}
Similarly, we have
\begin{align*}
&\int_{-\frac R2}^{\frac R2}\int_{-R}^{R}\left|\left(\partial_\gamma+i\sigma\left(\tau-\epsilon\frac{\tau^2}2 k_r\Big(\frac s \epsilon\Big)\right)\right)v\right|^2\, d\gamma d\tau\,,\\
&\leq (1+\kappa^{-\frac 14}) \int_{-\frac R2}^{\frac R2}\int_{-R}^{R} \big|\big(\partial_\gamma+i\sigma\tau\big)v\big|^2\,d\gamma d\tau
+ C \kappa^{\frac 14}\int_{-\frac R2}^{\frac R2}\int_{-R}^{R}\sigma^2\epsilon^2\tau^4|v|^2\,d\gamma d\tau\,,\\
&
                   \leq (1+\kappa^{-\frac 14}) \int_{-\frac R2}^{\frac R2}\int_{-\infty}^{+\infty} \big|\big(\partial_\gamma+i\sigma\tau\big)\varphi\big|^2\,d\gamma d\tau
                     + C\kappa^{\frac {13}4}\ell^5\quad\mathrm{because ~}|\tau|\leq R \,.\\
\end{align*}
Next, we may select $R_0$ sufficiently large so that, for all $R\geq R_0$, we have:
\begin{equation}\label{eq:R}
 |\tau|\geq \frac R2\Longrightarrow \frac {|\tau|}{\ln^2|\tau|}\geq R^\frac 12\,.
\end{equation}
The decay of $\varphi$ and~\eqref{eq:R} yield
\begin{align*}
 \int_{-\frac R2}^{\frac R2}\int_{-R}^{R} |v|^2 \,d\gamma d\tau&=  \int_{-\frac R2}^{\frac R2}\int_{-R}^{R} \big| \chi_R(\tau)\varphi\big|^2  \,d\gamma d\tau \\
                                                                                 &= \int_{-\frac R2}^{\frac R2}\int_{-\infty}^{+\infty} \big|\varphi\big|^2  \,d\gamma d\tau+\int_{-\frac R2}^{\frac R2}\int_{-\infty}^{+\infty} \big( \chi_R^2(\tau)-1\big)|\varphi|^2  \,d\gamma d\tau \\
& \geq \int_{-\frac R2}^{\frac R2}\int_{-\infty}^{+\infty} |\varphi|^2  \,d\gamma d\tau -  \int_{-\frac R2}^{\frac R2}\int_{{ |\tau|\geq R/2}} |\varphi|^2  \,d\gamma d\tau \\																												
&\geq  \int_{-\frac R2}^{\frac R2}\int_{-\infty}^{+\infty} |\varphi|^2  \,d\gamma d\tau- C\kappa^{\frac 12}\ell^{\frac 12}\,.
\end{align*}
Finally, we write the obvious inequality
\[\int_{-\frac R2}^{\frac R2}\int_{-R}^{R} |v|^4 \,d\gamma d\tau\leq \int_{-\frac R2}^{\frac R2}\int_{-\infty}^{+\infty} |\varphi|^4 \,d\gamma d\tau\,.\]
Gathering the foregoing estimates, we get
\begin{align}
\mathcal K(v) &\leq \frac {(1+\kappa^{-\frac 14})}{b} \mathcal G_{a,b,R}(\varphi)+C\Big(\kappa^{\frac 34}\ell+\kappa^{-\frac 34}\ell^{-1}+\kappa^{\frac {13}4} \ell^5+ \kappa^{\frac 12}\ell^{\frac 12}\Big)\,,\\
   &\leq \frac {(1+\kappa^{-\frac 14})}{b} \mathcal G_{a,b,R}(\varphi)+C \kappa^{\frac 18}\,. \label{eq:A} 
\end{align}
Invoking Theorem~\ref{thm:eba}, we implement~\eqref{eq:A} into~\eqref{eq:E5} to get the desired upper bound.
\end{proof}
\section{Local Estimates}\label{sec:local_est}

The aim of this section is to study the concentration of the minimizers $(\psi,\Ab)$ of the functional in~\eqref{eq:GL} near the set $\Gamma$ that separates the values of the applied magnetic field (see Assumption~\ref{assump}). This will be displayed by local estimates of the Ginzburg--Landau energy and the $L^4$-norm of the Ginzburg--Landau parameter in Theorem~\ref{thm:E_loc1}.

We will introduce the necessary notations and assumptions. Starting with the local energy of the configuration $(\psi,\Ab) \in
H^1(\Om;\C)\times\Hd$ in any open set $D\subset\Om$ as follows
\begin{equation}\label{eq:local0}
\begin{aligned}
&\mathcal E_0(\psi,\Ab;D) =\int_D\left( \big|(\nb-i \kp H \Ab)\psi\big|^2-\kp^2|\psi|^2 +\frac 12 \kp^2|\psi|^4 \right)\,dx\,,\\
&\mathcal E(\psi,\Ab;D)=\mathcal E_0(\psi,\Ab;D) +(\kp H)^2 \int_\Om|\curl\Ab-B_0|^2\,dx.
\end{aligned}
\end{equation}
Choose $t_0>0$ sufficiently small so that the Frenet coordinates of Section~\ref{sec:bc}  are valid in the tubular neighbourhood $\Gamma(t_0)$ defined in~\eqref{eq:Gam0}. For all $x \in \Gamma(t_0)$, define the point $p(x) \in \Gamma$ as follows 
\[\dist(x,p(x))=\dist(x,\Gamma)\,.\]
Let $\ell\approx \kappa^{-3/4}$ be a parameter in~\eqref{eq:ell=kp} (for some fixed choice of the constants $c_1$ and $c_2$). Let $x_0 \in \Gamma\setminus\partial\Omega$ that is allowed to {\bfseries vary} in such a manner that
\begin{equation}\label{eq:con-x0}
\mathrm{dist}(x_0,\partial\Omega)> 2 \ell\,.
\end{equation}
Consider the following neighbourhood of $x_0$,
\begin{equation}\label{eq:N0}
\mathcal N_{x_0}(\ell)=\{x \in \Omega~:~\dist_\Gamma(x_0,p(x)) < \frac \ell 2,\ -\ell <t(x) < \ell\}\,,
\end{equation}
where $t(\cdot)$ is defined in~\eqref{eq:t_x}. For $\kappa$ sufficiently large (hence $\ell$ sufficiently small),  we get that $\mathcal N_{x_0}(\ell)$ does not intersect the boundary $\partial\Omega$, thanks to~\eqref{eq:con-x0}.
 As a consequence of the assumption in~\eqref{eq:con-x0}, {\bfseries all the estimates that we will write will hold uniformly with respect to the point $x_0$}.

We assume that $a\in[-1,0)$ and $b>0$ are fixed and satisfy
\begin{equation}\label{eq:A3}
b>\frac 1{|a|}\,.
\end{equation}
When~\eqref{eq:A3} holds, we are able to use the exponential decay of the Ginzburg--Landau parameter away from the set $\Gamma$ and the surface $\partial\Omega$ (see Theorem~\ref{thm:decay}).

\begin{theorem}\label{thm:E_loc1}
	Let $a \in[-1,0)$ and $b>1/|a|$. There exists $\kappa_0>0$ and a function $\mathfrak r:[\kappa_0,+\infty)\rightarrow (0,+\infty)$ such that $\displaystyle\lim_{\kappa\to+\infty}\mathfrak r(\kappa)=0$ and the following is true. For $\kappa\geq \kappa_0$, $H=b\kappa$ and  $\ell\approx \kappa^{-3/4}$ as in~\eqref{eq:ell=kp}, for any $x_0 \in \Gamma$ satisfying~\eqref{eq:con-x0}, every minimizer $(\psi, \Ab) \in H^1(\Om;\C)\times\Hd $ of the functional  in~\eqref{eq:GL} satisfies
	\begin{equation}\label{eq:En_gamma}
	\left|\mathcal E_0\left(\psi,\Ab;\mathcal N_{x_0}(\ell)\right)-b^{-\frac 12}\kappa \ell \mathfrak e_a(b)\right|\leq \kappa \ell\mathfrak r(\kappa)\,,
	\end{equation}
	and
	\begin{equation}\label{eq:L4_gamma}
	\left|\frac 1 {\ell}\int_{\mathcal N_{x_0}(\ell)}|\psi|^4\,dx+2b^{-\frac 12}\kappa^{-1}\mathfrak e_a(b)\right|\leq \kappa^{-1}\mathfrak r(\kappa)\,,
	\end{equation}\	
	where $\mathcal N_{x_0}(\cdot)$ is the set in~\eqref{eq:N0}, $\mathcal E_0$ is the local energy  in~\eqref{eq:local0}, and $\mathfrak e_a(b)$ is the limiting energy defined in~\eqref{eq:eba1}.
	Furthermore, the function $\mathfrak r$ is independent of the point $x_0\in\Gamma$.
\end{theorem}
The proof of Theorem~\ref{thm:E_loc1} follows by collecting the results of Proposition~\ref{prop:L4_gamma_up} and  Proposition~\ref{prop:L4_gamma_low} below, which are derived along the lines of~\cite[Section 4]{helffer2017decay} in the study of local surface superconductivity.

Part of the proof of Theorem~\ref{thm:E_loc1} is based on the following remark. After performing a translation, we may assume that the Frenet coordinates of $x_0$ are $(s=0,t=0)$ (see Section~\ref{sec:bc}).  Recall the local Ginzburg--Landau  energy $\mathcal E_0$ introduced in~\eqref{eq:local0}. Let $\Fb$ be the vector field introduced in Lemma~\ref{A_1}.  We have the following relation
\begin{equation}\label{eq:uv}
\mathcal E_0\big(u,\Fb;\mathcal N_{x_0}(\ell)\big)=\mathfrak G\big(\tilde{v};\mathcal V(\ell)\big)\,,
\end{equation}
where $\mathfrak G$ is defined in~\eqref{eq:G_u}, $u \in H_0^1(\mathcal N_{x_0}(\ell))$, $\tilde{v}$ is the function associated to $v=e^{-i\kappa H\omega_{x_0}}u$ by the transformation $\Phi^{-1}$ (see~\eqref{eq:u-tu}), and $\omega_{x_0}$ is the gauge transformation function defined in Proposition~\ref{prop:Anew2}.

\subsection{Lower bound of the local energy} 
We start by establishing a lower bound for the local energy $\mathcal E_0\big(u,\Ab;\mathcal N_{x_0}(\ell)\big)$ for an arbitrary function $ u \in H_0^1(\mathcal N_{x_0}(\ell))$ satisfying $|u| \leq 1$.  {\bfseries We will work under the assumptions made in this section}, notably, we assume that~\eqref{eq:A3} holds,  and $\ell\approx\kappa^{-3/4}$ (see~\eqref{eq:ell=kp}), and {\bfseries in the regime where $H=b\kappa$}.

\begin{proposition}\label{prop:E0_low}
There exist two constants $\kappa_0>1$ and  $C>0$ such that, for $\kappa \geq \kappa_0$ and for all $x_0\in\Gamma$ satisfying~\eqref{eq:con-x0}, the following is true. If 
\begin{itemize}
\item  $(\psi,\Ab)\in H^1(\Omega;\C)\times \Hd$ is a solution of~\eqref{eq:Euler}\,.
\item $u \in H_0^1(\mathcal N_{x_0}(\ell))$ satisfies $|u| \leq 1$\,.
\end{itemize}  
then 
\begin{equation*}
\mathcal E_0\big(u,\Ab;\mathcal N_{x_0}(\ell)\big) \geq b^{-\frac12}\kappa \ell \mathfrak e_a(b)-C.
\end{equation*} 
where $\mathcal N_{x_0}(\cdot)$ is the neighbourhood defined  in~\eqref{eq:N0}, $\mathcal E_0$ is the functional defined in~\eqref{eq:local0}, and $\mathfrak e_a(b)$ is the limiting energy in~\eqref{eq:eba1}.
\end{proposition}
\begin{proof}
	Let $\alpha \in (0,1)$ and $\Fb$ be the vector field introduced in Lemma~\ref{A_1}. Define the function $\phi_{x_0}(x)=\Big(\Ab(x_0)-\Fb(x_0)\Big)\cdot x$. As a consequence of the fourth item in Theorem~\ref{thm:priori}, we get the following useful approximation of the vector potential  $\Ab$
	\begin{equation}\label{eq:AF}
	|\Ab(x)-\nabla\phi_{x_0}(x)-\Fb(x)| \leq \frac C\kappa \ell^\alpha\qquad \mathrm{for}\ x \in \mathcal N_{x_0}(\ell)\,.
	\end{equation} 
{\bfseries We choose}  $\alpha=2/3$ in~\eqref{eq:AF}. Define the function $w=e^{-i\kappa H\phi_{x_0}}u$. Using~\eqref{eq:AF} and Cauchy's inequality, we may write
	\begin{align*}
	\big|(\nb-i \kp H \Ab)u\big|^2 &\geq (1-\kappa^{-\frac 12})\big|(\nb-i \kp H \Fb)w\big|^2 -\kappa^{\frac 12}\kappa^2H^2|\Ab-\nabla\phi_{x_0}-\Fb|^2|w|^2\\
	&\geq (1-\kappa^{-\frac 12})\big|(\nb-i \kp H \Fb)w\big|^2 -\kappa^{\frac 52}\ell^{\frac 43}|w|^2\,.
	\end{align*} 
	By using that $|w| \leq 1$, we get further
	\begin{equation*}
	\mathcal E_0(u,\Ab;\mathcal N_{x_0}(\ell)) \geq (1-\kappa^{-\frac 12})\mathcal E_0(w,\Fb;\mathcal N_{x_0}(\ell))-C\Big(\kappa^{\frac 32}\ell^2+\kappa^{\frac 52}\ell^{\frac {10}3}\Big)\,.
	\end{equation*}
	Now, define the function $v=e^{-i\kappa H\omega_{x_0}}w$, where $\omega_{x_0}$ is introduced in Proposition~\ref{prop:Anew2}.  We may use the relation in~\eqref{eq:uv} to write 
	\begin{equation*}\label{eq:E02}
	\mathcal E_0(u,\Ab;\mathcal N_{x_0}(\ell))\geq  (1-\kappa^{-\frac 12})\mathfrak G\big(\tilde{v};\mathcal V(\ell)\big)-C\Big(\kappa^{\frac 32}\ell^2+\kappa^{\frac 52}\ell^{\frac {10}3}\Big)\,, 
	\end{equation*}
	Finally, we use the lower  bound in Lemma~\ref{lem:lower_G0} together with the inequality $\mathfrak e_a(b) \leq 0$. This finishes the proof of Proposition~\ref{prop:E0_low}.
\end{proof}

	\subsection{Sharp upper bound on $L^4$- norm}
	   We will derive  a lower bound of the local energy $\mathcal E_0\big(\psi,\Ab;\mathcal N_{x_0}(\ell)\big)$ and an upper bound of the $L^4$-norm of $\psi$ valid for any critical point $(\psi,\Ab)$  of the functional  in~\eqref{eq:GL}. Again, we remind the reader that {\bfseries  we assume that~\eqref{eq:A3} holds, $\ell\approx\kappa^{-3/4}$ (see~\eqref{eq:ell=kp}) and $H=b\kappa$}.

\begin{proposition}\label{prop:L4_gamma_up}
There exist two constants $\kappa_0>1$ and $C>0$ such that, for all $x_0\in \Gamma$ satisfying~\eqref{eq:con-x0}, the following is true. If $(\psi, \Ab) \in H^1(\Om;\C)\times\Hd $ is a critical point of the functional in~\eqref{eq:GL} for $\kappa\geq\kappa_0$, then
	 \begin{equation}\label{eq:Epsi_low_asymp}
	 \mathcal E_0\big(\psi,\Ab;\mathcal N_{x_0}({\ell})\big) \geq b^{-\frac12}\kappa \ell \mathfrak e_a(b)-C\kappa^{\frac 3{16}}\,,
	 \end{equation}
	 and
	\begin{equation}\label{eq:psi4_up_asymp}
		\frac 1\ell \int_{\mathcal N_{x_0}(\ell)}|\psi|^4\,dx \leq -2b^{-\frac12}\kappa^{-1}\mathfrak e_a(b)+C\kappa^{-\frac {17}{16}}\,.
	\end{equation}
Here $\mathcal N_{x_0}(\cdot)$, $\mathcal E_0$, and $\mathfrak e_a(b)$ are respectively defined in~\eqref{eq:N0},~\eqref{eq:local0}, and~\eqref{eq:eba1} .
\end{proposition}	
\begin{proof}
 In the sequel,
$\gamma=\kappa^{-3/16}$ and $\kappa$ is sufficiently large so that $\gamma \in (0,1)$. We denote by $\hat{\ell}=(1+\gamma)\ell$.

Consider a smooth function  $f$ satisfying
\begin{equation}\label{eq:f}
\begin{aligned}
&f=1\ \mathrm{in}\ \mathcal N_{x_0}(\ell)\,,\quad f=0 \ \mathrm{in}\  \mathcal N_{x_0}\big(\hat\ell\,\big)^\complement\,,\\
& 0\leq f \leq 1,\quad |\nabla f| \leq C\gamma^{-1} \ell^{-1}\ \mathrm{and}\  |\Delta f| \leq C\gamma^{-2} \ell^{-2}\ \mathrm{in}\ \Omega\,.
\end{aligned}
\end{equation}
\emph{Proof of~\eqref{eq:Epsi_low_asymp}.}
We will extract a lower bound of $\mathcal E_0\left(\psi,\Ab;\mathcal N_{x_0}(\hat{\ell})\right)$.
We use the following simple identity (see~\cite[p.~2871]{kachmar2016distribution})
	\begin{equation}\label{eq:Delta}
  \int_{\mathcal N_{x_0}(\hat{\ell})}\big|(\nb-i \kp H \Ab)f\psi\big|^2\,dx =\int_{\mathcal N_{x_0}(\hat{\ell})}\big|f(\nb-i \kp H \Ab)\psi\big|^2\,dx-\int_{\mathcal N_{x_0}(\hat{\ell})}f\Delta f|\psi|^2\,dx\,,
\end{equation}
obtained using integration by parts.
Having in hand~\eqref{eq:Delta}, $|\psi|\leq 1$ and $\big|\supp (\Delta f)\big|\leq C\gamma \ell^2$, we can write
\begin{align*}
\int_{\mathcal N_{x_0}(\hat{\ell})}\big|(\nb-i \kp H \Ab)f\psi\big|^2\,dx & \leq \int_{\mathcal N_{x_0}(\hat{\ell})}\big|f(\nb-i \kp H \Ab)\psi\big|^2\,dx+\int_{\mathcal N_{x_0}(\hat{\ell})}f|\Delta f||\psi|^2\,dx\,, \\
& \leq \int_{\mathcal N_{x_0}(\hat{\ell})}\big|f(\nb-i \kp H \Ab)\psi\big|^2\,dx+ C\gamma^{-2}\ell^{-2}\int_{\supp (\Delta f)}|\psi|^2\,dx\,,\\
& \leq \int_{\mathcal N_{x_0}(\hat{\ell})}\big|f(\nb-i \kp H \Ab)\psi\big|^2\,dx+ C\gamma^{-1}\,.\\
\end{align*}
On the other hand, we write
\begin{equation}\label{eq:1-f2}
\begin{aligned}
&\int_{\mathcal N_{x_0}(\hat{\ell})}f^2|\psi|^2\,dx\\
 &=\int_{\mathcal N_{x_0}(\hat{\ell})}|\psi|^2\,dx -\int_{\mathcal N_{x_0}(\hat{\ell})}(1-f^2)|\psi|^2\,dx\\
                                          &=\int_{\mathcal N_{x_0}(\hat{\ell})}|\psi|^2\,dx -\int_{\mathcal N_{x_0}(\hat{\ell})\cap \{|t(x)|\leq \gamma\ell\}}(1-f^2)|\psi|^2\,dx  -\int_{\mathcal N_{x_0}(\hat{\ell})\cap \{|t(x)|> \gamma\ell\}}(1-f^2)|\psi|^2\,dx
\end{aligned}
\end{equation}
where $t(\cdot)$ is the distance function defined in~\eqref{eq:t_x}. Recall that $\gamma=\kappa^{-3/16}$, then $\gamma\ell\gg \kappa^{-1}$ which, together with~\eqref{eq:A3}, allow us to use the exponential decay of $|\psi|^2$ in $\mathcal N_{x_0}(\hat{\ell})\cap \{|t(x)|> \gamma\ell\}$ (see~Theorem~\ref{thm:decay}).
Consequently, the integral over $\mathcal N_{x_0}(\hat{\ell})\cap \{|t(x)|> \gamma\ell\}$ in~\eqref{eq:1-f2} is exponentially small when $\kappa \rightarrow +\infty$ (having $0 \leq f \leq 1$). In addition, we have 
\[\left|\supp(1-f^2) \cap \mathcal N_{x_0}(\hat{\ell})\cap \{|t(x)|\leq \gamma\ell\}\right|=\mathcal O(\gamma^2\ell^2)\,.\]
This yields
\begin{equation*}
\int_{\mathcal N_{x_0}(\hat{\ell})}f^2|\psi|^2\,dx \geq \int_{\mathcal N_{x_0}(\hat{\ell})}|\psi|^2\,dx-C\gamma^2\ell^2\,.
\end{equation*}
Hence, 
\begin{align}
\mathcal E_0\big(f\psi,\Ab;\mathcal N_{x_0}(\hat{\ell})\big) &\leq \mathcal E_0\big(\psi,\Ab;\mathcal N_{x_0}(\hat{\ell})\big) +C\kappa^2\gamma^2\ell^2+C\gamma^{-1}\,.\nonumber\\
                                                                 &\leq \mathcal E_0\big(\psi,\Ab;\mathcal N_{x_0}(\hat{\ell})\big) +C\kappa^{\frac 3{16}}\,.\label{eq:f3}
\end{align}
The fact that $f\psi \in H_0^1\Big(\mathcal N_{x_0}(\hat{\ell})\Big)$ and $|f\psi| \leq 1$ allows us to use the lower bound result established in Proposition~\ref{prop:E0_low} for $u=f\psi$. This yields together with~\eqref{eq:f3} 
\begin{equation}\label{eq:E_psi_lower*}
\mathcal E_0\big(\psi,\Ab;\mathcal N_{x_0}(\hat{\ell})\big) \geq b^{-\frac12}\kappa \hat\ell \mathfrak e_a(b)-C\kappa^\frac 3{16}
\end{equation}
This finishes the proof of~\eqref{eq:Epsi_low_asymp}, but with $\hat\ell$ appearing instead of $\ell$. However, this is not harmful, as we could start the argument with $\check\ell =(1+\gamma)^{-1}\ell$ in place of $\ell$ and then modify $\hat\ell$ accordingly; in this case we would get $\hat\ell=(1+\gamma)\check\ell=\ell$ as required.

\emph{Proof of~\eqref{eq:psi4_up_asymp}.}
In light of the first equation in~\eqref{eq:Euler} satisfied by $(\psi,\Ab)$, we get using integration  by parts  (see~\cite[(6.2)]{fournais2011nucleation})
	\begin{equation*}\label{eq:f_psi}
	\int_{\mathcal N_{x_0}(\hat{\ell})}\left(\big|(\nb-i \kp H \Ab)f\psi\big|^2 -|\nabla f|^2|\psi|^2 \right)\,dx=\kappa^2\int_{\mathcal N_{x_0}(\hat{\ell})}\left(|\psi|^2-|\psi|^4\right)f^2\,dx\,.
	\end{equation*}
	Consequently,
	\begin{equation}\label{eq:f_psi_1}
	\mathcal E_0\big(f\psi,\Ab;\mathcal N_{x_0}(\hat{\ell})\big)=\kappa^2\int_{\mathcal N_{x_0}(\hat{\ell})}f^2\Big(-1+\frac 12 f^2\Big)|\psi|^4\,dx+\int_{\mathcal N_{x_0}(\hat{\ell})}|\nabla f|^2|\psi|^2\,dx\,.
	\end{equation}
	Since $f = 1$ in $\mathcal N_{x_0}(\ell)$ and $-1+1/2 f^2 \leq - 1/2$ in $\mathcal N_{x_0}(\hat{\ell})$, we get
	\[\int_{\mathcal N_{x_0}(\hat{\ell})}f^2\Big(-1+\frac 12 f^2\Big)|\psi|^4\,dx \leq -\frac 12\int_{\mathcal N_{x_0}(\ell)}|\psi|^4\,dx\,.\]
	We use the previous inequality,~\eqref{eq:f_psi_1} and the estimate $\big|\supp |\nabla f|\big|\leq C\gamma \ell^2$ to obtain
	\begin{align}
	\mathcal E_0\big(f\psi,\Ab;\mathcal N_{x_0}(\hat{\ell})\big) &\leq -\frac {\kappa^2}2\int_{\mathcal N_{x_0}(\ell)}|\psi|^4\,dx+C\gamma^{-1}\,,\nonumber \\
	&\leq -\frac {\kappa^2}2\int_{\mathcal N_{x_0}(\ell)}|\psi|^4\,dx+C\kappa^{\frac 3{16}}\,.\label{eq:E03}
	\end{align}
	We insert the lower bound in Proposition~\ref{prop:E0_low} into~\eqref{eq:E03}  to get the upper bound of the $L^4$-norm in~\eqref{eq:psi4_up_asymp}.
	\end{proof}

	\subsection{Sharp lower bound on the $L^4$-norm}
Complementary to Proposition~\ref{prop:L4_gamma_up}, we will prove  Proposition~\ref{prop:L4_gamma_low} below, whose conclusion holds for minimizing configurations only. We continue working under {\bfseries the assumption  that~\eqref{eq:A3} holds, $\ell\approx\kappa^{-3/4}$ (see~\eqref{eq:ell=kp}) and $H=b\kappa$}.	
	\begin{proposition}\label{prop:L4_gamma_low}
There exist two constants $\kappa_0>1$ and $C>0$ such that, for all $x_0\in \Gamma$ satisfying~\eqref{eq:con-x0}, the following is true. If $(\psi, \Ab) \in H^1(\Om;\C)\times\Hd $ is a \textit{\textbf{minimizer}} of the functional in~\eqref{eq:GL} for $\kappa\geq\kappa_0$ , then
		\begin{equation}\label{eq:Epsi_up_asymp}
		\mathcal E_0\big(\psi,\Ab;\mathcal N_{x_0}(\ell)\big) \leq b^{-\frac12}\kappa \ell \mathfrak e_a(b)+C\kappa^{\frac 3{16}}\,,
		\end{equation}
		and
		\begin{equation}\label{eq:psi4_low_asymp}
		\frac 1\ell \int_{\mathcal N_{x_0}(\ell)}|\psi|^4\,dx \geq -2b^{-\frac12}\kappa^{-1}\mathfrak e_a(b)-C\kappa^{-\frac {17}{16}}\,.
		\end{equation}
Here $\mathcal N_{x_0}(\cdot)$, $\mathcal E_0$, and $\mathfrak e_a(b)$ are respectively defined in~\eqref{eq:N0},~\eqref{eq:local0}, and~\eqref{eq:eba1}.
	\end{proposition}
	\begin{proof}
	The proof is divided into five steps.

	\paragraph{\itshape Step~1. Construction of a test function and decomposition of the energy}
		
		The construction of the test function is inspired from that by Sandier and Serfaty in their study of bulk superconductivity in~\cite{sandier2003decrease}. Let $(\psi,\Ab)$ be a minimizer of the function in~\eqref{eq:GL}. For $\gamma=\kappa^{- 3/16}$ and $\hat{\ell}=(1+\gamma)\ell$, we define the function
		\begin{equation}\label{eq:test}
		u(x)=\mathbbm{1}_{\mathcal N_{x_0}(\hat{\ell})}(x) e^{i\kappa H\phi_{x_0}(x)}v_0(x)+\eta(x)\psi(x)\,,
		\end{equation}
		where $v_0(x)=e^{i\kappa H\omega_{x_0}(x)}u_0\circ \Phi^{-1}(x)$, $\phi_{x_0}$ and $\omega_{x_0}$ are the gauge transformation functions introduced respectively in~\eqref{eq:AF} and Proposition~\ref{prop:Anew2}, $\Phi$ is the coordinate transformation in~\eqref{Frenet},
		$u_0$ is a minimizer of the functional $\mathfrak G\big(\cdot,\mathcal V(\hat{\ell})\big)$ defined in~\eqref{eq:G_u}, and $\eta$ is a smooth function satisfying
		\begin{equation}\label{eq:eta}
\begin{aligned}
&\eta=0\ \mathrm{in}\ \mathcal N_{x_0}\big(\hat{\ell}\,\big)\,,\quad \eta=1\ \mathrm{in}\ \mathcal N_{x_0}\big((1+2\gamma)\ell\big)^\complement,\\
&0\leq \eta \leq 1,\quad |\nabla \eta| \leq C\gamma^{-1} \ell^{-1}\ \mathrm{and}\  |\Delta \eta| \leq C\gamma^{-2} \ell^{-2}\ \mathrm{in}\ \Omega\,.
\end{aligned}
		\end{equation}
		Recalling the energies defined in~\eqref{eq:GL} and~\eqref{eq:local0}, we write
	\[\mathcal E_{\kappa,H}(\cdot,\Ab)=\mathcal E_0\big(\cdot,\Ab; \mathcal N_{x_0}(\hat{\ell})\big)+\mathcal E\big(\cdot,\Ab; \mathcal N_{x_0}(\hat{\ell})^\complement \big)\,.\]
	We denote by 
	\[\mathcal E_1(\cdot,\Ab)=\mathcal E_0\big(\cdot,\Ab; \mathcal N_{x_0}(\hat{\ell})\big),\
\mathcal E_2(\cdot,\Ab)=\mathcal E_0\big(\cdot,\Ab; \mathcal N_{x_0}(\tilde{\ell})\setminus\mathcal N_{x_0}(\hat{\ell})\big),\ \mathcal E_3(\cdot,\Ab)=\mathcal E\big(\cdot,\Ab; \mathcal N_{x_0}(\tilde{\ell})^\complement\big),\]
where $\tilde{\ell}=(1+2\gamma)\ell$.	
	Hence, we get the following decomposition of the functional in~\eqref{eq:GL},
	\[\mathcal E_{\kappa,H}(\cdot,\Ab)=\mathcal E_1(\cdot,\Ab)+\mathcal E_2(\cdot,\Ab)+\mathcal E_3(\cdot,\Ab)\,.\]

	\paragraph{\itshape Step~2. Estimating $\mathcal E_1(u,\Ab)$}
	Using the approximation in~\eqref{eq:AF} for $\alpha=2/3$, the Cauchy-Schwarz inequality and the  uniform bound  $|v_0| \leq 1$, we may write 
	\begin{align}
    \mathcal E_1(u,\Ab)  & \leq (1+\kappa^{-\frac 12})\mathcal E_0\big(v_0,\Fb; \mathcal N_{x_0}(\hat{\ell})\big)+\kappa^{-\frac 12}\kappa^2\int_{\mathcal N_{x_0}(\hat{\ell})}|v_0|^2\,dx \nonumber\\
    &     \quad +\kappa^{\frac 12}\kappa^2H^2\int_{\mathcal N_{x_0}(\hat{\ell})}|\Ab-\nabla\phi_{x_0}-\Fb|^2|v_0|^2\,dx\,,\nonumber\\
    &                      \leq (1+\kappa^{-\frac 12})\mathcal E_0\big(v_0,\Fb; \mathcal N_{x_0}(\hat{\ell})\big)+C\,.\label{eq:E1_0}
	\end{align}
		But  by~\eqref{eq:uv}, we have $\mathcal E_0\big(v_0,\Fb; \mathcal N_{x_0}(\hat{\ell})\big)=\mathfrak G\big(u_0,\mathcal V(\hat{\ell})\big)$. Hence, we insert the upper bound in Lemma~\ref{lem:upper_G0} into~\eqref{eq:E1_0} to get
	\begin{equation}\label{eq:up_E1_u}
	\mathcal E_1(u,\Ab) \leq  b^{-\frac12}\kappa\hat\ell \mathfrak e_a(b)+C\kappa^{\frac 16}\,. 
	\end{equation}
	
	\paragraph{\itshape Step~3. Estimating $\mathcal E_2(u,\Ab)$}
	Notice that $u=\eta\psi$ with $0 \leq \eta \leq 1$ in $\mathcal N_{x_0}(\tilde{\ell})\setminus\mathcal N_{x_0}(\hat{\ell})$. Then, we do a straightforward computation, similar to the one done in the proof of~\eqref{eq:f3}, replacing $f$ by $\eta$ and $\mathcal N_{x_0}(\hat{\ell})$ by $\mathcal N_{x_0}(\tilde{\ell})\setminus\mathcal N_{x_0}(\hat{\ell})$. This gives the following relation between $\mathcal E_2(u,\Ab)$ and $\mathcal E_2(\psi,\Ab)$
	\begin{equation}\label{eq:up_E2}
	\mathcal E_2(u,\Ab) \leq \mathcal E_2(\psi,\Ab)+C\kappa^{\frac 3{16}}\,.
	\end{equation}	
	
	\paragraph{\itshape Step~4. Estimating $\mathcal E_1(\psi,\Ab)$}
Since $(\psi,\Ab)$ is a minimizer of the functional  $\mathcal E_{\kappa,H}$ defined in~\eqref{eq:GL}, we write
	\[\mathcal E_{\kappa,H}(\psi,\Ab) \leq \mathcal E_{\kappa,H}(u,\Ab)\,,\]
	that is
	\[\mathcal E_1(\psi,\Ab)+\mathcal E_2(\psi,\Ab)+\mathcal E_3(\psi,\Ab)\leq \mathcal E_1(u,\Ab)+\mathcal E_2(u,\Ab)+\mathcal E_3(u,\Ab)\,.\]
    Noticing that $\mathcal E_3(u,\Ab)=\mathcal E_3(\psi,\Ab)$, we get
    \[\mathcal E_1(\psi,\Ab)+\mathcal E_2(\psi,\Ab)\leq \mathcal E_1(u,\Ab)+\mathcal E_2(u,\Ab)\,.\]
We use the estimate of $\mathcal E_2(u,\Ab)$ in~\eqref{eq:up_E2} to get
	\[\mathcal E_1(\psi,\Ab) \leq \mathcal E_1(u,\Ab)+C\kappa^{\frac 3{16}}\,.\]
	We insert the upper bound of $\mathcal E_1(u,\Ab)$ in~\eqref{eq:up_E1_u} in the previous inequality to get
	\begin{equation}\label{eq:up_E1_psi}
	\mathcal E_1(\psi,\Ab) \leq b^{-\frac12}\kappa\hat\ell \mathfrak e_a(b)+C\kappa^{\frac 3{16}}\,.
	\end{equation}
Recalling that $\mathcal E_1(\psi,\Ab)=\mathcal E_1\big(\psi,\Ab;\mathcal N_{x_0}(\hat\ell)\big)$, 	we see that~\eqref{eq:up_E1_psi} is nothing but~\eqref{eq:Epsi_up_asymp} with $\hat\ell$ appearing instead of $\ell$. Starting the argument with $\ell$ replaced by $\check\ell=(1+\gamma)^{-1}\ell$, we get~\eqref{eq:up_E1_psi} for $\hat\ell=(1+\gamma)\check\ell=\ell$, as required. Therefore, we finished the proof of~\eqref{eq:Epsi_up_asymp}.
	
	\paragraph{\itshape Step~5. Lower bound of the $L^4$-norm of $\psi$}
	
	Consider the function $f$ defined in~\eqref{eq:f}. 
	We use the properties of this function, mainly that $f=1$ in $\mathcal N_{x_0}(\ell)$ and $0 \leq f \leq 1$ in $\Omega$,
	to obtain
	\begin{align*}
	\int_{\mathcal N_{x_0}(\hat{\ell})}f^2\Big(-1+\frac 12 f^2\Big)|\psi|^4\,dx &= -\frac 12\int_{\mathcal N_{x_0}(\ell)}|\psi|^4\,dx+\int_{\mathcal N_{x_0}(\hat{\ell})\setminus \mathcal N_{x_0}(\ell)}f^2\Big(-1+\frac 12 f^2\Big)|\psi|^4\,dx \\
	                                     &\geq -\frac 12\int_{\mathcal N_{x_0}(\ell)}|\psi|^4\,dx-\int_{\mathcal N_{x_0}(\hat{\ell})\setminus \mathcal N_{x_0}(\ell)}|\psi|^4\,dx\,. 
	\end{align*}
	Following an argument similar to the one for~\eqref{eq:1-f2}, we divide the set $\mathcal N_{x_0}(\hat{\ell})\setminus \mathcal N_{x_0}(\ell)$ into the two sets 
	$\left(\mathcal N_{x_0}(\hat{\ell})\setminus \mathcal N_{x_0}(\ell)\right)\cap \{|t(x)|\leq \gamma\ell\}$ and $\left(\mathcal N_{x_0}(\hat{\ell})\setminus \mathcal N_{x_0}(\ell)\right) \cap \{|t(x)|> \gamma\ell\}$, and we use this time the exponential decay of $|\psi|^4$ deduced from~Theorem~\ref{thm:decay} to get 
	\begin{equation}\label{eq:f2}
	\int_{\mathcal N_{x_0}(\hat{\ell})}f^2\Big(-1+\frac 12 f^2\Big)|\psi|^4\,dx \geq -\frac 12\int_{\mathcal N_{x_0}(\ell)}|\psi|^4\,dx-C\kappa^{-\frac{15}{8}}\,.
	\end{equation}
	Inserting~\eqref{eq:f2} into the identity in~\eqref{eq:f_psi_1} 
	gives us
	\[\mathcal E_1(f\psi,\Ab) \geq -\frac {\kappa^2}2\int_{\mathcal N_{x_0}(\ell)}|\psi|^4\,dx-C\kappa^\frac 18\,.\]
	The previous inequality together with~\eqref{eq:f3} and~\eqref{eq:up_E1_psi} establish the lower bound of the $L^4$-norm of $\psi$ as $\kappa \rightarrow +\infty$.
		\end{proof}

\subsection{Proof of Theorem~\ref{thm:E_loc1}}	Estimate~\eqref{eq:En_gamma} in Theorem~\ref{thm:E_loc1} is obtained by gathering results in~\eqref{eq:Epsi_low_asymp} and in~\eqref{eq:Epsi_up_asymp}, while Estimate~\eqref{eq:L4_gamma} follows from~\eqref{eq:psi4_up_asymp} and in~\eqref{eq:psi4_low_asymp}. 
\section{Surface Superconductivity}\label{sec:surf}
In Section 6, we worked under the assumption
\begin{equation*}
b>1/|a|,\qquad\ a \in[-1,0)\,.
\end{equation*}
We investigated the local behaviour of the sample in a tubular neighbourhood of $\Gamma$. In this section, and {\bfseries under the same assumption}, we are concerned in the local behaviour of the sample near the boundary of $\Omega$.

The analysis of superconductivity near $\partial \Om$ in our case of  a step  magnetic field ($B_0$ satisfying~\ref{assump}) is essentially  the same as that in the uniform  field case, since $B_0$ is constant in each of $\Om_1$ and $\Om_2$. Thereby, the results presented in this section are well-known in literature since the celebrated work of Saint-James and de\,Gennes~\cite{saint1963onset}. We refer to~\cite{almog2007distribution,Correggi ,fournais2005energy,helffer2011superconductivity,fournais2011nucleation,fournais2013ground,lu1999estimates,pan2002surface} for rigorous results in general 2D and 3D samples subjected to a constant magnetic field, and to~\cite{ning2009observation} for recent experimental results. Particularly, local surface estimates were recently established in~\cite{helffer2017decay}, when  $B_0 \in C^{0,\alpha}(\overline{\Omega})$ for some $\alpha \in (0,1)$. We will adapt these results to our discontinuous magnetic field  (see Theorem~\ref{thm:surf} below).

The statement of our main result,  Theorem~\ref{thm:surf}, involves the surface energy  $E_\mathrm{surf}$ that we  introduce in the next  section.

\subsection{The surface energy function}
Let $b>1$ and $R>0$. Consider the reduced Ginzburg--Landau
functional:
\[\mathcal W(U_R) \ni \phi \mapsto \mathcal E_{b,R}(\phi)=\int_{U_R} \left(b\left|(\nabla-i\Ab_0)\phi\right|^2-|\phi|^2+\frac 12|\phi|^4\right)\,d\gamma d\tau \,,\]
where $(\gamma,\tau) \in \R^2$, $\Ab_0(\gamma,\tau)=(-\tau,0)$, $U_R=\left(-R/2,R/2\right) \times (0,+\infty)$ and
\[\mathcal W(U_R)=\{u\in L^2(U_R)~:~(\nabla-i\Ab_0)u \in L^2(U_R),\,u(\pm R,\cdot)=0 \}\,.\]
Let $d(b,R)$ be the ground state energy defined by
\[d(b,R)=\inf_{\phi \in \mathcal W(U_R)}\mathcal E_{b,R}(\phi)\,.\]
Pan proved in~\cite{pan2002surface} the existence of a non-decreasing continuous function $E_\mathrm{surf}:[1,\Theta_0^{-1}] \rightarrow (-\infty,0]$ such that 
\begin{equation}\label{eq:Esurf}
E_\mathrm{surf}(b)=\lim_{R\rightarrow +\infty}\frac {d(b,R)}{R}\,,
\end{equation}
where $\Theta_0$ is defined in~\eqref{eq:theta1}. The surface energy $E_\mathrm{surf}$ has also been described by a 1D-problem (see~\cite{almog2007distribution,helffer2011superconductivity}). Recently,  Correggi--Rougerie~\cite{Correggi}  proved that for all $b\in(1,\Theta_0^{-1})$, $E_\mathrm{surf}=E^{1D}_b$, where $E^{1D}_b$ is the energy introduced in~\eqref{eq:E-1D*}.
One important property of the function $E_\mathrm{surf}$ is  (see~\cite{fournais2005energy})
\begin{equation}\label{eq:Esurf-prop}
E_\mathrm{surf}(\Theta_0^{-1})=0\ \mathrm{and}\ E_\mathrm{surf}(b)<0,\ \text{for all}\ b \in\left[1,\Theta_0^{-1}\right)\,.
\end{equation} 
This property allows us to extend the function $E_\mathrm{surf}$ continuously to $[1,+\infty)$, by setting it to zero on $[\Theta_0^{-1},+\infty)$. This extension of the surface energy  is still denoted by $E_\mathrm{surf}$ for simplicity.

\subsection{Local surface superconductivity}
Let $t_0>0$ and $j \in \{1,2\}$. We define the following set:
\begin{equation}\label{eq:Om0j}
\Omega_j(t_0)=\left\{ x \in \Omega_j~:~\dist\left(x,\partial \Omega_j \cap \partial \Omega\right)<t_0 \right\}\,.
\end{equation}
 Assume that $t_0$ is sufficiently small, then for any $x \in \Omega_j(t_0)$, there exists a unique point $p(x) \in \partial \Omega_j \cap \partial \Omega$ satisfying
\[\dist\left(x,\partial \Omega_j \cap \partial \Omega\right)=\dist\left(x,p(x)\right)\,.\]
Let $\ell\approx \kappa^{-3/4}$ be the parameter in~\eqref{eq:ell=kp}. Choose $x_0 \in \partial \Omega_j \cap \partial \Omega$ satisfying
\begin{equation}\label{eq:con-x0-1}
\mathrm{dist}(x_0,\Gamma)> 2\ell\,.
\end{equation}
We introduce the following small neighbourhood of $x_0$:
\begin{equation}\label{eq:N_j}
\mathcal N^j_{x_0}(\ell)=\{x \in \Omega_j~:~\dist_{\partial \Omega}\left(x_0,p(x)\right)< \frac \ell 2,\ \dist_\Omega\left(x,p(x)\right)< \ell\}\,.
\end{equation} 
The assumption on $x_0$ in~\eqref{eq:con-x0-1} guarantees that  $\mathcal N^j_{x_0}(\ell)$ and $\Gamma$ do not intersect, for sufficiently large $\kappa$ (hence sufficiently small $\ell$). Consequently,  the estimates in this section (and particularly in Theorem~\ref{thm:surf} below) hold uniformly with respect to the point $x_0$.

Recall the magnetic field $B_0$ defined in~\ref{assump} ($B_0={\mathbbm 1}_{\Omega_1}+a{\mathbbm 1}_{\Omega_2}$).
\begin{theorem}\label{thm:surf}
 Let $a \in [-1,0)$ and $b>1/|a|$ . There exists $\kappa_0>0$ and a function $ \breve{\mathfrak r}:[\kappa_0,+\infty)\rightarrow (0,+\infty)$ such that $\displaystyle\lim_{\kappa\to+\infty} \breve{\mathfrak r}(\kappa)=0$ and the following is true. For $\kappa\geq \kappa_0$, $H=b\kappa$, $\ell$ as in~\eqref{eq:ell=kp}, $j\in\{1,2\}$, $x_0 \in \partial \Omega_j \cap \partial \Omega$ satisfying~\eqref{eq:con-x0-1}, and every minimizer $(\psi, \Ab) \in H^1(\Om;\C)\times\Hd $ of the functional  in~\eqref{eq:GL}, we have
	\[\left|\mathcal E_0\left(\psi,\Ab;\mathcal N^j_{x_0}(\ell)\right)-b^{-\frac 12}|B_0(x_0)|^{-\frac 12}\kappa \ell E_\mathrm{surf}\left(b|B_0(x_0)|\right)\right|\leq \kappa \ell \breve{\mathfrak r}(\kappa)\,,\]
	and
	\[\left|\frac 1 {\ell}\int_{\mathcal N^j_{x_0}(\ell)}|\psi|^4\,dx+2b^{-\frac 12}|B_0(x_0)|^{-\frac 12}\kappa^{-1} E_\mathrm{surf}\left(b|B_0(x_0)|\right)\right|\leq  \kappa^{-1}\breve{\mathfrak r}(\kappa)\,,\]	
	where $\mathcal N^j_{x_0}(\cdot)$ is defined in~\eqref{eq:N_j}, and $\mathcal E_0$ is the local energy defined in~\eqref{eq:local0}.
	Furthermore, the function $\breve{\mathfrak r}$ is independent of the point $x_0$.
\end{theorem}

The  estimates in Theorem~\ref{thm:surf} are already established in~\cite{helffer2017decay} when  $B_0 \in C^{0,\alpha}(\overline{\Omega})$ for some $\alpha \in (0,1)$. They still hold in our case because $B_0$ is constant in $\Omega_1$ and $\Omega_2$ (see~Assumption~\ref{assump}) by repeating the proof given in~\cite{helffer2017decay}. 

\section{Proof of Main Results}\label{sec:main}
\subsection{Proof of Theorem~\ref{thm:Eg}}
We will work under the assumptions of Theorem~\ref{thm:Eg} restricted to the non-trivial case
\[a \in[-1,0)\quad \mathrm{and}\quad b>1/|a|\,,\]
and we will gather the results of the two previous sections to establish, as $\kappa$ tends to $+\infty$, asymptotic estimates of the global ground state energy $\Es$ in~\eqref{eq:gr_st} and of the $L^4$-norm of the order parameter $\psi$, where $(\psi,\Ab)$ is a minimizer of this energy.
	\begin{figure}[h]
		\centering
		\includegraphics[scale=0.5]{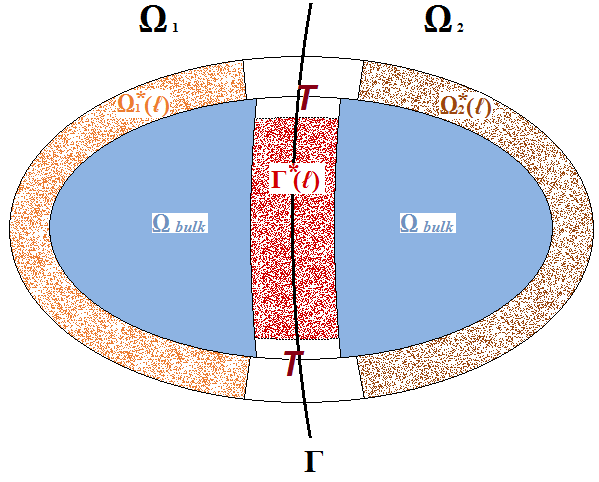} 
		\caption{Schematic representation of the sets $\Gamma^*(\ell)$, $\Omega_1^*(\ell)$, $\Omega_2^*(\ell)$, $\Omega_\text{bulk}$ and $T$, where $\ell \approx \kappa^{-3/4}$. In the regime where $a \in [-1,0)$, $H =b\kp$ and $1/|a|<b$, the blue region is in the  normal state, while the other regions  may carry superconductivity.}
		\label{fig:gamma_l}
	\end{figure}

 \subsubsection{ Energy lower bound.}
		Let $(\psi,\Ab)$ be a minimizer of~\eqref{eq:GL}. The definition of $\Es$ obviously  implies that
		\[\Es=\mathcal E_0\big(\psi,\Ab;\Omega\big) +\kp^2H^2\int_\Om\big|\curl\Ab-B_0\big|^2\,dx\geq \mathcal E_0\big(\psi,\Ab;\Omega\big)\,,\]
		where $\mathcal E_0$ is defined in~\eqref{eq:local0} and $B_0$ is defined in~\ref{assump}. Hence, it suffices to find a relevant lower bound of $\mathcal E_0\big(\psi,\Ab;\Omega\big)$. To that end, we will decompose the sample $\Omega$ into the sets 
		$\Gamma^*(\ell)$, $\Omega_1^*(\ell)$, $\Omega_2^*(\ell)$, $\Omega_\text{bulk}$ and $T$ introduced later in this section (see~Figure~\ref{fig:gamma_l}) and will establish a lower bound of the energy  $\mathcal E_0\big(\psi,\Ab;\cdot)$ in each of the decomposition sets.
		{\bfseries We assume $\ell$ to be the parameter in~\eqref{eq:ell=kp}} which satisfies $\ell\approx \kappa^{-3/4}$.

		\paragraph{\itshape Lower bound in a neighbourhood of the magnetic barrier} 
		We start by introducing the set $\Gamma^*(\ell)$ which covers almost all of the set $\Gamma$. Recall the assumption that $\Gamma$ consists of a finite collection of simple smooth curves that may intersect $\partial\Omega$ transversely. For the simplicity of the exposition, we will focus on the particular case of a single curve intersecting $\partial\Omega$ at two points. The construction below may be adjusted to cover the general case by considering every single component of $\Gamma$ separately.
We may select two constants $\ell_0\in(0,1)$ and $c>2$, and for all $\ell\in(0,\ell_0)$,   a collection of pairwise distinct  points $(x_i)_{i=1}^N\subset\Gamma$  such that, 
\begin{equation}\label{eq:l*}
(x_i)_{i=1}^N\subset \{u\in\Gamma~:~\mathrm{dist}(u,\partial\Omega)>2\ell\}\,,
\end{equation}
\begin{equation}\label{eq:I}
\forall~i\in\{1,...,N-1\},\quad  \dist_{\Gamma}(x_i,x_{i+1})=\ell\,,
		\end{equation}
	\begin{equation}\label{eq:NI}
		|\Gamma|\ell^{-1}-c\leq N \leq |\Gamma|\ell^{-1}\,,
		\end{equation}
		and 
		\begin{equation}\label{eq:gamma*}
		\{x\in\Omega~:~\mathrm{dist}(x,\Gamma)<\ell\,,\ \mathrm{dist}(x,\partial\Omega)>c\ell\}\subset \Gamma^*(\ell):=\left(\bigcup_{i=1}^N	\overline{\mathcal N_{x_i}(\ell)}\right)^{\circ}\,,
		\end{equation}
		where $\mathcal N_{x_i}(\ell)$ is the set introduced in~\eqref{eq:N0}. Note that the family $\left(\mathcal N_{x_i}(\ell)\right)_{1\leq i\leq N}$ consists of pairwise disjoint sets. 
		Consequently,
		\[ \mathcal E_0\big(\psi,\Ab;\Gamma^*(\ell)\big)=\sum_{i=1}^N\mathcal E_0\big(\psi,\Ab;\mathcal N_{x_i}(\ell)\big)\,.\]
		A uniform lower bound for the local energies $\Big(\mathcal E_0\big(\psi,\Ab;\mathcal N_{x_i}(\ell)\big)\Big)_{1\leq i \leq N}$ is already established in Theorem~\ref{thm:E_loc1}. Using the results of the aforementioned theorem, we write
		\begin{align}
		\mathcal E_0\big(\psi,\Ab;\Gamma^*(\ell)\big)
		&\geq N\left(b^{-\frac 12}\kappa\ell \mathfrak e_a(b)-\kappa\ell\mathfrak r(\kappa)\right)\,,\nonumber\\
		&\geq |\Gamma|b^{-\frac 12}\kappa \mathfrak e_a(b)-C\kappa\mathfrak r(\kappa)\,. \label{eq:Gamma_l}
		\end{align}
		The last inequality follows from~\eqref{eq:NI} and the fact that $\mathfrak e_a(b)\leq0$.

	\paragraph{\itshape Lower bound in a neighbourhood of the boundary} 
		Now, we define the two sets $\Omega_1^*(\ell)$ and $\Omega_2^*(\ell)$ which cover almost all of the set $\partial \Omega$. 
		In a similar fashion of the definition of $\Gamma^*(\ell)$, we fix $\ell_0\in(0,1)$ and $c>2$ and  
		we select collections of points $(y_j)_{j=1}^{N_1}\subset\partial \Om_1$ and $(z_k)_{k=1}^{N_2}\subset\partial \Om_2$ such that, 
		\begin{equation}\label{eq:yz}
		(y_j)_{j=1}^{N_1}\subset \{u\in\partial \Om_1~:~\mathrm{dist}(u,\Gamma)>2\ell\}\,,\qquad (z_k)_{k=1}^{N_2}\subset \{u\in\partial \Om_2~:~\mathrm{dist}(u,\Gamma)>2\ell\}\,,
		\end{equation}
		\begin{equation}\label{eq:JK}
		\forall~j\in\{1,...,N_1-1\},\quad  \dist_{\partial \Om_1}(y_j,y_{j+1})=\ell\,,\qquad \forall~k\in\{1,...,N_2-1\},\quad  \dist_{\partial \Om_2}(z_k,z_{k+1})=\ell\,,
		\end{equation}
		\begin{equation}\label{eq:N12}
		|\partial \Om_1|\ell^{-1}-c\leq N_1 \leq |\partial \Om_1|\ell^{-1}\,,\qquad 	|\partial \Om_2|\ell^{-1}-c\leq N_2 \leq |\partial \Om_2|\ell^{-1}\,,
		\end{equation}
		and 
		\begin{align}
		\{x\in\Omega~:~\mathrm{dist}(x,\partial \Om_1)<\ell\,,\ \mathrm{dist}(x,\Gamma)>c\ell\}\subset \Om_1^*(\ell)&:=\left(\bigcup_{j=1}^{N_1}	\overline{\mathcal N_{y_j}(\ell)}\right)^{\circ}\,, \label{eq:Om1*}\\
		\{x\in\Omega~:~\mathrm{dist}(x,\partial \Om_2)<\ell\,,\ \mathrm{dist}(x,\Gamma)>c\ell\}\subset \Om_2^*(\ell)&:= \left(\bigcup_{k=1}^{N_2}\overline{\mathcal N_{z_k}(\ell)}\right)^{\circ}
		\,. \label{eq:Om2*}
		\end{align}
		where $\mathcal N^1_{y_j}(\ell)$ and $\mathcal N^2_{z_k}(\ell)$  were defined in~\eqref{eq:N_j}. 
Hence following similar steps as in~\eqref{eq:Gamma_l}, we use the uniform lower bound in Theorem~\ref{thm:surf} together with the estimates in~\eqref{eq:N12} to get 
		\begin{equation}\label{eq:Om1_l}
		\mathcal E_0\big(\psi,\Ab;\Omega_1^*(\ell)\big)\geq |\partial\Omega_1 \cap \partial\Omega|b^{-\frac 12}\kappa E_\mathrm{surf}(b)-C\kappa\breve{\mathfrak r}(\kappa)\,, 
		\end{equation} 
		and 
		\begin{equation}\label{eq:Om2_l}
		\mathcal E_0\big(\psi,\Ab;\Omega_2^*(\ell)\big)\geq |\partial\Omega_2 \cap \partial\Omega|b^{-\frac 12}|a|^{-\frac12}\kappa E_\mathrm{surf}\big(b|a|\big)-C\kappa \breve{\mathfrak r}(\kappa)\,. 
	\end{equation}

\paragraph{\itshape Lower bound in the bulk} Next, we introduce the set  $\Omega_\text{bulk}$ representing the bulk of the sample. Let
\begin{equation}\label{eq:Om_bulk}
\Omega_\text{bulk}=\left\{x \in \Omega~:~\dist(x,\partial \Omega_1 \cup \partial \Omega_2) > \ell \right\}\,.
\end{equation}
Under our assumptions on $b$ in \eqref{eq:A3} and $\ell$ in \eqref{eq:ell=kp}, the exponential decay in Theorem~\ref{thm:decay} allows us to neglect the energy contribution in the bulk, and to particularly write
\begin{equation}\label{eq:bulk}
	|\mathcal E_0\big(\psi,\Ab;\Omega_\text{bulk}\big)|\leq C \hat{\mathfrak r}(\kappa)\,.
\end{equation}
where $\hat{\mathfrak r}$ is a real-valued function satisfying $\lim_{\kappa\rightarrow +\infty} \hat{\mathfrak r}(\kappa)=0$.

\paragraph{\itshape Lower bound in a neighbourhood of the $T$-zone} We finally introduce the remaining set in the decomposition of $\Omega$, the neighbourhood $T$ of  $\Gamma \cap \partial \Omega$
\[T=\Omega \setminus\overline {\left(\bigcup_{j=1}^2 \Omega^*_j(\ell)\cup \Gamma^*(\ell)\cup \Omega_\text{bulk} \right)}\,.\]
The definition of the sets $\Gamma^*(\ell)$, $\Omega_1^*(\ell)$, $\Omega_2^*(\ell)$ and $\Omega_\text{bulk}$ in~\eqref{eq:gamma*},~\eqref{eq:Om1*},~\eqref{eq:Om2*} and~\eqref{eq:Om_bulk} ensures that $|T|=\mathcal O(\ell^2)$ as $\ell \rightarrow 0$. This small size of $T$ together with $|\psi| \leq 1$ and the first item in Theorem~\ref{thm:priori} imply the following
\begin{align}
|\mathcal E_0\big(\psi,\Ab;T\big)|&\leq C\kappa^2\ell^2\,, \nonumber\\
                                  &\leq C\kappa^\frac 12\,. \label{eq:remain1}
\end{align}
We used the assumption $\ell\approx \kappa^{- 3/4}$ in the last inequality.

Now it is time to gather pieces, and to derive a lower bound for the global GL energy $ \mathcal E_0\big(\psi,\Ab;\Omega\big)$. So, we use the estimates  in~\eqref{eq:Gamma_l},~\eqref{eq:Om1_l},~\eqref{eq:Om2_l},~\eqref{eq:bulk} and~\eqref{eq:remain1} to conclude
\begin{align}
\mathcal E_0\big(\psi,\Ab;\Omega\big)&= \mathcal E_0\big(\psi,\Ab;\Gamma^*(\ell)\big)+\mathcal E_0\big(\psi,\Ab;\Omega_1^*(\ell)\big) +\mathcal E_0\big(\psi,\Ab;\Omega_2^*(\ell)\big) \nonumber\\
                                         	&\quad  +\mathcal E_0\big(\psi,\Ab;\Omega_\text{bulk}\big)+ \mathcal E_0\big(\psi,\Ab;T\big) \nonumber\\
                                           	& \geq |\Gamma|b^{-\frac 12}\kappa \mathfrak e_a(b)+|\partial \Omega_1 \cap \partial \Omega|b^{-\frac 12}\kappa E_\mathrm{surf}(b) \nonumber\\
                                           	&\quad+|\partial \Omega_2 \cap \partial \Omega|b^{-\frac 12}|a|^{-\frac 12}\kappa E_\mathrm{surf}\big(b|a|\big)-C\kappa\,\mathrm R(\kappa)\,.\label{eq:E0_ful_low}
\end{align}		
where 
$\mathrm R(\kappa)= \mathfrak r(\kappa)+ \breve{\mathfrak r}(\kappa)+ \hat{\mathfrak r}(\kappa)+\kappa^{-1/2}=o(1)$, when $\kappa \rightarrow +\infty$.
	
\subsubsection{Energy upper bound}

The upper bound of the ground state energy $\Es$ can be derived by the help of a suitable trial configuration.

In this section, {\bfseries we are still considering the parameter $\ell$ as in~\eqref{eq:ell=kp}}. Let $\Fb$ be the magnetic potential introduced in Lemma~\ref{A_1}. We define the function $w_\Gamma \in H^1(\Omega;\C)\cap H^1_0\big(\Gamma^*(\ell)\big)$ 
\[w_\Gamma(x)=\sum_{i=1}^N \mathbbm{1}_{\mathcal N_{x_i}(\ell)}(x) v_i(x)\,,\]
where $\Gamma^*(\ell)$ is the set defined in~\eqref{eq:gamma*}, $\mathcal N_{x_i}(\ell)$ is the set in~\eqref{eq:N0}, $v_i(x)=e^{i\kappa H\omega_{x_i}(x)}u_i \circ \Phi^{-1}(x)$, $\omega_{x_i}$ is the gauge transformation function defined in Proposition~\ref{prop:Anew2}, $\Phi$ is the coordinate transformation in~\eqref{Frenet}, $u_i$ is defined by $u_i(s,t)=u_0(s-s_i,t)$ where $(s_i,t_i)=\Phi^{-1}(x_i)$, and $u_0$ is the minimizer of the functional $\mathfrak G(\cdot,\mathcal V(\ell))$ defined in~\eqref{eq:G_u}.\\
From the definition of~$v_i$, we derive the following identity (see~\eqref{eq:uv})
\[\mathcal E_0\big(v_i,\Fb;\mathcal N_{x_i}(\ell)\big)=\mathfrak G\left(u_0,\mathcal V(\ell)\right)\,,\]
where $\mathcal E_0$ is the energy in~\eqref{eq:local0}.
We use this identity, the results in Lemma~\ref{lem:upper_G0},~\eqref{eq:NI} and~($\ell\approx \kappa^{- 3/4}$) to derive the following upper bound
\begin{align}
\mathcal E_0\big(w_\Gamma,\Fb;\Omega\big)
&=\sum_{i=1}^N\mathcal E_0\big(v_i,\Fb;\mathcal N_{x_i}(\ell)\big) \nonumber\\
&\leq  |\Gamma|b^{- 1/2}\kappa \mathfrak e_a(b)+C\kappa^\frac{11}{12} \,. \label{eq:gamma_up}
\end{align}
Similarly, for $j \in\{1,2\}$, using the results of Theorem~\ref{thm:surf}, one may define a function $w_j \in H^1(\Omega;\C)\cap H^1_0\big(\Omega^*_j(\ell)\big)$  satisfying
\begin{align}
\mathcal E_0\big(w_1,\Fb;\Omega^*_1(\ell))&\leq |\partial \Omega_1 \cap \partial \Omega|b^{-\frac 12}\kappa E_\mathrm{surf}(b)+C\kappa\mathfrak r_1(\kappa)\,,\nonumber\\
\mathcal E_0\big(w_2,\Fb;\Omega^*_2(\ell))&\leq |\partial \Omega_2 \cap \partial \Omega|b^{-\frac 12}|a|^{-\frac 12}\kappa E_\mathrm{surf}\big(b|a|\big)+C\kappa\mathfrak r_2(\kappa)\,,  \label{eq:Omi_up}
\end{align}
where  $\Omega^*_j(\ell)$ is defined in~\eqref{eq:Om1*} and~\eqref{eq:Om2*}, and $\mathfrak r_j$ is a real-valued function tending to zero when $\kappa \rightarrow +\infty$.
Now, we define the trial function
\[w(x)= \mathbbm{1}_{\Gamma^*(\ell)}(x) w_\Gamma(x)+\mathbbm{1}_{\Omega^*_1(\ell)}(x)w_1(x)+\mathbbm{1}_{\Omega^*_2(\ell)}(x)w_2(x)\,,\]
Gathering results in~\eqref{eq:gamma_up} and~\eqref{eq:Omi_up}, we get
\begin{align*}
\mathcal E\big(w,\Fb;\Omega\big)&=\mathcal E_0\big(w,\Fb;\Omega\big)\\
                                     &=\mathcal E_0\big(w_\Gamma,\Fb;\Gamma^*(\ell)\big)+\mathcal E_0\big(w_1,\Fb;\Omega^*_1(\ell)\big)+\mathcal E_0\big(w_2,\Fb;\Omega^*_2(\ell)\big)\\
                                     &\leq  |\Gamma|b^{-\frac 12}\kappa \mathfrak e_a(b)+|\partial \Omega_1 \cap \partial \Omega|b^{-\frac 12}\kappa E_\mathrm{surf}(b)\\
                                     &\quad +|\partial \Omega_2 \cap \partial \Omega|b^{-\frac 12}|a|^{-\frac 12}\kappa E_\mathrm{surf}\big(b|a|\big)+C\kappa \tilde{\mathsf R}(\kappa)\,,
\end{align*}
where 
$\tilde {\mathsf R}(\kappa)= \mathfrak r_1(\kappa)+ \mathfrak r_2(\kappa)+\kappa^{-1/12}=o(1)$
, when $\kappa \rightarrow +\infty$.
We finally derive the upper bound in~\eqref{eq:Eg} from the fact that $\Es \leq \mathcal E(w,\Fb;\Omega)$.
 \subsubsection{$L^4$-norm asymptotics}
In light of~\eqref{eq:Euler}, any minimizing configuration $(\psi,\Ab)$ of the functional in~\eqref{eq:GL} satisfies
\begin{equation}\label{eq:Euler1}
\big(\kn\big)^2\psi=\kp^2(|\psi|^2-1)\psi \ \mathrm{in}\ \Om\,.
\end{equation}
We multiply both sides of~\eqref{eq:Euler1} by $\overline \psi$ and integrate by parts, using the boundary condition in~\eqref{eq:Euler}. We get
\begin{equation}\label{eq:Euler2}
\mathcal E_0(\psi,\Ab;\Omega)=-\frac 12 \kappa^2\int_{\Omega}|\psi|^4\,dx\,.
\end{equation}
Having 
\[\mathcal E_0(\psi,\Ab;\Omega)\leq \Es\,,\]
the lower bound in~\eqref{eq:psi4} follows from~\eqref{eq:Euler2} and~\eqref{eq:Eg}. The upper bound is derived from~\eqref{eq:Euler2} and~\eqref{eq:E0_ful_low}. 

\subsection{Proof of Theorem~\ref{thm:E_loc}} The proof of Theorem~\ref{thm:E_loc} follows from the local estimates stated in Theorems~\ref{thm:E_loc1} and~\ref{thm:surf} together with the decay result in Theorem~\ref{thm:decay}.
 
\appendix 
\section{Some Spectral Properties of Fiber Operators}\label{sec:}
Let $a \in [-1,1)\setminus\{0\}$ and $\xi \in \R$.
Recall the operator $\mathfrak h_a[\xi]$ introduced in~\eqref{eq:potential} and its associated quadratic form $q_a[\xi]$ defined in~\eqref{eq:quad}. The embedding of the domain of $q_a[\xi]$ is compact in $L^2(\R)$, hence the spectrum of $\mathfrak h_a[\xi]$  is an increasing sequence of eigenvalues converging to $+\infty$. 
The first eigenvalue $\mu_a(\xi)$ of this operator was defined in~\eqref{mu_a_1} by the min-max principle.

The result in the following proposition may be derived similarly as done in~\cite[Section~3.2.1]{fournais2010spectral}:
\begin{proposition}\label{mu_simple}
The first eigenvalue $\mu_a(\xi)$ of $\mathfrak h_a[\xi]$ is simple. Furthermore, there exists a positive eigenfunction $f_{a,\xi}$ normalized with respect to the norm of $\|\cdot\|_{L^2(\R)}$. $f_{a,\xi}$ is the unique function satisfying such properties.
\end{proposition}
\textbf{Notation.} For the positive eigenfunction $f_{a,\xi}$ defined in the previous proposition, we associate the de\,Gennes parameter:
\begin{equation}\label{eq:gamma_a}
\gamma_a(\xi)= \left (\frac{f'_{a,\xi}}{f_{a,\xi}}  \right)(0_+)=\left ( \frac{f'_{a,\xi}}{f_{a,\xi}}  \right)(0_{-})
\end{equation}

The next theorem calls some regularity properties: 
\begin{theorem}\label{thm:regularity}
The functions $\xi\mapsto \mu_a(\xi)$, $\xi\mapsto f_{a,\xi}$, and $\xi\mapsto \gamma_a(\xi)$ are in $C^\infty$\,.
\end{theorem}
\begin{proof}
	Note that the domain of the operator  $\mathfrak h_a[\xi]$ can be expressed, independently of $\xi$, by
	\[\dom\big(\mathfrak h_a[\xi]\big)=\left\{u\in B^1(\R)~:~\Big(-\frac{d^2}{dt^2}+\mathsf s^2t^2\Big)u \in L^2(\R)\,,~u'(0_+)=u'(0_-)\right\}\,,\]
	where $\mathsf s(t)=\mathbbm{1}_{\R_+}(t)+a\mathbbm{1}_{\R_-}(t)$. 
	Also, given a function $u\in \dom\big(\mathfrak h_a[\xi]\big)$, the function
	\[\R\ni\xi\mapsto \mathfrak h_a[\xi](u)\]
	is holomorphic.
	Hence, $\Big( \mathfrak h_a[\xi]\Big)_{\xi}$ is a self-adjoint holomorphic family of type (A) with compact resolvent defined for $\xi \in \R$. The regularity results follow then from the perturbation theory of Kato (see~\cite[Theorem C.2.2]{fournais2010spectral}).
\end{proof}

Next, we establish some bounds of $\mu_a(\xi)$ involving the first Neumann and Dirichlet eigenvalues $\mu^N(\cdot)$ and $\mu^D(\cdot)$ respectively defined in~\eqref{eq:mu_n} and~\eqref{eq:mu_D}. These bounds are used to set down important limits in Proposition~\ref{prop:mu_a_lim}.
\begin{lemma}\label{lem:mu_a_bound}
Let $a \in [-1,1)\setminus\{0\}$. It holds
\begin{itemize}
\item If $a \in (0,1)$, then 
\[\min \left(\mu^N(-\xi),\ a\mu^N\Big(\frac \xi{\sqrt{a}}\Big)\right)\leq \mu_a(\xi) \leq \min \left(\mu^D(-\xi),\ a\mu^D\Big(\frac \xi{\sqrt{a}}\Big)\right)\,.\]
\item If $a \in [-1,0)$, then 
\[\min \left(\mu^N(-\xi),\ |a|\mu^N\Big(-\frac \xi{\sqrt{|a|}}\Big)\right)\leq \mu_a(\xi) \leq \min \left(\mu^D(-\xi),\ |a|\mu^D\Big(-\frac \xi{\sqrt{|a|}}\Big)\right)\,.\]
\end{itemize}
\end{lemma}
\begin{proof} The proof is similar to the one done in~\cite[Lemma 3.2.2]{Kachmar}. \end{proof}
\begin{proposition}\label{prop:mu_a_lim}
Let $a \in [-1,1)\setminus\{0\}$. We have
\begin{itemize}
\item For  $a \in (0,1)$, 
\[
\lim_{\xi\rightarrow-\infty}\mu_a(\xi)= 1 \quad \mathrm{and}\quad  \lim_{\xi\rightarrow+\infty}\mu_a(\xi)=a\,.
\]
\item For $a \in [-1,0)$, 
\[
 \lim_{\xi\rightarrow-\infty}\mu_a(\xi)= |a|  \quad \mathrm{and}\quad \lim_{\xi\rightarrow+\infty}\mu_a(\xi)=+\infty\,.
\]
\end{itemize}
\end{proposition}
\begin{proof}

It is sufficient to apply Lemma~\ref{lem:mu_a_bound}, and Theorems~\ref{thm:dirichlet} and~\ref{thm:sturm}.
\end{proof}
We conclude this section by determining the derivative of $\mu_a(\xi)$ with respect to $\xi$. Note that we are interested in establishing the result only in the case where $a \in [-1,0)$\,.
\begin{proposition}\label{prop:mu_deriv}
For any $a\in[-1,0)$ and $\xi \in \R$ we have 
\begin{equation}\label{eq:mu_deriv}
\partial_\xi \mu_a(\xi)=\left(1-\frac1a\right)\left (\gamma^2_a(\xi)+\mu_a(\xi)-\xi^2 \right)|f_{a,\xi}(0)|^2\,.
\end{equation}
where $f_{a,\xi}$ is the eigenfunction in Proposition~\ref{mu_simple}, and $\gamma_a(\xi)$ as in~\eqref{eq:gamma_a}.
\end{proposition} 
\begin{proof}\textit{(Feynman-Hellmann)}.
	For simplicity, we write $\mu$, $f$ and $\mathfrak h$ respectively for $\mu_a(\xi)$, $f_{a,\xi}$, and $\mathfrak h_a[\xi]$\,. Differentiating with respect to $\xi$ and integrating by parts in
	\begin{equation}\label{eq:h}
(\mathfrak h-\mu)f=0.
	\end{equation}
we get 
\[\langle(\partial_\xi \mathfrak h-\partial_\xi \mu)f,f \rangle+ \langle (\mathfrak h-\mu)\partial_\xi f,f \rangle=0\,.  \]	
Hence using 
\[\langle (\mathfrak h-\mu)\partial_\xi f,f \rangle=\langle \partial_\xi f,(\mathfrak h-\mu)f \rangle=0\,,\]
and recalling that $f$ is normalized, we obtain
\begin{align}
\partial \mu_\xi&=\langle \partial_\xi\mathfrak hf,f \rangle \nonumber\\
                &=2\int_{-\infty}^0(\xi+at)f^2\,dt+ 2\int_0^{+\infty}(\xi+t)f^2\,dt\,. \label{eq:mu_xi}
\end{align}
Integrating by parts the right hand side of~\eqref{eq:mu_xi}, and using~\eqref{eq:h} together with the Neumann condition $\Big(f'(0_-)=f'(0_+)\Big)$ establish the result.
\end{proof}

\newcommand{\etalchar}[1]{$^{#1}$}

\end{document}